\documentclass[10pt]{amsart}

\usepackage{graphicx}
\usepackage{mathptmx}
\DeclareSymbolFont{cmlargesymbols}{OMX}{cmex}{m}{n}
\DeclareMathSymbol{\mycoprod}{\mathop}{cmlargesymbols}{"60}
\let\coprod\mycoprod
\usepackage{amsfonts,amssymb,amsmath,amscd}
\usepackage[all]{xy}

\numberwithin{equation}{section}

\newtheorem{theorem}{Theorem}[section]
\newtheorem{lemma}[theorem]{Lemma}
\newtheorem{proposition}[theorem]{Proposition}
\newtheorem{corollary}[theorem]{Corollary}

\theoremstyle{definition}
\newtheorem{definition}[theorem]{Definition}

\theoremstyle{remark}
\newtheorem{Claim}[theorem]{Claim}

\newtheorem{algorithm}{Algorithm}[section]
\newtheorem{notation}{Notation}
\newtheorem{remark}{Remark}[section]
\newtheorem{example}{Example}[section]

\newtheorem{convention}{Convention}

\numberwithin{equation}{section}

\newcommand{\Tor}{{\rm Tor}}
\newcommand{\Hom}{{\rm Hom}}

\newcommand{\Ker}{{\rm Ker}}
\newcommand{\rank}{{\rm rank}}
\newcommand{\lcm}{{\rm lcm}}

\newcommand{\Span}{{\rm Span}}
\newcommand{\Spec}{{\rm Spec}}

\newcommand{\EExt}{{\mathbb E}{\rm xt}}
\newcommand{\RHom}{{\mathcal H}om}
\newcommand{\val}{\upsilon}

\newcommand{\CCC}{{\mathbb C}}
\newcommand{\FF}{{\mathbb F}}
\newcommand{\EE}{{\mathbb E}}
\newcommand{\RR}{{\mathbb R}}
\newcommand{\ZZ}{{\mathbb Z}}
\newcommand{\PP}{{\mathbb P}}
\newcommand{\NN}{{\mathbb N}}
\newcommand{\GG}{{\mathbb G}}

\newcommand{\DD}{{\mathbb D}}
\newcommand{\QQ}{{\mathbb Q}}
\newcommand{\KK}{{\mathbb K}}
\newcommand{\LL}{{\mathbb L}}

\newcommand{\AAA}{{\mathbb A}}

\newcommand{\kk}{{\Bbbk}}
\newcommand{\CL}{{\mathcal L}}

\newcommand{\CO}{{\mathcal O}}

\newcommand{\CM}{{\mathcal M}}

\newcommand{\CF}{{\mathcal F}}

\newcommand{\CE}{{\mathcal E}}
\newcommand{\CX}{{\mathcal X}}
\newcommand{\CB}{{\mathcal B}}

\newcommand{\CC}{{\mathcal C}}

\newcommand{\codim}{{\rm codim}}

\begin{document}

\title{Tropical geometry and correspondence theorems via toric stacks}
\author{Ilya Tyomkin}
\address{Department of Mathematics, Ben-Gurion University of the Negev, P. O. Box 653, Be'er Sheva, 84105, Israel}
\email{tyomkin@cs.bgu.ac.il}
\thanks{The research leading to these results has received funding from the European Union Seventh Framework Programme (FP7/2007-2013) under grant agreement 248826.}

\begin{abstract}
In this paper we generalize correspondence theorems of Mikhalkin and Nishinou-Siebert providing a correspondence between algebraic and parameterized tropical curves. We also give a description of a canonical tropicalization procedure for algebraic curves motivated by Berkovich's construction of skeletons of analytic curves. Under certain assumptions, we construct a one-to-one correspondence between algebraic curves satisfying toric constraints and certain combinatorially defined objects, called ``stacky tropical reductions", that can be enumerated in terms of tropical curves satisfying linear constraints. Similarly, we construct a one-to-one correspondence between elliptic curves with fixed $j$-invariant satisfying toric constraints and ``stacky tropical reductions" that can be enumerated in terms of tropical elliptic curves with fixed tropical $j$-invariant satisfying linear constraints. Our theorems generalize previously published correspondence theorems in tropical geometry, and our proofs are algebra-geometric. In particular, the theorems hold in large positive characteristic.
\end{abstract}
\maketitle
\newpage
\section{Introduction.}

Recently, tropical varieties appeared in various fields of study, such as string theory, mirror symmetry, and enumerative geometry. Roughly speaking, tropical variety is an integral piece-wise linear polyhedral complex equipped with an integral affine structure. One can also think about tropical varieties as algebraic varieties over the $(\max, +)$ semi-ring. Till now several applications of tropical geometry to algebraic geometry have been found.

In 2005, Mikhalkin \cite{M05} discovered a ``tropical" formula for enumeration of curves of genus $g$ in a linear system $\CL$ on a toric surface $X$ passing through an appropriate number of points in general position. The main ingredient in the proof was a ``correspondence theorem" that provided a {\em one-to-one correspondence} between certain algebraic and parameterized complex tropical curves. Mikhalkin gave two descriptions of parameterized tropical curves: a combinatorial description as weighted balanced graphs in $\RR^n$, and an algebraic description as algebraic curves over $(\max, +)$ semi-ring. He showed that any algebraic curve on a toric surface defines a parameterized (complex) tropical curve in $\RR^2$. To assign a parameterized tropical curve to an algebraic curve Mikhalkin analyzed the Hausdorff limits of certain logarithmic degenerations of algebraic curves in the logarithmic image $\RR^2$ of the complex torus $(\CCC^*)^2$, and showed that these limits are piece-wise linear graphs in $\RR^2$ that can be equipped with weights turning them into parameterized tropical curves. Similarly, he associated complex tropical curves to algebraic curves. Then, by using analytic and symplectic techniques, Mikhalkin proved that under certain transversality assumptions there exists unique algebraic curve defining a given complex tropical curve. Finally, he described a couple of combinatorial formulae that count the number of complex tropical, hence also algebraic, curves in terms of parameterized tropical curves and lattice paths. In his ICM paper, Mikhalkin presents the correspondence theorem  \cite[Theorem 2]{M06} as an application of a result about realization of regular parameterized tropical curves by complex algebraic curves \cite[Theorem 1]{M06}, which holds true in arbitrary dimension.

An algebra-geometric proof of Mikhalkin's theorem based on Viro's patchworking method was proposed by Shustin \cite{Sh05} in 2005. He showed that over a non-Archimedean field, any algebraic curve on a toric surface defines a degeneration of the surface corresponding to a convex subdivision of the Newton polygon, and the subdivision is combinatorially dual to the corresponding parameterized tropical curve. By taking the closure of the curve in the family of toric surfaces, Shustin obtained a degenerated algebraic curve sitting in a degenerated toric surface, and called such a pair tropicalization. Then he introduced refined tropicalizations, and used patchworking techniques to prove that under certain conditions one can reconstruct uniquely the algebraic curve from its refined tropicalization.

In 2006, Nishinou and Siebert \cite{NS06} proved another correspondence theorem for rational curves in higher dimensional toric varieties, using the techniques of log-geometry. Similarly to Shustin, they constructed a toric degeneration of the ambient toric variety controlled by the parameterized tropical curve. Then they looked at the corresponding degeneration of the algebraic curve and equipped it with the natural log-structure coming from the degeneration of the toric variety. Finally, they proved that under certain conditions the algebraic curve can be reconstructed uniquely from its degeneration as log-variety.

In \cite{M05,NS06,Sh05}, the parameterized tropical curve $\Gamma$ corresponding to an algebraic curve $C$ was constructed in terms of the morphism from $C$ to the toric variety. As a result, the underlying tropical curve, i.e., the metric graph, depended on the toric variety.

The first goal of this paper is to describe a canonical procedure associating a tropical curve $\Gamma$ to an algebraic curve with marked points $(C,D)$ over the separable closure $\overline{\FF}$ of the field of fractions $\FF$ of a discrete valuation ring $R$. In Subsection~\ref{subsec:trcurvesforcurvewithmp}, we define the underlying graph of $\Gamma$ to be {\it the dual graph of the stable reduction of the pair $(C,D)$}, and we define the metric on $\Gamma$ in a natural way in terms of the singularities of the total space of the stable model. If, in addition, we are given a morphism $f\colon C\setminus D\to (\overline{\FF}^*)^n$, then in Subsection~\ref{subsec:paramtrcurforacurwithamap}, we construct a natural parameterized tropical curve $h\colon\Gamma\to \RR^n$. Our construction is canonical, and the parameterized tropical curves constructed in \cite{M05,NS06,Sh05} are obtained from $\Gamma$ above by contraction of maximal connected subgraphs contracted by $h$. We note here that similar approach to tropicalization of algebraic curves was used by Baker \cite{B08}. Note also that there is an alternative description of $\Gamma$. Namely, given a curve with marked points $(C,D)$, one considers the corresponding Berkovich analytification $(B,D)$. If $(C,D)$ is stable, then $B$ contains a distinguished skeleton, which is a metric graph; and it is possible to show that this graph is naturally isometric to $\Gamma$. In fact, our definition of $\Gamma$ was motivated by the Berkovich's construction of skeletons of analytic curves. We will not use the language of Berkovich spaces in this paper, but an interested reader may look at \cite{Ber90} for an introduction to Berkovich spaces, analytifications, and skeletons of analytic curves.

The second goal of this paper is to generalize the correspondence theorems of \cite{M05,NS06,Sh05}.
Our Theorem~\ref{thm:corrthm} is a generalization of the theorems of Mikhalkin and Nishinou-Siebert for curves satisfying toric constraints, e.g., passing through given points in general position, and Theorem~\ref{thm:yacs} gives an algebraic-tropical correspondence for elliptic curves satisfying toric constraints and having given $j$-invariant.

Our approach is as follows: Let $T_N$ be an algebraic torus, and $O_1,\dotsc, O_k$ be general orbits of some subtori of $T$. Let $(C,D)$ be an algebraic curve with marked points, and $f\colon C\setminus D\to T_N$ be a morphism, such that $f$ extends to the first $k$ marked points and maps them to the orbits $O_1,\dotsc, O_k$. Set $\Gamma$ to be the parameterized tropical curve associated to $(C,D,f)$. As a first step, we construct a minimal partial compactification $X$ of $T_N$ such that $f$ extends to $C$. Then, we construct a canonical integral model $f_{R_{\overline{\FF}}}\colon C_{R_{\overline{\FF}}}\to X_{R_{\overline{\FF}}}$ of $f\colon C\to X$ over the integers ${R_{\overline{\FF}}}\subset \overline{\FF}$. One must think about the integral model as a degeneration similar to the degenerations in \cite{NS06,Sh05}. We also construct an integral model $Y_{R_{\overline{\FF}}}$ of the constraint $Y=\cup O_i$. Furthermore, we introduce a natural structure of a Deligne-Mumford stack $\CX_{R_{\overline{\FF}}}$ and $\CC_{R_{\overline{\FF}}}$ on $X_{R_{\overline{\FF}}}$ and $C_{R_{\overline{\FF}}}$. We note here that the partial compactification $X$, the integral model $f_{R_{\overline{\FF}}}\colon C_{R_{\overline{\FF}}}\to X_{R_{\overline{\FF}}}$, and the stacky structures are determined by the parameterized tropical curve. We call the reduction of $(\CC_{R_{\overline{\FF}}},D_{R_{\overline{\FF}}})$ together with the morphism to $\CX_{R_{\overline{\FF}}}$ the stacky tropical reduction of $(C,D,f)$. Finally, we show that under certain assumptions, $(C,D,f)$ can be reconstructed uniquely from its stacky tropical reduction. Our approach to Theorem~\ref{thm:yacs} is similar.

Several remarks are in place here. First, in order to construct the natural stacky structures, we introduced singular toric Deligne-Mumford stacks generalizing toric stacks of Borisov, Chen, and Smith \cite{BCS05}. Second, one of the assumptions of the correspondence theorems is that $(C,D)$ is a simple Mumford curve, i.e., its stable reduction has rational components with precisely three special points on each component. Third, the number of stacky tropical reductions can be described combinatorially in terms of the corresponding parameterized tropical curve $\Gamma$; see Propositions \ref{prop:ParamOfTrLimitsConstr} and \ref{prop:ParamOfConstrStakyTrLimits}. Thus, under the assumptions of the correspondence theorems, one obtains a {\em one-to-one correspondence between the simple Mumford algebraic curves satisfying certain constraints and stacky tropical reductions satisfying the degenerations of those constraints}, which, in turn, can be enumerated in terms of the corresponding parameterized tropical curves combinatorially. Finally, note that our approach is algebra-geometric and works in large positive characteristics. The case of small characteristics involves technical difficulties since new phenomena occur, and it will be studied in a separate paper. We note here that in \cite{M05,NS06,Sh05} the authors assume the ground field to be of characteristic zero. Plainly, the approach of Mikhalkin does not work in positive characteristic. Similarly, Shustin's approach uses the characteristic assumption a lot. However, to the best of our understanding, Nishinou-Siebert's approach must work in large positive characteristic though this is not claimed in \cite{NS06}.

We wish to conclude the introduction by saying that the deformation-theoretic pattern developed in this paper can be used in other problems as well. For instance, one can prove that any regular tropical curve is representable (cf. Remark~\ref{rem:representabilityOfregularcurves}), which extends Mikhalkin's \cite[Theorem 1]{M06} to the case of large positive characteristic. Moreover, it is possible to obtain representability results for superabundant tropical curves, but this will be discussed in a separate paper.

Recently, Nishinou posted a preprint \cite{N09}, where he extends the logarithmic techniques of \cite{NS06}. He proves a version of the correspondence theorem over $\CCC$ for regular tropical curves, and also for superabundant genus-one curves. There is an overlap between the results presented in our paper and in the paper of Nishinou, but the results were obtained independently and the techniques are different.

\subsection*{Acknowledgements}
This research was initiated while I was visiting Max-Planck-Institut f\"ur Mathematik at Bonn in Summer 2007, and an essential part of it was done while being a Moore Instructor at MIT. I am very grateful to these institutions for their hospitality.
Many thanks are due to Dan Abramovich for a series of enlightening conversations we had, and for sharing his ideas with me. I would also like to thank V. Berkovich, K. Kremnizer, E.Shustin, and M. Temkin for helpful discussions.

\subsection{Conventions and notation}\label{subsec:convnot}
\begin{description}
\item[Non-Archimedean base field]
Throughout this paper, $\kk$ denotes an algebraic\-ally closed field, $R$ denotes a complete discrete valuation ring with residue field $\kk$ and field of fractions $\FF$, $\overline{\FF}$ denotes the separable closure of $\FF$, and $\val$ denotes the valuation on $\overline{\FF}$ normalized such that $\val(\FF^*)=\ZZ$. For an intermediate extension $\FF\subseteq\LL\subseteq \overline{\FF}$, $R_\LL$ denotes the ring of integers in $\LL$. Note that if $[\LL:\FF]<\infty$ then $R_\LL$ is a complete discrete valuation ring since $R$ is so. For two finite intermediate extensions $\FF\subseteq\KK\subseteq\LL\subseteq \overline{\FF}$, the relative ramification index $[\val(\LL^*):\val(\KK^*)]$ is denoted by $e_{\LL/\KK}$, and if $\KK=\FF$ then it is denoted simply by $e_{\LL}$. For a finite intermediate extension $\FF\subseteq\LL\subseteq \overline{\FF}$, $t_\LL$ denotes a uniformizer in $R_\LL$. Note that if $char (\kk)=0$ then $R\simeq \kk[[t]]$ and $R_\LL\simeq \kk[[t^{1/e_\LL}]]$. Hence, we may assume that $t_\LL=t^{1/e_\LL}$ in this case.

\item[Latices and toric varieties]
Throughout this paper, $M$ denotes a lattice of finite rank, $N=\Hom_\ZZ(M,\ZZ)$ denotes the dual lattice. For an abelian group $G$, we denote $M_G:=M\otimes_\ZZ G$ and $N_G:=N\otimes_\ZZ G$. All toric varieties are considered over $\ZZ$. In particular, $T_N$ denotes the torus $\Spec \ZZ[M]$, and $T_{N,\LL}$ denotes the torus $\Spec \LL[M]$. The monomials in $\ZZ[M]$ and $\LL[M]$ are denoted by $x^m$. If $\Sigma$ is a fan in $N_\RR$, and $\sigma,\tau\in\Sigma$, then $\Sigma^k$ denotes the set of cones of dimension $k$ in $\Sigma$, $X_\sigma$ denotes the toric variety $\Spec\ZZ[\check{\sigma}\cap M]$, and $X_{\sigma\tau}$ denotes the toric variety $X_\sigma\cap X_\tau=X_{\sigma\cap\tau}$.

\item[Graphs]
The graphs we consider in this paper are finite connected graphs. They are allowed to have loops and multiple edges. For a given graph $\Gamma$, the sets of vertices and edges of $\Gamma$ are denoted by $V(\Gamma)$ and $E(\Gamma)$. For $v\in V(\Gamma)$, $val(v)$ denotes the valency of $v$. $V_k(\Gamma)$ denotes the set of vertices of valency $k$. If $v,v'\in V(\Gamma)$ then $E_{vv'}(\Gamma)$ denotes the set of edges connecting $v$ and $v'$. Most graphs in the paper are topological graphs, i.e. CW complexes of dimension one consisting of: (i) a 0-dimensional cell for each vertex, and (ii) a 1-dimensional cell for each edge glued to the 0-dimensional cells corresponding to the boundary vertices of the edge.

\item[Curves]
Throughout this paper, $(C,D)$ denotes a smooth complete curve with marked points $D=\{q_1,\dotsc, q_{|D|}\}$ over the field $\overline{\FF}$, and $(C_{R_\LL}, D_{R_\LL})$ denotes a nodal model of $(C,D)$, i.e., $C_{R_\LL}\to\Spec R_\LL$ is a proper curve, where $\LL/\FF$ is a {\em finite separable} extension, $D_{R_\LL}$ is a finite ordered set of $R_\LL$-points in $C_{R_\LL}$, the total space of $C_{R_\LL}$ is normal, the reduction $(C_{R_\LL}, D_{R_\LL})\times_{\Spec R_\LL}\Spec\kk$ is a reduced nodal curve with marked points, and we are given an isomorphism $(C_{R_\LL}, D_{R_\LL})\times_{\Spec R_\LL}\Spec\overline{\FF}\simeq (C,D)$.
\end{description}

\subsection{Plan of the paper.}
In the appendix, we summarize the basic facts about the nodal and the semi-stable models of algebraic curves. As our definition and treatment of (parameterized) tropical curves is motivated by these facts, we suggest to start reading the paper by looking at the appendix.

Several different definitions of (parameterized) tropical curves can be found in the literature. In Section~\ref{sec:trcur}, we give a version of the definitions that are most convenient for our approach. In particular, since sometimes we work with nodal models of algebraic curves with marked points, we must allow tropical curves with vertices of valency less than three, and unbounded ends of zero slope. Thus, Section~\ref{sec:trcur} is devoted to the definitions of tropical and parameterized tropical curves, to the discussion of basic theory of (parameterized) tropical curves, toric and elliptic constraints, and deformation theory of parameterized tropical curves. In this section we also introduce most of the notation we use throughout the paper. Subsections~\ref{subsec:trcurvesforcurvewithmp} and \ref{subsec:paramtrcurforacurwithamap} give the motivating examples for our definitions, and describe the canonical tropicalization of algebraic curves. For the convenience of the reader, after each definition we give a remark comparing it to other definitions in the literature, and explaining the differences.

In Section~\ref{sec:trdegtrlim}, we construct the integral models of $f\colon C\to X$ and of the constraints and define the notion of tropical reduction. At the end of the section we explain the reason for introducing the stacky structures.

In Section~\ref{sec:toricstacksStackLimits}, we introduce singular Deligne-Mumford toric stacks, and use them to construct the natural stacky structure on the tropical degenerations and reductions.

In Section~\ref{sec:deftheory}, we discuss the deformation theory needed for the correspondence theorems.

Finally, we formulate and prove the correspondence theorems in Section~\ref{sec:CorrThm}.

\section{Tropical curves and parameterized tropical curves.}\label{sec:trcur}
\subsection{Tropical curves.}\label{subsec:deftrcur}

\begin{definition}\label{def:trcur}
$ $

\begin{enumerate}
    \item A {\it tropical curve} is a topological graph $\Gamma$ equipped with a complete, possibly degenerate, inner metric, and with the following structure (s1),(s2), and satisfying the following properties (p1),(p2),(p3):
        \begin{enumerate}
            \item[(s1)] the vertices of $\Gamma$ are subdivided into two groups: {\it finite vertices} and {\it infinite vertices},
            \item[(s2)] the set of infinite vertices is equipped with a total order, and is denoted by $V^\infty(\Gamma)$; the set of finite vertices is just a set, and is denoted by $V^f(\Gamma)$;
            \item[(p1)] $\Gamma$ has finitely many vertices and edges;
            \item[(p2)] any infinite vertex has valency one, and is connected to a finite vertex by an edge, called {\it unbounded edge}. Other edges are called {\it bounded edges}. The set of bounded edges is denoted by $E^b(\Gamma)$, and the set of unbounded edges is denoted by $E^\infty(\Gamma)$;
           \item[(p3)] any bounded edge $e$ is isometric to a closed interval $\left[0,|e|\right]$, where $|e|\in\RR$ denotes the length of $e$, and any unbounded edge $e$ is isometric to $[0,\infty]$, where the isometry maps the infinite vertex to $\infty$. Hence $|e|=\infty$ for any unbounded edge $e$, and the restriction of the metric to $\Gamma\setminus V^\infty(\Gamma)$ is non-degenerate.
         \end{enumerate}
      \item A {\it $\QQ$-tropical curve} is a tropical curve such that $|e|\in\QQ$ for any $e\in E^b(\Gamma)$.
      \item A tropical curve is called {\it irreducible} if the underlying graph $\Gamma$ is connected.
      \item The {\it genus} of a tropical curve $\Gamma$ is defined by $$g(\Gamma):= 1-\chi(\Gamma)=1-|V(\Gamma)|+|E(\Gamma)|.$$ If $\Gamma$ is irreducible then $g(\Gamma)=b_1(\Gamma)$.
     \item A tropical curve is called {\it stable} if all its finite vertices have valency at least three.
     \item An {\it isomorphism} of tropical curves is an isomorphism of metric graphs.
\end{enumerate}
\end{definition}
\begin{remark}
Among the definitions of tropical curves existing in the literature, \cite[Definition 1.1]{GK08} of Gathmann and Kerber is the closest to  Definition~\ref{def:trcur}. They also allow vertices of valency less than three, and add vertices at infinity. The only difference is that Gathmann and Kerber don't order the infinite vertices and call the unbounded edges unbounded ends.

We shall mention that from the point of view of our motivating example, as presented in Subsection~\ref{subsec:trcurvesforcurvewithmp} below, it would be more natural to consider tropical curves equipped with a function $g\colon V^f(\Gamma)\to \ZZ_+$ associating to each vertex a non-negative integer, called {\em the genus of the vertex}; and to modify the definition of the genus of a tropical curve by setting $g(\Gamma):=1-\chi(\Gamma)+\sum_{v\in V^f(\Gamma)} g(v)$. The notion of the stabilization defined below should then also be modified. However, since for the purpose of this paper the more standard definitions are sufficient, we decided not to change the standard definitions too much.
\end{remark}
\begin{remark}
Note that the isomorphism class of a tropical curve is completely determined by the underlying graph with the extra structure (s1)-(s2) on the set of vertices and the positive lengths of the bounded edges. Vice versa, given such data, one can easily construct a tropical curve in the corresponding class. Note also, that given an isomorphism $\phi$ of the underlying graphs of two tropical curves $\Gamma$ and $\Gamma'$, there exists at most one isomorphism of the tropical curves inducing $\phi$. In particular, there exist no non-trivial automorphism of a tropical curve inducing the identity maps on the sets of vertices and edges. However, in general, there may exist several isomorphisms between $\Gamma$ and $\Gamma'$. Thus, we will not identify tropical curves with their isomorphism classes.
\end{remark}

\begin{algorithm}\label{alg:trcurves}
Given a tropical curve $\Gamma$ one can construct a new tropical curve $\Gamma'$ using the following three steps (compare to Algorithm~\ref{alg:models} in the opposite order):
\begin{enumerate}
\item subdivide each bounded edge $e$ into finitely many pieces, i.e., mark $k_e\ge 0$ distinct points on the edge $e$, add them to the set of finite vertices, and replace the edge $e$ with the subintervals defined by the points and equipped with the induced metric;
\item in a similar way, subdivide each unbounded edge into finitely many pieces;
\item attach metric trees to certain finite vertices $v\in V^f(\Gamma)$, i.e., pick a metric tree $T_v$, such that all edges but maybe some of the leaves of $T_v$ have finite length, and identify the root of $T_v$ with $v$.
\end{enumerate}
\end{algorithm}
\begin{Claim}
Let $\Gamma$ be an irreducible tropical curve satisfying
\begin{equation}\label{stcond}
g(\Gamma)+\frac{|V^\infty(\Gamma)|+1}{2}\ge 2.
\end{equation}
Then there exists a unique stable tropical curve $\Gamma^{\rm st}$, such that $V^\infty(\Gamma)=V^\infty(\Gamma^{\rm st})$ and $\Gamma$ can be obtained from $\Gamma^{\rm st}$ by the three steps of Algorithm \ref{alg:trcurves}. In particular, if $\Gamma$ is stable then $\Gamma^{\rm st}=\Gamma$.
\end{Claim}
\begin{proof} To construct $\Gamma^{\rm st}$ we must first, remove from $\Gamma$ the maximal forest of trees all of whose leaves are finite vertices. Then we remove the two-valent vertices and ``glue'' the corresponding pairs of edges. As a result, we obtain an irreducible tropical curve $\Gamma^{\rm st}$, and there are three possibilities: (1) $\Gamma^{\rm st}$ has neither finite leaves nor two-valent vertices, i.e., $\Gamma^{\rm st}$ is stable; (2) $\Gamma^{\rm st}$ consists of a unique finite vertex, at most two infinite vertices, and no bounded edges; (3) $\Gamma^{\rm st}$ consists of a unique finite vertex and a loop, i.e., a unique bounded edge connecting the vertex to itself. However, since by the assumption and the construction of $\Gamma^{\rm st}$ it must satisfy \eqref{stcond}, it follows that $\Gamma^{\rm st}$ is stable. Note that by the construction $V^\infty(\Gamma)=V^\infty(\Gamma^{\rm st})$ and $\Gamma$ can be obtained from $\Gamma^{\rm st}$ using Algorithm \ref{alg:trcurves}. The uniqueness part of the claim is obvious.
\end{proof}
\begin{remark} Condition \eqref{stcond} is the analog of the stability condition for the algebraic curves: any rational curve must have at least three special points, and any elliptic curve must have at least one special point. For tropical curves of positive genus, \eqref{stcond} holds if and only if the Euler characteristic of the punctured graph is negative: $\chi(\Gamma\setminus V^\infty(\Gamma))<0$.
\end{remark}
\begin{definition}
$\Gamma^{\rm st}$ is called the {\it stabilization} of $\Gamma$.
\end{definition}
\begin{remark}
In \cite{GM08}, Gathmann and Markwig defined stabilization for tropical curves without one-valent finite vertices. In the latter case, the two stabilizations coincide.
\end{remark}
\begin{remark} If $\Gamma$ has no one-valent finite vertices then the underlying metric topological spaces of $\Gamma$ and $\Gamma^{\rm st}$ are naturally isometric. Vice versa, if the underlying metric topological spaces of $\Gamma$ and $\Gamma'$ are isometric then $\Gamma^{\rm st}\cong \Gamma'^{\rm st}$.
\end{remark}

\subsubsection{The $\QQ$-tropical curve assigned to a pair $(C,D)$.}\label{subsec:trcurvesforcurvewithmp}
Let $(C,D)$ be as in Subsection~\ref{subsec:convnot}, and $(C_{R_\LL}, D_{R_\LL})$ be a nodal model of $(C,D)$. One can associate to it a tropical curve, which will be denoted by $\Gamma_{C_{R_\LL},D_{R_\LL}}$. The underlying graph of $\Gamma_{C_{R_\LL},D_{R_\LL}}$ is defined as follows: the set of finite vertices is the set of irreducible components of the reduction of $C_{R_\LL}$, and the set of infinite vertices is the set of marked points $D\equiv D_{R_\LL}$.
The set of edges connecting two finite vertices is defined to be the set of common nodes of the corresponding components. In particular, if a component $C_v$ is singular then each singular point of $C_v$ corresponds to a loop at the corresponding finite vertex. Finally, if a marked point specializes to certain component then the corresponding vertices are connected by an unbounded edge.

\begin{notation}\label{not:qvCvpvpe} For $v\in V^f(\Gamma_{C_{R_\LL},D_{R_\LL}})$, the corresponding component is denoted by $C_v$, and for $v\in V^\infty(\Gamma_{C_{R_\LL},D_{R_\LL}})$, the corresponding marked point is denoted by $q_v$. If $e$ is a bounded (resp. unbounded) edge then the corresponding node (resp. specialization of the marked point) is denoted by $p_e$. Finally, $p_v$ denotes the specialization of $q_v$.
\end{notation}
\noindent It remains to specify the lengths of the bounded edges of $\Gamma_{C_{R_\LL},D_{R_\LL}}$. For a bounded edge $e$, set $|e|:=\frac{r_e+1}{e_\LL}$ if $C_{R_\LL}$ has singularity of type $A_{r_e}$ at $p_e$. Observe that the length $|e|$ is independent of $\LL$. Indeed, if $\LL\subset\LL'$,
$(C_{R_{\LL'}},D_{R_{\LL'}})=(C_{R_{\LL}},D_{R_{\LL}})\times_{\Spec R_\LL}\Spec R_{\LL'}$, and $C_{R_\LL}$ has singularity of type $A_r$ at a node $p$ then $C_{R_{\LL'}}$ has singularity of type $A_{e_{\LL'/\LL}(r+1)-1}$; hence $\frac{r+1}{e_\LL}=\frac{e_{\LL'/\LL}(r+1)-1+1}{e_{\LL'}}$.

Note that if $(C,D)$ is stable then it admits a distinguished model, namely {\em the stable model,} and the associated tropical curve is independent of the field extension $\LL$. Furthermore, the latter is the stabilization of the tropical curve associated to {\it any} nodal model $(C_{R_{\LL}},D_{R_{\LL}})$ of $(C,D)$.
\begin{notation} If $(C,D)$ is stable then the tropical curve associated to the stable model of $(C,D)$ is denoted by $\Gamma^{\rm st}_{C,D}$.
\end{notation}

\subsubsection{Existence of models with given metric graphs.}\label{subsec:modelforgraph}
It is natural to ask the following question: {\it Given a $\QQ$-tropical curve $\Gamma$, are there an extension $\LL$ and a nodal model $(C_{R_\LL},D_{R_\LL})$ such that $\Gamma=\Gamma_{C_{R_\LL},D_{R_\LL}}$?} The answer is given in the following proposition:

\begin{proposition}\label{prop:modelforgraph} Assume that $(C,D)$ is stable. Let $\Gamma$ be a $\QQ$-tropical curve. Then $\Gamma=\Gamma_{C_{R_\LL},D_{R_\LL}}$ for a nodal model $(C_{R_{\LL}}, D_{R_\LL})$ if and only if $\Gamma^{\rm st}=\Gamma^{\rm st}_{C, D}$. The model $(C_{R_{\LL}}, D_{R_\LL})$ is defined over any field $\LL$ satisfying the following two conditions: (a) the stable model is defined over $\LL$, and (b)
\begin{equation}\label{eq:resonlen}
|e|\in\frac{1}{e_\LL}\NN,
\end{equation}
for any $e\in E^b(\Gamma)$. Moreover, if $\Gamma$ can be obtained from $\Gamma^{\rm st}_{C, D}$ using only steps one and two of Algorithm \ref{alg:trcurves} then this model is unique up to unique isomorphism and field extensions.
\end{proposition}
\begin{proof}
Since any nodal model dominates the stable model, and the stable model can be obtained from it by Algorithm \ref{alg:models}, the ``only if'' part follows.
Let us now show the ``if'' part. Let $\LL$ be an extension over which the stable model is defined, and condition \eqref{eq:resonlen} is satisfied for any $e\in E^b(\Gamma)$. Let $\Gamma'$ be the metric graph obtained from $\Gamma$ by subdividing any bounded edge $e$ into a chain of $r_e+1=e_{\LL}|e|$ subintervals of length $\frac{1}{e_\LL}$. It is sufficient to construct the model $(C'_{R_{\LL}},D'_{R_\LL})$ with $\Gamma'=\Gamma_{C'_{R_\LL},D'_{R_\LL}}$, since $(C_{R_{\LL}}, D_{R_\LL})$ is obtained from $(C'_{R_{\LL}},D'_{R_\LL})$ by a uniquely defined sequence of blow-downs, namely one must blow down the projective lines on $C'_{R_{\LL}}$ corresponding to $V^f(\Gamma')\setminus V^f(\Gamma)$.

Note that $\Gamma'$ can be obtained from $\Gamma^{\rm st}_{C, D}$ by the three steps of Algorithm~\ref{alg:trcurves}. Moreover, it can be obtained from $\Gamma_{C^{\rm mrss}_{R_\LL},D_{R_\LL}}$ using only step two and step three with bounded trees, where $C^{\rm mrss}_{R_\LL}$ denotes the minimal regular semi-stable model dominating $C^{\rm st}_{R_\LL}$. It is easy to see that there exists a sequence of blowups along smooth points of the reduction of the minimal regular semi-stable model $C^{\rm mrss}_{R_\LL}\to C^{\rm st}_{R_\LL}$ such that the resulting regular semi-stable model $(C'_{R_{\LL}},D'_{R_\LL})$ has metric graph $\Gamma'=\Gamma_{C'_{R_\LL},D'_{R_\LL}}$.  For the moreover part, note that the sequence of blowups corresponding to the second step of the algorithm is uniquely defined. Indeed, if $e\in E^\infty(\Gamma_{C^{\rm mrss}_{R_\LL},D_{R_\LL}})$ is subdivided into $k$ pieces then the corresponding sequence consists of $k$ consecutive blowups along the reduction of $q_e$.
\end{proof}

\subsection{Parameterized tropical curves.}
\begin{definition}\label{def:partrcur}
Let $N$ be a lattice.
\begin{enumerate}
\item An {\it $N_\RR$-parameterized tropical curve} is a pair $(\Gamma, h_\Gamma)$ consisting of a tropical curve $\Gamma$ and a map $h_\Gamma\colon V(\Gamma)\to N_\RR$ that satisfy the following properties:
\begin{enumerate}
  \item $h_\Gamma(v)\in N$ for any infinite vertex $v\in V^\infty(\Gamma)$;
  \item $\frac{1}{|e|}(h_\Gamma(v)-h_\Gamma(v'))\in N$ for any bounded edge $e\in E_{vv'}(\Gamma)$;
  \item ({\it Balancing condition}) for any finite vertex $v$ the following holds: $$\sum_{v'\in V^f(\Gamma),\, e\in E_{vv'}(\Gamma)}\frac{1}{|e|}\left(h_\Gamma(v')-h_\Gamma(v)\right)+\sum_{v'\in V^\infty(\Gamma),\, e\in E_{vv'}(\Gamma)}h_\Gamma(v')=0.$$
\end{enumerate}
\item If $h_\Gamma(v)\in N_\QQ$ for any $v$ then $\Gamma$ is called {\it $N_\QQ$-parameterized $\QQ$-tropical curve}.
\end{enumerate}
\end{definition}
\begin{remark}
If no confusion is possible, we will often omit $h_\Gamma$ and, by abuse of language and notation, will refer to $\Gamma$ as $N_\RR$-parameterized tropical curve.
\end{remark}
\begin{remark}
Usually, one defines a parameterized tropical curve as a tropical curve $\Gamma$ equipped with a map $h\colon\Gamma\setminus V^\infty(\Gamma)\to N_\RR$ satisfying certain properties. Note, that after one identifies the edges with straight intervals, a parameterized tropical curve in the sense of Definition~\ref{def:partrcur} defines a usual parameterized tropical curve as follows: $h$ is the unique continuous map that coincides with $h_\Gamma$ on the set of finite vertices, maps boun\-ded edges $e\in E_{vv'}^b(\Gamma)$ linearly onto the intervals $[h_\Gamma(v), h_\Gamma(v')]$, and maps unbounded edges $e\in E_{vv'}^\infty(\Gamma)$ linearly onto the rays $\{h_\Gamma(v)+th_\Gamma(v')\,|\,t\in \RR_+\}$ if $v\in V^f(\Gamma)$ and $v'\in V^\infty(\Gamma)$.

Note, that although $h_\Gamma$ is defined for any vertex, it has different meanings for finite and for infinite vertices: If $v\in V^f(\Gamma)$ then one must think about $h_\Gamma(v)$ as {\em a point in the affine space} $N_\QQ$, and if ${v\in V^\infty(\Gamma)}$ then one must think about $h_\Gamma(v)$ as {\em a vector in the corresponding vector space} $N_\QQ$.
\end{remark}
\begin{definition}\label{def:proppartropcur} Let $\Gamma$ be an $N_\RR$-parameterized tropical curve, $v\in V^f(\Gamma)$ be a finite vertex, and $e\in E_{vv'}(\Gamma)$ be an edge.
\begin{enumerate}
  \item The {\it multiplicity $l(e)$} of an edge $e$ is the integral length of $h_\Gamma(v')$ if $e$ is unbounded, and is the integral length of $\frac{1}{|e|}(h_\Gamma(v)-h_\Gamma(v'))$ if $e$ is bounded; in the latter case, the multiplicity is exactly the factor by which $h_\Gamma$ stretches $e$ with respect to the lattice length on $N_\RR$.
  \item If $v'\in V^\infty(\Gamma)$ then the {\it multiplicity $l(v')$} of $v'$ is the integral length of $h_\Gamma(v')$.
  \item The {\it slope} of $e$ is $\RR\cdot(h_\Gamma(v)-h_\Gamma(v'))\subseteq N_\RR$ if $e$ is bounded, and $\RR\cdot h_\Gamma(v')\subseteq N_\RR$ if $e$ is unbounded. The slope of $e$ is denoted by $N_{\RR,e}$, and the lattice $N\cap N_{\RR,e}$ is denoted by $N_e$. If the slope $N_{\RR,e}$ is not trivial then $N_e$ and $N_{\RR,e}$ have a generator $n_e$, given by $n_e=\frac{1}{l(e)|e|}(h_\Gamma(v)-h_\Gamma(v'))$ if $e$ is bounded, and $n_e=\frac{1}{l(e)}h_\Gamma(v')$ if $e$ is unbounded. In the second case it is a {\it distinguished} generator, while in the first case it is defined only up-to a sign. However, if an {\em orientation} of the bounded edge is given then the generator is also distinguished.
  \item The {\it degree $\deg(\Gamma)$} of $\Gamma$ is the collection of pairs $(n_k,d_k)$, where $\{n_1,\dotsc ,n_s\}$ is the set of non-zero distinguished generators of slopes of unbounded edges, and $d_k=\sum_{e\in E^\infty(\Gamma), n_e=n_k}l(e)$.
\end{enumerate}
\end{definition}
\begin{remark}
Balancing condition implies $\sum_{(n,d)\in\deg(\Gamma)}dn=0$.
\end{remark}

\begin{proposition}\label{prop:paramtropcurandsubdivisions}
Let $\Gamma$ and $\Gamma'$ be tropical curves, such that $\Gamma'$ is obtained from $\Gamma$ using Algorithm~\ref{alg:trcurves}.
If $\Gamma$ has the structure of an $N_\RR$-parameterized tropical curve, then there exists a unique structure of an $N_\RR$-parameterized tropical curve on $\Gamma'$ satisfying the following two properties: (a) $h_{\Gamma'}(v)=0$ for all $v\in V^\infty(\Gamma')\setminus V^\infty(\Gamma)$, and (b) $h_{\Gamma'}(v)=h_\Gamma(v)$ for all $v\in V(\Gamma')\cap V(\Gamma)$. Vice versa, if $\Gamma'$ has the structure of an $N_\RR$-parameterized tropical curve such that $h_{\Gamma'}(v)=0$ for all $v\in V^\infty(\Gamma')\setminus V^\infty(\Gamma)$ then its restriction to $\Gamma$ defines the structure of an $N_\RR$-parameterized tropical curve on $\Gamma$.
\end{proposition}

\begin{proof} It is sufficient to prove the proposition for $\Gamma'$ obtained using only one step of the algorithm; furthermore, we may assume that only one edge (resp. metric tree) is subdivided (resp. attached).

First, assume that $\Gamma'$ is obtained from $\Gamma$ by a subdivision of a bounded edge {$e\in E_{vv'}(\Gamma)$}. Let $v_1,\dotsc ,v_r$ be the new vertices, $v_0=v$, $v_{r+1}=v'$, and $e_k\in E_{v_kv_{k+1}}(\Gamma')$ be the new edges, $\sum_{k=0}^r|e_k|=|e|$. If $\Gamma'$ has a structure of an $N_\RR$-parameterized tropical curve then it follows from the balancing condition that $$\frac{h_{\Gamma'}(v_{k+1})-h_{\Gamma'}(v_k)}{|e_k|}=\frac{h_{\Gamma'}(v_k)-h_{\Gamma'}(v_{k-1})}{|e_{k-1}|}=\frac{h_{\Gamma'}(v')-h_{\Gamma'}(v)}{|e|}$$ for all $1\le k\le r$. Thus $h_{\Gamma'}(v_k)=\frac{\sum_{j=k}^{r}|e_j|}{|e|}h_{\Gamma'}(v)+\frac{\sum_{j=0}^{k-1}|e_j|}{|e|}h_{\Gamma'}(v')$ for  $1\le k\le r$, which implies the uniqueness and the vice versa parts of the proposition.

Second, assume that $\Gamma'$ is obtained from $\Gamma$ by a subdivision of an unbounded edge $e\in E_{vv'}(\Gamma)$ with $v'\in V^\infty(\Gamma)$. Let $v_1,\dotsc ,v_r$ be the new vertices, $v_0=v$, $v_{r+1}=v'$, and $e_k\in E_{v_kv_{k+1}}(\Gamma')$ be the new edges. If $\Gamma'$ has a structure of an $N_\RR$-parameterized tropical curve then it follows from the balancing condition that $$\frac{1}{|e_k|}\left(h_{\Gamma'}(v_{k+1})-h_{\Gamma'}(v_k)\right)=\frac{1}{|e_{k-1}|}\left(h_{\Gamma'}(v_k)-h_{\Gamma'}(v_{k-1})\right)=h_{\Gamma'}(v')$$ for all $1\le k\le r-1$. Thus $h_{\Gamma'}(v_k)=h_{\Gamma'}(v)+(\sum_{j=0}^{k-1}|e_j|)h_{\Gamma'}(v')$ for  $1\le k\le r$, which implies the uniqueness and the vice versa parts of the proposition.

Third, assume that $\Gamma'$ is obtained from $\Gamma$ by attaching a metric tree $T$ to a vertex $v\in V^f(\Gamma)$, and that it has a structure of an $N_\RR$-parameterized tropical curve satisfying (a) and (b). Then $h_{\Gamma'}$ vanishes on the infinite leaves of the tree $T$, which, by balancing condition, implies that the slopes of all edges of $T$ are trivial. Hence, $h_{\Gamma'}(w)=h_{\Gamma'}(v)$ for all vertices $w$ of $T$ but infinite leaves, which implies the uniqueness and the vice versa parts of the proposition.

To prove the existence part, we define $h_{\Gamma'}$ on the new vertices by the formulae obtained above. Namely, in the first case set $h_{\Gamma'}(v_k):=\frac{\sum_{j=k}^{r}|e_j|}{|e|}h_{\Gamma}(v)+\frac{\sum_{j=0}^{k-1}|e_j|}{|e|}h_{\Gamma}(v')$, in the second case set $h_{\Gamma'}(v_k):=h_{\Gamma}(v)+(\sum_{j=0}^{k-1}|e_j|)h_{\Gamma}(v')$, and in the third case set $h_{\Gamma'}(w):=h_{\Gamma}(v)$ for all new vertices $w$ but infinite leaves, for which set $h_{\Gamma'}$ to be zero. Then, it is easy to see that $(\Gamma',h_{\Gamma'})$ is an $N_\RR$-parameterized tropical curve.
\end{proof}
\begin{corollary}\label{cor:l(e)fordifferentmodels}
Let $\Gamma$ and $\Gamma'$ be as in Proposition~\ref{prop:paramtropcurandsubdivisions}. Then the following holds: $\{l(e)\}_{e\in E(\Gamma)}\subseteq\{l(e)\}_{e\in E(\Gamma')}\subseteq \{l(e)\}_{e\in E(\Gamma)}\cup\{0\}$. Moreover, if $\Gamma'$ is obtained from $\Gamma$ using only the first two steps of the algorithm then $\{l(e)\}_{e\in E(\Gamma)}=\{l(e)\}_{e\in E(\Gamma')}$.
\end{corollary}
\begin{proof}
Obvious.
\end{proof}
\begin{corollary}\label{cor:lengthsfromweights}
Let $\Gamma$ be an $N_\RR$-parameterized tropical curve, and $\Gamma'$ be a graph obtained from $\Gamma$ by subdivision of some of the edges with non-trivial slopes. Assume that for each bounded (resp. unbounded) edge $e\in E_{vv'}(\Gamma)$ with $v\in V^f(\Gamma)$, which is subdivided into $r+1$ edges $e_k\in E_{v_kv_{k+1}}(\Gamma')$, $0\le k\le r$, we are given a sequence of numbers $0=\lambda_0<\lambda_1<\dotsb<\lambda_r<\lambda_{r+1}=1$ (resp. $0=\lambda_0<\lambda_1<\dotsb<\lambda_r$). Set $n_{v_k}:=\lambda_k h_{\Gamma}(v')+ (1-\lambda_k)h_{\Gamma}(v)$ ( resp. $n_{v_k}:=h_{\Gamma}(v)+\lambda_k h_{\Gamma}(v')$) for all $1\le k\le r$. Then there exists a unique structure of an $N_\RR$-parameterized tropical curve on $\Gamma'$ such that $h_{\Gamma'}(v)=h_{\Gamma}(v)$ for all $v\in V(\Gamma')\cap V(\Gamma)$, the lengths of non-subdivided edges in the two curves coincide, and $h_{\Gamma'}(v_k)=n_{v_k}$ for the new vertices $v_k\in V(\Gamma')\setminus V(\Gamma)$.
\end{corollary}
\begin{proof} Without loss of generality we may assume that $\Gamma'$ is obtained by subdivision of only one edge.
For the existence part, set $|e_k|:=(\lambda_{k+1}-\lambda_k)|e|$ for all $0\le k\le r$ (resp. $|e_k|:=\lambda_{k+1}-\lambda_k$ for any $0\le k\le r-1$, and $|e_r|=\infty$), and apply Proposition~\ref{prop:paramtropcurandsubdivisions}. For the uniqueness part, observe that $h_{\Gamma'}(v_k)=\frac{\sum_{j=k}^{r}|e_j|}{|e|}h_{\Gamma'}(v)+\frac{\sum_{j=0}^{k-1}|e_j|}{|e|}h_{\Gamma'}(v')$ (resp. $h_{\Gamma'}(v_k)=h_{\Gamma'}(v)+(\sum_{j=0}^{k-1}|e_j|)h_{\Gamma'}(v')$) for all $1\le k\le r$ (see the proof of Proposition~\ref{prop:paramtropcurandsubdivisions}). Note that $h_{\Gamma'}(v')=h_{\Gamma}(v')\ne h_{\Gamma}(v)=h_{\Gamma'}(v)$ (resp. $h_{\Gamma'}(v')=h_{\Gamma}(v')\ne 0$) since the slope of $e$ is non-trivial. Thus, $\lambda_k=\frac{\sum_{j=0}^{k-1}|e_j|}{|e|}$ (resp. $\lambda_k=\sum_{j=0}^{k-1}|e_j|$) for all $1\le k\le r$, which implies $|e_k|=(\lambda_{k+1}-\lambda_k)|e|$ for all $0\le k\le r$ (resp. $|e_k|=\lambda_{k+1}-\lambda_k$ for any $0\le k\le r-1$, and $|e_r|=\infty$).
\end{proof}
\begin{proposition}\label{prop:contOfTrCurve} Let $\Gamma$ be an $N_\RR$-parameterized tropical curve, and $\Gamma_0\subset \Gamma$ be the maximal metric subgraph satisfying the following two properties: (1)  $V(\Gamma_0)=V^f(\Gamma)$, $E(\Gamma_0)\subseteq E^b(\Gamma)$, and (2) $N_e=0$ for all edges $e\in E(\Gamma_0)$. Consider the weighted metric graph $\overline{\Gamma}=\Gamma/\Gamma_0$ obtained from $\Gamma$ by contracting the maximal connected subgraphs of $\Gamma_0$ to vertices. Then $\overline{\Gamma}$ is an $N_\RR$-parameterized tropical curve and $g(\overline{\Gamma})\le g(\Gamma)$.
\end{proposition}
\begin{proof} Obvious.
\end{proof}

\subsubsection{The $N_\QQ$-parameterized $\QQ$-tropical curve assigned to an algebraic curve with a rational map to a torus.}\label{subsec:paramtrcurforacurwithamap}
Let $f\colon C\setminus D\to T_{N,\overline{\FF}}$ be a morphism, and let $(C_{R_\LL},D_{R_\LL})$ be a nodal model of $(C, D)$. Then the $\QQ$-tropical curve $\Gamma=\Gamma_{C_{R_{\LL}},D_{R_\LL}}$ inherits the structure of an $N_\QQ$-parameterized $\QQ$-tropical curve from $f$. Indeed, let $v$ be a vertex. To simplify the notation let us identify it with the corresponding marked point $q_v$ or the irreducible component $C_v$. Then, the order of vanishing ${\rm ord}_v(f^*(x^m))$ is a linear function on $M$, hence an element of $N$. Set $h_\Gamma(v):=\frac{1}{e_\LL}{\rm ord}_v(f^*(x^\bullet))\in N_\QQ$ if $v$ is finite, and $h_\Gamma(v):={\rm ord}_v(f^*(x^\bullet))\in N$ if $v$ is infinite.

\begin{remark}\label{rem:divisib}
$ $

\begin{enumerate}
    \item[(1)] For $v\in V^\infty$, $h_\Gamma(v)=0$ if and only if $f$ can be extended to $q_v$.
    \item[(2)] If $e\in E_{vv'}$ is a bounded edge then $l(e)(r_e+1)$ is equal to the integral length of $e_\LL(h_\Gamma(v)-h_\Gamma(v'))$; in particular, the latter is divisible by $l(e)$.
    \item[(3)] If $\LL\subset\LL'$ is a finite extension, and $(C_{R_{\LL'}},D_{R_{\LL'}})=(C_{R_{\LL}},D_{R_{\LL}})\times_{\Spec R_\LL}\Spec R_{\LL'}$ then there exists a cano\-nical isomorphism $\iota\colon \Gamma=\Gamma_{C_{R_\LL},D_{R_\LL}}\to\Gamma_{C_{R_{\LL'}},D_{R_{\LL'}}}=\Gamma'$ and $h_\Gamma=h_{\Gamma'}\circ\iota$.
\end{enumerate}
\end{remark}
\begin{lemma}\label{lem:bal} $(\Gamma, h_\Gamma)$ is an $N_\QQ$-parameterized $\QQ$-tropical curve.
\end{lemma}
\begin{proof} All we have to show is that (i) the vector $\frac{1}{|e|}(h_\Gamma(v)-h_\Gamma(v'))\in N_\QQ$ is integral for any bounded edge $e\in E_{vv'}$, and (ii) the balancing condition of Definition~\ref{def:partrcur} is satisfied.

Assume first, that $(C_{R_\LL},D_{R_\LL})$ is {\it regular}. Then all bounded edges of $\Gamma$ have length $\frac{1}{e_\LL}$. Note that $t_\LL^{-(e_\LL h_\Gamma(v),m)}f^*(x^m)$ is a rational function on $C_v$ for all $m\in M$, hence $\frac{1}{|e|}(h_\Gamma(v')-h_\Gamma(v))=e_\LL(h_\Gamma(v')-h_\Gamma(v))\in N$ for all bounded edges $e\in E_{vv'}$, and the degree of the divisor of the function $t_\LL^{-(e_\LL h_\Gamma(v),m)}f^*(x^m)|_{C_v}$ is equal to
$$\sum_{v'\in V^f,\, e\in E_{vv'}}(e_\LL h_\Gamma(v')-e_\LL h_\Gamma(v), m)+\sum_{v'\in V^\infty,\, e\in E_{vv'}}(h_\Gamma(v'), m).$$
However, the degree of a rational function is zero, thus $$\sum_{v'\in V^f,\, e\in E_{vv'}}\frac{1}{|e|}(h_\Gamma(v')-h_\Gamma(v))+\sum_{v'\in V^\infty,\, e\in E_{vv'}}h_\Gamma(v')=0,$$ and we are done.

In general, let $(C'_{R_\LL},D'_{R_\LL})$ be the minimal regular nodal model dominating the model $(C_{R_\LL},D_{R_\LL})$. Then the graph $\Gamma_{C'_{R_\LL},D'_{R_\LL}}$ is obtained from $\Gamma_{C_{R_\LL},D_{R_\LL}}$ by subdivision of any bounded edge $e\in E_{vv'}(\Gamma_{C_{R_\LL},D_{R_\LL}})$ into $r_e+1=e_\LL|e|$ subintervals of length $\frac{1}{e_\LL}$, and $\Gamma_{C'_{R_\LL},D'_{R_\LL}}$ is an $N_\QQ$-parameterized $\QQ$-tropical curve. Thus, by Proposition~\ref{prop:paramtropcurandsubdivisions}, $\Gamma_{C_{R_\LL},D_{R_\LL}}$ satisfies the balancing condition, and for any $e\in E_{vv'}^b(\Gamma_{C_{R_\LL},D_{R_\LL}})$ there exists $e_0\in E_{v_0v_1}^b(\Gamma_{C'_{R_\LL},D'_{R_\LL}})$ such that $\frac{1}{|e|}(h_\Gamma(v)-h_\Gamma(v'))=\frac{1}{|e_0|}(h_\Gamma(v_0)-h_\Gamma(v_1))\in N$.
\end{proof}

\begin{notation} Assume that $(C,D)$ is stable, and $f\colon C\setminus D\to T_{N,\overline{\FF}}$ is a morphism. We denote by $\Gamma^{\rm st}_{C, D,f}$ the  $N_\QQ$-parameterized $\QQ$-tropical curve associated to the stable model of $(C,D)$.
\end{notation}

\subsubsection{Two complexes associated to an $N_\RR$-parameterized tropical curve.}
In this subsection we will work with a fixed $N_\RR$-parameterized tropical curve $\Gamma$. Fix an orientation of the bounded edges of $\Gamma$, and let $G$ be an abelian group. Consider the following linear maps
\begin{equation}\label{eq:complex}
b_G\colon\left(\bigoplus_{v\in V^f(\Gamma)}N_G\right)\oplus\left(\bigoplus_{e\in E^b(\Gamma)}(N_e)_G\right)\to\bigoplus_{e\in E^b(\Gamma)}N_G
\end{equation}
given by $b_G\colon x_v\mapsto \sum_{e\in E^b(\Gamma)}\epsilon(e,v)x_v$ and $b_G\colon x_e\mapsto x_e$ where $\epsilon(e,v)=-1$ if $v$ is the initial point of $e$, $\epsilon(e,v)=1$ if $v$ is the target of $e$, and $\epsilon(e,v)=0$ otherwise; and
\begin{equation}\label{eq:stackycomplex}
\beta_G\colon\left(\bigoplus_{v\in V^f(\Gamma)}N_G\right)\oplus\left(\bigoplus_{e\in E^b(\Gamma)}(N_e)_G\right)\to\bigoplus_{e\in E^b(\Gamma)}N_G
\end{equation}
given by $\beta_G\colon x_v\mapsto \sum_{e\in E^b(\Gamma)}\epsilon(e,v)x_v$ and $\beta_G\colon x_e\mapsto l(e)x_e$.
\begin{notation} Denote by $\EE_G^1(\Gamma)$ and $\EE_G^2(\Gamma)$ the kernel and the cokernel of $b_G$, and by
$\CE_G^1(\Gamma)$ and $\CE_G^2(\Gamma)$ the kernel and the cokernel of $\beta_G$. Set $\EE^\bullet(\Gamma):=\EE^\bullet_\ZZ(\Gamma)$ and $\CE^\bullet(\Gamma):=\CE^\bullet_\ZZ(\Gamma)$, and denote by $c(\Gamma)$ the number of bounded edges of $\Gamma$ having trivial slope.
\end{notation}
\begin{remark}\label{rem:anotherformcomplex} Let $\bigoplus_{v\in V^f(\Gamma)}N_G\to\bigoplus_{e\in E^b(\Gamma)}(N/N_e)_G$ be the complex, in which the map is given by $x_v\mapsto\sum_{e\in E^b(\Gamma)}(\epsilon(e,v)x_v \bmod (N_e)_G)$.
It is naturally quasi-isomorphic to complex \eqref{eq:complex}. Hence, below we will think about $\EE_G^1(\Gamma)$ and $\EE_G^2(\Gamma)$ as the kernel and the cokernel of either of these complexes. Note also, that if $l(e)\colon G\to G$ is an isomorphism for all $e$ then the complex
$$\bigoplus_{v\in V^f(\Gamma)}N_G\to\bigoplus_{e\in E^b(\Gamma)}(N/l(e)N_e)_G$$
with the map is given by $x_v\mapsto\sum_{e\in E^b(\Gamma)}(\epsilon(e,v)x_v \bmod (l(e)N_e)_G)$ is naturally quasi-isomorphic to complex \eqref{eq:stackycomplex}.
\end{remark}
\begin{Claim} In the above notation,
$\EE^2_G(\Gamma)=\EE^2(\Gamma)\otimes_\ZZ G$, and there is a natural exact sequence
$$0\to \EE^1(\Gamma)\otimes_\ZZ G\to \EE^1_G(\Gamma)\to Tor^1_\ZZ(\EE^2(\Gamma), G)\to 0\, .$$
The same statement holds true if $\EE^\bullet$ is replaced by $\CE^\bullet$.
\end{Claim}
\begin{proof}
Since $0\to \EE^1(\Gamma)\to\left(\bigoplus_{v\in V^f(\Gamma)}N\right)\oplus\left(\bigoplus_{e\in E^b(\Gamma)}N_e\right)\to\bigoplus_{e\in E^b(\Gamma)}N\to 0$ is a free resolution of $\EE^2(\Gamma)$, the cohomology of
$$0\to\EE^1(\Gamma)\otimes_\ZZ G\to\left(\bigoplus_{v\in V^f(\Gamma)}N_G\right)\oplus\left(\bigoplus_{e\in E^b(\Gamma)}(N_e)_G\right)\to\bigoplus_{e\in E^b(\Gamma)}N_G\to 0$$
computes the torsion groups $Tor^\bullet_\ZZ(\EE^2(\Gamma), G)$. Note that $Tor^2_\ZZ(\EE^2(\Gamma), G)=0$ since $\EE^2(\Gamma)$ admits a free resolution of length two. This implies the two identities for $\EE^\bullet_G(\Gamma)$. The proof for $\CE^\bullet_G(\Gamma)$ is identical.
\end{proof}
\begin{corollary} Let $\KK$ be a field. Then
$\EE^1_\KK(\Gamma)=\EE^1(\Gamma)\otimes_\ZZ\KK$ if and only if the order of the torsion part of $\EE^2(\Gamma)$ is prime to the characteristic of $\KK\, .$ The same statement holds true if $\EE^\bullet$ is replaced by $\CE^\bullet$.
\end{corollary}

\begin{proposition}\label{prop:relEECEGamma} There exists a natural exact sequence
$$
\begin{array}{lll}
  0 &\hspace{-0.25cm} \to &\hspace{-0.25cm}  \bigoplus_{e\in E^b(\Gamma), N_e\ne 0}\left(\mu_{l(e)}(G)\right)\to \CE^1_G(\Gamma)\to \EE^1_G(\Gamma)\to\\
    &\hspace{-0.25cm} \to &\hspace{-0.25cm}  \bigoplus_{e\in E^b(\Gamma), N_e\ne 0}\left(G/l(e)G\right)\to\CE^2_G(\Gamma)\to \EE^2_G(\Gamma)\to 0 ,
\end{array}
$$
where $\mu_{l(e)}(G)=\ker\left(l(e)\colon G\to G\right)$, and $l(e)\colon G\to G$ is the multiplication by $l(e)$.
\end{proposition}
\begin{proof} If the slope $N_e$ is not trivial then $(N_e)_G$ is canonically isomorphic to $G$ since $N_e$ has a distinguished generator (cf. Definition~\ref{def:proppartropcur}). Thus, $\bigoplus_{e\in E^b(\Gamma), N_e\ne 0}\left(\mu_{l(e)}(G)\right)$ and $\bigoplus_{e\in E^b(\Gamma), N_e\ne 0}\left(G/l(e)G\right)$ are the kernel and the cokernel of the natural morphism
$$\left(\bigoplus_{v\in V^f(\Gamma)}N_G\right)\oplus\left(\bigoplus_{e\in E^b(\Gamma)}(N_e)_G\right)\to \left(\bigoplus_{v\in V^f(\Gamma)}N_G\right)\oplus\left(\bigoplus_{e\in E^b(\Gamma)}(N_e)_G\right)$$
given by the identity map on $N_G$-s and by the multiplication by $l(e)$ on $(N_e)_G$-s. It extends to a morphism from \eqref{eq:stackycomplex} to \eqref{eq:complex} by the identity map between the right terms. The proposition now follows from the standard homological algebra computation, which we leave to the reader.
\end{proof}
\begin{corollary}\label{cor:EECE}
Assume that $l(e)$ is prime to $char(\kk)$ for all $e$. Then $\EE^\bullet_\kk(\Gamma)=\CE^\bullet_\kk(\Gamma)$, $\EE^2_{\kk^*}(\Gamma)=\CE^2_{\kk^*}(\Gamma)$, and $|\CE^1_{\kk^*}(\Gamma)|=|\EE^1_{\kk^*}(\Gamma)|\prod_{e\in E^b(\Gamma), N_e\ne 0}l(e)$.
\end{corollary}
\begin{proof}
Note that $l(e)\kk=\kk$, $\mu_{l(e)}(\kk)=0$, $|\mu_{l(e)}(\kk^*)|=l(e)$, and $(\kk^*)^{l(e)}=\kk^*$ since $l(e)$ is not divisible by the characteristic and $\kk$ is algebraically closed. The corollary now follows from the proposition.
\end{proof}
\begin{definition} The spaces $\EE^2_G(\Gamma)$ and $\CE^2_G(\Gamma)$ are called the {\it obstruction space} and the {\it stacky obstruction space} of $\Gamma$ over $G$. Following Mikhalkin \cite[Definition 2.22]{M05}, we say that $\Gamma$ is {\it $G$-regular} if $\CE^2_G(\Gamma)=0$. Otherwise it is called {\it $G$-superabundant}.
\end{definition}
\begin{remark} Note that by Proposition~\ref{prop:relEECEGamma}, if $G/l(e)G=0$ for any bounded edge $e$ with non-trivial slope then $\Gamma$ is $G$-regular if and only if $\EE^2_G(\Gamma)=0$. In particular, this is the case if either $G=\kk$ or $G=\kk^*$ and $l(e)$ are prime to the characteristic for all $e$.
\end{remark}

\begin{definition}
Let $\Gamma$ and $\Gamma'$ be $N_\RR$-parameterized tropical curves, and assume that $\Gamma$ is obtained from $\Gamma'$ by subdivision of (some of) its edges. Given an orientation on $\Gamma'$ we define the {\it induced} orientation on $\Gamma$ as follows: if edge $e\in E^b_{vv'}(\Gamma')$ is oriented from $v$ to $v'$ and is subdivided into a chain of edges $e_0,\dotsc, e_r\in E(\Gamma)$, i.e., there exists a chain of vertices $v=v_0, v_1,\dotsc, v_{r+1}=v'\in V(\Gamma)$ such that $E_{v_kv_{k+1}}(\Gamma)=\{e_k\}$ for all $0\le k\le r$, and $val(v_k)=2$ and $v_k\notin V(\Gamma')$ for all $1\le k\le r$, then we set the target of $e_k$ to be $v_{k+1}$; if $e\in E^\infty_{vv'}(\Gamma')$ is subdivided into a chain of edges $e_0,\dotsc ,e_r$ then all $e_k$ are oriented in the direction of the infinite vertex.
\end{definition}
\begin{proposition}\label{prop:E12ForSubdivisions} Let $\Gamma$ and $\Gamma'$ be $N_\RR$-parameterized tropical curves. Assume that $\Gamma$ is obtained from $\Gamma'$ by subdivision of (some of) its edges. Pick an orientation on $\Gamma'$, and consider the induced orientation on $\Gamma$.
Then $\EE^2_G(\Gamma')\cong \EE^2_G(\Gamma)$, $\CE^2_G(\Gamma')\cong \CE^2_G(\Gamma)$, and the following sequences are exact:
$$0\to \bigoplus_{v\in V^f(\Gamma)\setminus V^f(\Gamma')}(N_{e_v})_G\to \EE^1_G(\Gamma)\to \EE^1_G(\Gamma')\to 0$$ and $$0\to \bigoplus_{v\in V^f(\Gamma)\setminus V^f(\Gamma')}(N_{e_v})_G\to \CE^1_G(\Gamma)\to \CE^1_G(\Gamma')\to 0,$$ where $e_v\in E(\Gamma')$ denotes the edge subdivided by the vertex $v$. In particular, if $\Gamma$ satisfies \eqref{stcond} and $\Gamma'=\Gamma^{\rm st}$ then $\bigoplus_{v\in V^f(\Gamma)\setminus V^f(\Gamma')}N_{e_v}=\bigoplus_{v\in V_2^f(\Gamma)}N_{e_v}$.
\end{proposition}
\begin{proof}
Let $e\in E_{vv'}(\Gamma')$ be an edge with $v\in V^f(\Gamma')$. Then there exists a chain of vertices of $\Gamma$: $v_0=v, v_1,\dotsc, v_{r+1}=v'\in V(\Gamma)$, such that $E_{v_kv_{k+1}}(\Gamma)=\{e_k\}$ for all $0\le k\le r$, and $val(v_k)=2$ and $v_k\notin V(\Gamma')$ for all $1\le k\le r$. Furthermore, $N_{e_k}=N_e$ for all $k$ due to the balancing condition (cf. the proof of Proposition \ref{prop:paramtropcurandsubdivisions}), and the chains of edges $e_0,\dotsc, e_r$ are disjoint for different edges $e$. Thus, for each bounded edge $e$, the maps $\bigoplus_{k=0}^rN_G\to N_G$ and $\bigoplus_{k=0}^r(N_{e_k})_G\to (N_e)_G$ mapping $(x_k)$ to $\sum_{k=0}^rx_k$ are well defined. For unbounded edges $e$ consider the trivial maps $\bigoplus_{k=0}^{r-1}N_G\to 0$ and $\bigoplus_{k=0}^{r-1}(N_{e_k})_G\to 0$. Thus, by taking direct sums we obtain surjective linear maps $\bigoplus_{e\in E^b(\Gamma)}N_G\to \bigoplus_{e\in E^b(\Gamma')}N_G$ and $\bigoplus_{e\in E^b(\Gamma)}(N_e)_G\to \bigoplus_{e\in E^b(\Gamma')}(N_e)_G$.

Consider the natural projection $\bigoplus_{v\in V^f(\Gamma)}N_G\to \bigoplus_{v\in V^f(\Gamma')}N_G$ and the maps constructed above. Since the orientation on $\Gamma$ is induced from the orientation on $\Gamma'$, it follows from the definition of $b_G$ that the following diagram with exact rows is commutative:
$$\xymatrix{
0\ar[r] & \EE^1_G(\Gamma)\ar[r]\ar[d] & \left(\bigoplus_{v\in V^f}N_G\right)\oplus\left(\bigoplus_{e\in E^b(\Gamma)}(N_e)_G\right)\ar[r]\ar@{->>}[d]&\\
0\ar[r] & \EE^1_G(\Gamma')\ar[r] & \left(\bigoplus_{v\in V^f(\Gamma')}N_G\right)\oplus\left(\bigoplus_{e\in E^b(\Gamma')}(N_e)_G\right)\ar[r]&
}\hspace{1cm}$$
$$\hspace{5cm}\xymatrix{
\ar[r] & \bigoplus_{e\in E^b(\Gamma)}N_G\ar@{->>}[d]\ar[r] & \EE^2_G(\Gamma)\ar[d]\ar[r] & 0\\
\ar[r] & \bigoplus_{e\in E^b(\Gamma')}N_G\ar[r] & \EE^2_G(\Gamma')\ar[r] & 0
}$$
Consider the induced map $\phi$ between the kernels of the two central vertical arrows.
By the construction, it decomposes into a direct sum over the edges $e\in E(\Gamma')$ of the following summands:
$$\left(\bigoplus_{k=1}^{r(e)}N_G\right)\oplus\ker\left(\bigoplus_{k=0}^{r(e)}(N_{e_k})_G\to (N_e)_G\right)\to \ker\left(\bigoplus_{k=0}^{r(e)}N_G\to N_G\right)$$
if $e$ is bounded and
$$\left(\bigoplus_{k=1}^{r(e)}N_G\right)\oplus\left(\bigoplus_{k=0}^{r(e)-1}(N_{e_k})_G\right)\to \bigoplus_{k=0}^{r(e)-1}N_G$$
if $e$ is unbounded; where, as before, $e$ is subdivided by the vertices $v_1,\dotsc, v_{r(e)}$ into edges $e_0,\dotsc, e_{r(e)}$, and $v_0\in V^f(\Gamma')$. Observe that in both cases the map is surjective, hence so is $\phi$. Furthermore, we see that the kernel of $\phi$ is canonically isomorphic to $\bigoplus_{v\in V^f(\Gamma)\setminus V^f(\Gamma')}(N_{e_v})_G$, which implies the proposition. The proof for $\CE^\bullet(\Gamma)$ is identical.
\end{proof}
\begin{proposition}\label{prop:E1E2GammaGammaBar} Let $\Gamma$ and $\overline{\Gamma}$ be as in Proposition \ref{prop:contOfTrCurve}. Then the natural map $\EE^1_G(\overline{\Gamma})\to \EE^1_G(\Gamma)$ is an isomorphism, and there exists an exact sequence
$$0\to \EE^2_G(\overline{\Gamma})\to \EE^2_G(\Gamma)\to N_G^{g(\Gamma)-g(\overline{\Gamma})}\to 0.$$
The same statement holds true if $\EE^\bullet$ is replaced by $\CE^\bullet$.
\end{proposition}
\begin{proof} Let $\Gamma_0\subset \Gamma$ be as in Proposition \ref{prop:contOfTrCurve}. Then the finite vertices of $\overline{\Gamma}$ correspond to the connected components of $\Gamma_0$. Below we shall think about $\overline{v}\in V(\overline{\Gamma})$ as both: the vertices of $\overline{\Gamma}$ and the connected metric subgraphs of $\Gamma$. Consider the following diagram with exact rows:

$$\xymatrix{
0\ar[r] & \EE^1_G(\overline{\Gamma})\ar[r]\ar[d] & \left(\bigoplus_{\overline{v}\in V^f(\overline{\Gamma})}N_G\right)\oplus\left(\bigoplus_{e\in E^b(\overline{\Gamma})}(N_e)_G\right)\ar[r]\ar@{^(->}[d] & \\
0\ar[r] & \EE^1_G(\Gamma)\ar[r] & \left(\bigoplus_{v\in V^f(\Gamma)}N_G\right)\oplus\left(\bigoplus_{e\in E^b(\Gamma)}(N_e)_G\ar[r]\right) &
}\hspace{1cm}$$
$$\hspace{5cm}\xymatrix{
\ar[r] & \bigoplus_{e\in E^b(\overline{\Gamma})}N_G\ar@{^(->}[d]\ar[r] & \EE^2_G(\overline{\Gamma})\ar[d]\ar[r] & 0\\
\ar[r] & \bigoplus_{e\in E^b(\Gamma)}N_G\ar[r] & \EE^2_G(\Gamma)\ar[r] & 0
}$$
where the vertical arrow $\bigoplus_{\overline{v}\in V^f(\overline{\Gamma})}N_G\to \bigoplus_{v\in V^f(\Gamma)}N_G=\bigoplus_{\overline{v}\in V^f(\overline{\Gamma})}\left(\bigoplus_{v\in V(\overline{v})}N_G\right)$ is the direct sum of the diagonal embeddings, and $\bigoplus_{e\in E^b(\overline{\Gamma})}(N_e)_G\to \bigoplus_{e\in E^b(\Gamma)}(N_e)_G$ and $\bigoplus_{e\in E^b(\overline{\Gamma})}N_G\to \bigoplus_{e\in E^b(\Gamma)}N_G$ are the natural embedding. Then the above diagram is commutative. Consider the induced map of the cokernels of the middle vertical arrows:
$\bigoplus_{\overline{v}\in V^f(\overline{\Gamma})}\left[\bigoplus_{v\in V(\overline{v})}N_G/\Delta N_G\to \bigoplus_{e\in E(\overline{v})}N_G\right].$
Plainly, it is injective and its cokernel is canonically isomorphic to
$$\bigoplus_{\overline{v}\in V^f(\overline{\Gamma})}H^1(\overline{v}, N_G)\approx \bigoplus_{\overline{v}\in V^f(\overline{\Gamma})}\left(N_G^{g(\overline{v})}\right)\approx N_G^{g(\Gamma)-g(\overline{\Gamma})}\, .$$ This implies the proposition for $\EE^\bullet(\Gamma)$. The proof for $\CE^\bullet(\Gamma)$ is identical.
\end{proof}

\subsubsection{Deformations of $N_\RR$-parameterized tropical curves.}
\begin{definition}\label{def:deforpartrcur} Let $\Gamma$ be an $N_\RR$-parameterized tropical curve. By a {\it deformation} of $\Gamma$ we mean a germ (at $1$) of a continuous family $\{\Gamma_s\}_{s\in\RR}$ of $N_\RR$-parameterized tropical curves, with $\Gamma_1=\Gamma$.
\end{definition}
\begin{remark}
$ $

\begin{enumerate}
\item Any deformation of $\Gamma$ induces a deformation of the underlying graph. Since we work only with finite graphs, any such deformation can be canonically trivialized. Hence we may consider only deformations of $\Gamma$ inducing the trivial deformation of the underlying graph.
\item The multiplicities and the slopes of the edges are preserved by deformations since $h_{\Gamma_s}(v)\in N$ for any $v\in V^\infty$, and $\frac{1}{|e|_s}(h_{\Gamma_s}(v)-h_{\Gamma_s}(v'))\in N$ for any bounded edge $e\in E_{vv'}$. This also shows that the lengths $|e|_s$ of the edges $e\in E_{vv'}$ with non-trivial slopes are uniquely defined by the values of $h_{\Gamma_s}(v)$ and $h_{\Gamma_s}(v')$.
\end{enumerate}
\end{remark}
Fix an orientation of the bounded edges of $\Gamma$. Then the germ of the {\it universal} deformation of $\Gamma$, i.e., of the space of deformations up to isomorphism, can be identified naturally with the germ of the group $\EE_\RR^1(\Gamma)\times\RR_+^{c(\Gamma)}$ at the identity, where $c(\Gamma)$, as usual, denotes the number of bounded edges of $\Gamma$ with trivial slope.
\begin{definition} The {\it rank} of an $N_\RR$-parameterized tropical curve $\Gamma$ is the dimension of the universal deformation of $\Gamma$, i.e., $\rank(\Gamma)=c(\Gamma)+\rank(\EE^1(\Gamma))$.
\end{definition}
\begin{lemma} The rank of $\Gamma$ is given by the following formula: $$\rank(\Gamma)=(\rank(N)-3)\chi(\Gamma)+|E^\infty(\Gamma)|-ov(\Gamma)+\rank(\EE^2(\Gamma)),$$ where  overvalency $ov(\Gamma)$ is defined to be $ov(\Gamma)=\sum_{v\in V^f(\Gamma)}(val(v)-3)$.
\end{lemma}
\begin{proof}
By definition,
$$\rank(\Gamma)=c(\Gamma)+\dim\left(\bigoplus_{v\in V^f(\Gamma)}N_\RR\right)-\dim\left(\bigoplus_{e\in E^b(\Gamma)}(N/N_e)_\RR\right)+\rank(\EE^2(\Gamma)).$$ Since $N_e$ is either trivial or has rank one, and $c(\Gamma)$ is the number of bounded edges with trivial slope, it follows that
$$\rank(\Gamma)=\rank(N)|V^f(\Gamma)|-(\rank(N)-1)|E^b(\Gamma)|+\rank(\EE^2(\Gamma)).$$
Note that $\chi(\Gamma)=|V^f(\Gamma)|-|E^b(\Gamma)|$, since $|E^\infty(\Gamma)|=|V^\infty(\Gamma)|$. Thus,
$$\rank(\Gamma)=(\rank(N)-3)\chi(\Gamma)+3|V^f(\Gamma)|-2|E^b(\Gamma)|+\rank(\EE^2(\Gamma)),$$ and since $|E^\infty(\Gamma)|+\sum_{v\in V^f(\Gamma)}val(v)=|V^\infty(\Gamma)|+\sum_{v\in V^f(\Gamma)}val(v)=2|E(\Gamma)|$, the following holds:
$$3|V^f(\Gamma)|-2|E^b(\Gamma)|=|E^\infty(\Gamma)|+3|V^f(\Gamma)|-\sum_{v\in V^f(\Gamma)}val(v)=|E^\infty(\Gamma)|-ov(\Gamma).$$ Hence
$\rank(\Gamma)=(\rank(N)-3)\chi(\Gamma)+|E^\infty(\Gamma)|-ov(\Gamma)+\rank(\EE^2(\Gamma))$.
\end{proof}

\subsubsection{Linear constraints.}
\begin{definition}\label{def:afconstr}
Let $L_i\subset N$, $1\le i\le k$, be sublattices of coranks greater than or equal to two, such that $N/L_i$ is torsion free for any $i$. Let $A=\{A_i\}_{i=1}^k$, $A_i\subset N_\RR$, be a set of affine subspaces with tangent spaces $(L_i)_\RR$. Consider an $N_\RR$-parameterized tropical curve $\Gamma$.
\begin{enumerate}
  \item We say that $\Gamma$ {\it satisfies the affine constraint} $A$ if for any $1\le i\le k$ the following holds: $h_\Gamma(v_i)=0$ and $h_\Gamma(v'_i)\in A_i$, where $v_i\in V^\infty(\Gamma)$ is the i-th infinite vertex, and $v'_i$ is the unique finite vertex connected to $v_i$.
  \item If $\Gamma$ satisfies $A$ then we say that $A$ is a {\it simple constraint} for $\Gamma$ if for all $1\le i\le k$ the following holds: $v'_i$ is trivalent, $N_e\ne 0$, and $N_e\cap L_i=0$, where $e$ is any bounded edge containing $v'_i$ (note that the slopes of the two bounded edges containing $v'_i$ coincide due to the balancing condition, since $h_\Gamma(v_i)=0$).
  \item We define the codimension of $A$ to be $\codim (A):=\sum_{i=1}^k \codim_{N_\RR} (A_i)$.
\end{enumerate}
\end{definition}

\begin{example}\label{example:constr}
Let $L_i$ be as in Definition~\ref{def:afconstr}. Consider the subtori $T_{L_i,\FF}\subset T_{N,\FF}$, and a set $O:=\{O_i\}_{i=1}^k$ of $T_{L_i,\FF}$-orbits in $T_{N,\FF}$. Then $O$ defines an affine constraint $A=A_O$. Indeed, any $\FF$-point $q\in O_i$ defines a linear map from $M$ to $\ZZ$: $m\mapsto \val(x^m(q))$. The set of all such maps forms a lattice of maximal rank in some $L_i$-affine subspace, and we define $A_i$ to be this subspace.

Let $(C, D, f)$ be as usual, and assume that $f(q_i)\in O_i$ for all $1\le i\le k$. In this case we say that $(C, D, f)$ satisfies the {\it toric constraint $O$}. Let $(C_{R_\LL}, D_{R_\LL})$ be a semi-stable model, and let $\Gamma$ be the associated $N_\QQ$-parameterized $\QQ$-tropical curve. Then $\Gamma$ satisfies the corresponding affine constraint $A$.
\end{example}

Let $\Gamma$ be an $N_\RR$-parameterized tropical curve, $A$ be an affine constraint satisfied by $\Gamma$,  $v_1,\dotsc, v_k$ be the infinite vertices corresponding to $q_1,\dotsc, q_k$, and $G$ be an abelian group. Consider the map
$$\gamma_G\colon\left(\bigoplus_{v\in V^f(\Gamma)}N_G\right)\oplus\left(\bigoplus_{e\in E^b(\Gamma)}(N_e)_G\right)\to\bigoplus_{i=1}^k (N/L_i)_G$$
defined by $\gamma_G(x_e)=0$, $\gamma_G(x_v)=(x_v \bmod L_i)$ if $v$ is connected to $v_i$, and $\gamma_G(x_v)=0$ otherwise.
\begin{notation} We denote by $\EE_G^1(\Gamma, A)$ and $\EE_G^2(\Gamma, A)$ the kernel and the cokernel of the linear map $b_G\times\gamma_G$
\begin{equation}\label{eq:bGxgammaG}
\left(\bigoplus_{v\in V^f(\Gamma)}N_G\right)\oplus\left(\bigoplus_{e\in E^b(\Gamma)}(N_e)_G\right)\to\left(\bigoplus_{e\in E^b(\Gamma)}N_G\right)\times\left(\bigoplus_{i=1}^k (N/L_i)_G\right)
\end{equation}
Similarly, we denote by $\CE_G^1(\Gamma, A)$ and $\CE_G^2(\Gamma, A)$ the kernel and the cokernel of the linear map $\beta_G\times\gamma_G$
\begin{equation}\label{eq:betaGxgammaG}
\left(\bigoplus_{v\in V^f(\Gamma)}N_G\right)\oplus\left(\bigoplus_{e\in E^b(\Gamma)}(N_e)_G\right)\to\left(\bigoplus_{e\in E^b(\Gamma)}N_G\right)\times\left(\bigoplus_{i=1}^k (N/L_i)_G\right)
\end{equation}
If $G=\ZZ$ then we use shorter notation $\EE^\bullet(\Gamma, A)$ and $\CE^\bullet(\Gamma, A)$ instead of $\EE^\bullet_\ZZ(\Gamma, A)$ and $\CE^\bullet_\ZZ(\Gamma, A)$.
\end{notation}
\begin{remark}
One may think about $\EE_G^1(\Gamma, A)$ and $\EE_G^2(\Gamma, A)$ as the kernel and the cokernel of the map $$\bigoplus_{v\in V^f(\Gamma)}N_G\to\left(\bigoplus_{e\in E^b(\Gamma)}(N/N_e)_G\right)\oplus\left(\bigoplus_{i=1}^k (N/L_i)_G\right)$$ (cf. Remark~\ref{rem:anotherformcomplex}). Similarly, if $l(e)\colon G\to G$ is an isomorphism for all $e$ then one may think about $\CE_G^1(\Gamma, A)$ and $\CE_G^2(\Gamma, A)$ as the kernel and the cokernel of the map $$\bigoplus_{v\in V^f(\Gamma)}N_G\to\left(\bigoplus_{e\in E^b(\Gamma)}(N/l(e)N_e)_G\right)\oplus\left(\bigoplus_{i=1}^k (N/L_i)_G\right).$$
\end{remark}
\begin{Claim} Consider the germ of the universal deformation of $\Gamma$ which, as before, we identify with the germ of the group $\EE_\RR^1(\Gamma)\times\RR_+^{c(\Gamma)}$ at the identity. Then the locus of the deformations satisfying $A$ can be identified naturally with the germ at the identity of $\EE_\RR^1(\Gamma, A)\times\RR_+^{c(\Gamma)}$.
\end{Claim}
\begin{proof}
Obvious.
\end{proof}
\begin{definition}
Let $\Gamma$ be an $N_\RR$-parameterized tropical curve satisfying an affine constraint $A$, and $G$ be an abelian group. The pair $(\Gamma, A)$ is called {\it $G$-regular} if $A$ is a simple constraint, and $\CE_G^2(\Gamma, A)=0$. Otherwise it is called {\it $G$-superabundant}.
\end{definition}
\begin{proposition}\label{prop:E1E2GammaAGamma}
Consider the natural projections $a_G\colon\EE^1_G(\Gamma)\to\bigoplus_{i=1}^k (N/L_i)_G$ and $\alpha_G\colon\CE^1_G(\Gamma)\to\bigoplus_{i=1}^k (N/L_i)_G$. Then
$\EE^1_G(\Gamma, A)=\ker(a_G)$, $\CE^1_G(\Gamma, A)=\ker(\alpha_G)$, and there exist natural exact sequences $0\to {\rm coker}(a_G)\to \EE^2_G(\Gamma, A)\to \EE^2_G(\Gamma)\to 0$ and $0\to {\rm coker}(\alpha_G)\to \CE^2_G(\Gamma, A)\to \CE^2_G(\Gamma)\to 0$.
\end{proposition}
\begin{proof} Straightforward from the definitions.
\end{proof}

\begin{Claim}\label{cl:e1k*}
If $(\Gamma, A)$ is $\kk$-regular then it is $\kk^*$-regular. If, in addition, $c(\Gamma)=0$ and $\codim (A)=\rank(\Gamma)$ then $|\EE^1_{\kk^*}(\Gamma,A)|=|\EE^2(\Gamma, A)|$ and $|\CE^1_{\kk^*}(\Gamma,A)|=|\CE^2(\Gamma, A)|$.
\end{Claim}
\begin{proof} Since $(\Gamma, A)$ is $\kk$-regular, $\CE^2_{\kk}(\Gamma,A)=\CE^2(\Gamma,A)\otimes\kk=0$. Hence $\CE^2(\Gamma, A)$ is a torsion group of order prime to $char(\kk)$. Thus, $\CE^2_{\kk^*}(\Gamma,A)=\CE^2(\Gamma, A)\otimes_\ZZ\kk^*=0$ since $\kk$ is algebraically closed. Hence $(\Gamma, A)$ is $\kk^*$-regular. If $\codim (A)=\rank(\Gamma)$ and $c(\Gamma)=0$ then $\CE^1(\Gamma, A)=0$. Thus, $\CE^1_{\kk^*}(\Gamma,A)={\rm Tor}_\ZZ^1(\CE^2(\Gamma, A), \kk^*)$, and hence $|\CE^1_{\kk^*}(\Gamma,A)|=|\CE^2(\Gamma, A)|$. The proof for $\EE^\bullet$ is similar.
\end{proof}

The proofs of the following three propositions are identical to the proofs of Propositions \ref{prop:relEECEGamma}, \ref{prop:E12ForSubdivisions}, and \ref{prop:E1E2GammaGammaBar}:
\begin{proposition}\label{prop:relEECEGammaConstraints} There exists a natural exact sequence
$$\begin{array}{lll}
    0&\hspace{-0.25cm} \to&\hspace{-0.25cm} \bigoplus_{e\in E^b(\Gamma), N_e\ne 0}\left(\mu_{l(e)}(G)\right)\to \CE^1_G(\Gamma, A)\to \EE^1_G(\Gamma, A)\to\\
     &\hspace{-0.25cm} \to&\hspace{-0.25cm}  \bigoplus_{e\in E^b(\Gamma), N_e\ne 0}\left(G/l(e)G\right)\to \CE^2_G(\Gamma, A)\to\EE^2_G(\Gamma, A)\to 0
  \end{array}$$
where $\mu_{l(e)}(G)=\ker\left(l(e)\colon G\to G\right)$, and $l(e)\colon G\to G$ is the multiplication by $l(e)$.
\end{proposition}
\begin{proposition}\label{prop:E12forSubdivisionsConstraints}
Let $\Gamma$ and $\Gamma'$ be $N_\RR$-parameterized tropical curves, and assume that $\Gamma$ is obtained from $\Gamma'$ by subdivision of some of its edges. Pick an orientation on $\Gamma'$, and consider the induced orientation on $\Gamma$. Assume that $\Gamma'$ satisfies an affine constraint $A$. Then $\Gamma$ satisfies $A$, $\EE^2_G(\Gamma',A)\cong \EE^2_G(\Gamma,A)$, and the following sequence is exact:
$$0\to \bigoplus_{v\in V^f(\Gamma)\setminus V^f(\Gamma')}N_{e_v}\to \EE^1_G(\Gamma,A)\to \EE^1_G(\Gamma',A)\to 0,$$ where $e_v\in E(\Gamma')$ is the edge subdivided by the vertex $v$. In particular, if $\Gamma$ satisfies \eqref{stcond} and $\Gamma'=\Gamma^{\rm st}$ then $\bigoplus_{v\in V^f(\Gamma)\setminus V^f(\Gamma')}N_{e_v}=\bigoplus_{v\in V_2^f(\Gamma)}N_{e_v}$. The same statement holds true if $\EE^\bullet$ is replaced by $\CE^\bullet$.
\end{proposition}
\begin{proposition}\label{prop:E1E2GammaGammaBarConstraints} Let $\Gamma$ and $\overline{\Gamma}$ be as in Proposition \ref{prop:contOfTrCurve}.
If $\Gamma$ satisfies a simple constraint $A$ then $\overline{\Gamma}$ satisfies $A$, $A$ is a simple constraint for $\overline{\Gamma}$, the natural map $\EE^1_G(\overline{\Gamma}, A)\to \EE^1_G(\Gamma, A)$ is an isomorphism, and there exists an exact sequence
$$0\to \EE^2_G(\overline{\Gamma},A)\to \EE^2_G(\Gamma,A)\to N_G^{g(\Gamma)-g(\overline{\Gamma})}\to 0.$$ The same statement holds true if $\EE^\bullet$ is replaced by $\CE^\bullet$.
\end{proposition}

\subsection{Elliptic constraint.}\label{subsec:ellconstr}

\begin{definition}
Let $\Gamma$ be a tropical curve of genus one. {\it Tropical $j$-invariant of $\Gamma$} is the minimal length of a cycle generating the first homology of $\Gamma$.
\end{definition}
Let $\Gamma$ be the tropical curve corresponding to an integral model of $(C,D)$, and assume that $g(\Gamma)=g(C)=1$. As any two integral models of $(C,D)$ are dominated by a third one, and the tropical curve corresponding to the dominating model is obtained from $\Gamma$ by the steps of Algorithm~\ref{alg:trcurves}, it follows that the tropical $j$-invariant of $\Gamma$ is independent of the model; hence depends only on $C$, and is independent of $D$. The following theorem (to the best of our knowledge) goes back to Tate:
\begin{theorem}\label{thm:jtrjalg}
Let $\Gamma$ be the tropical curve corresponding to an integral model of $(C,D)$. If $g(C)=g(\Gamma)=1$ then the tropical $j$-invariant of $\Gamma$ is equal to $-\val (j(C))$, where $j(C)$ is the algebraic $j$-invariant of $C$.
\end{theorem}
\begin{proof}
Since the tropical $j$-invariant of $(C,D)$ depends only on $C$, we may assume that $D$ consists of one point and $\Gamma=\Gamma^{\rm st}_{C,D}$. Assume for simplicity that $char(\kk)\ne 2$. Then, there exists a finite extension $\FF\subseteq\LL$, and a scalar $\lambda\in \LL$, such that $C$ is a plane cubic given by $y^2=x(x-1)(x-\lambda)$. After replacing $\lambda$ with $\frac{1}{\lambda}$ if necessary we may assume that $\lambda\in R_\LL$. If $\val(\lambda)=\val(1-\lambda)=0$ then $C$ has good reduction and $g(\Gamma)=0$ which is a contradiction. After replacing $\lambda$ with  $\frac{\lambda-1}{\lambda}$ if necessary we may assume that $1-\lambda$ is invertible in $R_\LL$. Thus, $j(\Gamma)=2\val(\lambda)$, since, locally, the singularity of the total space is of the form $XY=\lambda^2$. Recall that the $j$-invariant of $C$ is given by $j(C)=2^8\frac{(\lambda^2-\lambda+1)^3}{\lambda^2(\lambda-1)^2}$. Thus, $j(\Gamma)=2\val(\lambda)=-\val(j(C))$ as required.
\end{proof}
\noindent We would like to define the groups $\EE_G^\bullet(\Gamma, A, j)$ and $\CE_G^\bullet(\Gamma, A, j)$ similarly to $\EE_G^\bullet(\Gamma, A)$  and $\CE_G^\bullet(\Gamma, A)$. It turns out that $\EE_G^\bullet(\Gamma, A, j)$ can be defined only if $l(e)\colon G\to G$ is an isomorphism for any bounded edge $e$ with non-trivial slope, e.g., $G$ is a field of characteristic prime to $l(e)$ for all $e$. However, $\CE_G^\bullet(\Gamma, A, j)$ can be defined for any $G$.

Let $\Gamma$ be an $N_\RR$-parameterized tropical curve of genus one satisfying an affine constraint $A$, and $e_1,\dotsc, e_k$ be the finite edges in the shortest cycle generating the first homology of $\Gamma$. Fix an orientation on the edges of $\Gamma$ for which $e_1,\dotsc, e_k$ is an oriented cycle. This induces an isomorphism $N_e\cong\ZZ$ for all $e\in E^b(\Gamma)$ with $N_e\ne 0$. Let $G$ be an abelian group, and $\delta_G\colon \left(\bigoplus_{v\in V^f(\Gamma)}N_G\right)\oplus\left(\bigoplus_{e\in E^b(\Gamma)}(N_e)_G\right)\to G$ be the map given by $\delta_G(x_v)=0$, $\delta_G(x_e)=x_e$ if $e=e_i$ for some $1\le i\le k$, and $\delta_G(x_e)=0$ otherwise.
\begin{notation} Denote by $\CE_G^1(\Gamma, A, j)$ and $\CE_G^2(\Gamma, A, j)$ the kernel and the cokernel of the map $\beta_G\times\gamma_G\times\delta_G$
\begin{equation}\label{eq: betaGxgammaGxdeltaG}
\left(\bigoplus_{v\in V^f(\Gamma)}N_G\right)\oplus\left(\bigoplus_{e\in E^b(\Gamma)}(N_e)_G\right)\to \left(\bigoplus_{e\in E^b(\Gamma)}N_G\right)\times\left(\bigoplus_{i=1}^k(N/L_i)_G\right)\times G\, .
\end{equation}
\end{notation}
\begin{Claim}\label{claim:EvsEJ}
There exists a natural exact sequence
$$0\to \CE_G^1(\Gamma, A, j)\to\CE_G^1(\Gamma, A)\to G\to \CE_G^2(\Gamma, A, j)\to\CE_G^2(\Gamma, A)\to 0\, .$$
\end{Claim}
\begin{proof} This sequence is nothing but the long exact sequence of cohomology associated to $0\to G[-1]\to \eqref{eq: betaGxgammaGxdeltaG}\to \eqref{eq:bGxgammaG}\to 0$.
\end{proof}
\begin{definition}
Let $\Gamma$ be an $N_\RR$-parameterized tropical curve of genus one satisfying an affine constraint $A$, and $G$ be an abelian group. The pair $(\Gamma, A)$ is called {\it elliptically $G$-regular} if $A$ is a simple constraint for $\Gamma$, and $\CE_G^2(\Gamma, A, j)=0$. Otherwise it is called {\it elliptically $G$-superabundant}.
\end{definition}

It is not difficult to check that $(\Gamma, A)$ is elliptically $\kk$-regular if and only if $(\Gamma, A)$ is $\kk$-regular, and the locus of tropical curves satisfying the constraint $A$ and having fixed tropical $j$-invariant $j(\Gamma)$ in the universal deformation space of $\Gamma$ has codimension $\codim(A) + 1$.

The proofs of the following two propositions and of the claim are identical to the proofs of Propositions  \ref{prop:E12ForSubdivisions}, \ref{prop:E1E2GammaGammaBar} and Claim~\ref{cl:e1k*}:
\begin{proposition}\label{prop:E12forSubdivisionsConstraintsJ}
Let $\Gamma$ and $\Gamma'$ be $N_\RR$-parameterized tropical curves, and assume that $\Gamma$ is obtained from $\Gamma'$ by subdivision of some of its edges. Pick an orientation on $\Gamma'$, and consider the induced orientation on $\Gamma$. Assume that $\Gamma'$ satisfies an affine constraint $A$. Then $\Gamma$ satisfies $A$, $\CE^2_G(\Gamma',A,j)\cong \CE^2_G(\Gamma,A,j)$, and the following sequence is exact:
$$0\to \bigoplus_{v\in V^f(\Gamma)\setminus V^f(\Gamma')}N_{e_v}\to \CE^1_G(\Gamma,A,j)\to \CE^1_G(\Gamma',A,j)\to 0,$$ where $e_v\in E(\Gamma')$ is the edge subdivided by the vertex $v$. In particular, if $\Gamma$ satisfies \eqref{stcond} and $\Gamma'=\Gamma^{\rm st}$ then $\bigoplus_{v\in V^f(\Gamma)\setminus V^f(\Gamma')}N_{e_v}=\bigoplus_{v\in V_2^f(\Gamma)}N_{e_v}.$
\end{proposition}
\begin{proposition}\label{prop:E1E2GammaGammaBarConstraintsJ} Let $\Gamma$ and $\overline{\Gamma}$ be as in Proposition \ref{prop:contOfTrCurve}.
Assume that $\Gamma$ satisfies a simple constraint $A$, and that $g(\Gamma)=g(\overline{\Gamma})=1$. Then $A$ is a simple constraint for $\overline{\Gamma}$, and the natural maps $\CE^1_G(\overline{\Gamma}, A,j)\to \CE^1_G(\Gamma, A,j)$ and $\CE^2_G(\overline{\Gamma},A,j)\to \CE^2_G(\Gamma,A,j)$ are isomorphisms.
\end{proposition}
\begin{Claim}\label{cl:e1jk*}
Assume that the pair $(\Gamma, A)$ is elliptically $\kk$-regular. Then $(\Gamma, A)$ is elliptically $\kk^*$-regular. If, in addition, $c(\Gamma)=0$ and $\codim (A)+1=\rank(\Gamma)$ then $|\CE^1_{\kk^*}(\Gamma, A, j)|=|\CE^2(\Gamma, A, j)|$.
\end{Claim}

\section{Tropical degenerations and tropical reductions.}\label{sec:trdegtrlim}

\subsection{$\Gamma^{\rm tr}$ and two fans.}
Let $\Gamma$ be an $N_\QQ$-parameterized $\QQ$-tropical curve.
\begin{definition} We define $\Sigma_{\Gamma, \eta}$ to be the fan in $N_\RR$ generated by the rays $\RR_+h_\Gamma(v)$ for $v\in V^\infty(\Gamma)$.
\end{definition}
Let us now define another fan associated to $\Gamma$. Consider the set $K_\Gamma$ consisting of the following cones in $N_\RR\oplus\RR$:
\begin{itemize}
    \item[(i)] the zero cone,
    \item[(ii)] $\rho_v=\Span_{\RR_+}\{(h_\Gamma(v), 1)\}$ for $v\in V^f(\Gamma)$,
    \item[(iii)] $\rho_v=\Span_{\RR_+}\{(h_\Gamma(v), 0)\}$ for $v\in V^\infty(\Gamma)$,
    \item[(iv)] $\sigma_e=\Span_{\RR_+}\{(h_\Gamma(v), 1), (h_\Gamma(v'), 1)\}$ for $e\in E^b_{vv'}(\Gamma)$,
    \item[(v)] $\sigma_e=\Span_{\RR_+}\{(h_\Gamma(v), 1), (h_\Gamma(v'), 0)\}$ for $e\in E^\infty_{vv'}(\Gamma)$ with $v\in V^f(\Gamma)$.
\end{itemize}
Note that $K_\Gamma$ is not necessarily a fan, since the intersection of two different cones in $K_\Gamma$ need not be a common face of these cones. However, after subdividing the cones in $K_\Gamma$, one gets a fan. The following claim is obvious:
\begin{Claim} Let $W$ be the set of rays consisting of the one-dimensional cones of the form $\sigma\cap\tau$ with $\sigma,\tau\in K_\Gamma$. Then there exists a unique fan $\Sigma_\Gamma$, such that $\Sigma_\Gamma^1=W$, and $|\Sigma_\Gamma|=\cup_{\sigma\in K_\Gamma}\sigma$.
\end{Claim}
\begin{proposition}\label{prop:Gammatr}
There exists a unique $N_\QQ$-parameterized $\QQ$-tropical curve $\Gamma^{\rm tr}$, obtained from $\Gamma$ by subdivision of edges with non-trivial slopes, such that $\Sigma_\Gamma=K_{\Gamma^{\rm tr}}$. In particular, $\Sigma_\Gamma=\Sigma_{\Gamma^{\rm tr}}$.
\end{proposition}
\begin{proof}
To construct $\Gamma^{\rm tr}$ we apply steps 1 and 2 of Algorithm~\ref{alg:trcurves} to the graph $\Gamma$: Let $e\in E_{vv'}(\Gamma)$ be an edge with a non-trivial slope, and $\sigma_e$ be the corresponding cone. If $\sigma_e\in \Sigma_\Gamma$ then we leave this edge as is, otherwise $\sigma_e$ is a union of two-dimensional cones in $\Sigma_\Gamma$. Let $\rho_1,\dotsc ,\rho_r\in\Sigma^1_\Gamma$ be the rays contained in the interior of $\sigma_e$, and let $(n_k,1)\in\rho_k$ be the corresponding vectors. Then we subdivide the edge $e$ with $r$ new vertices $v_1,\dotsc ,v_r\in e$, and set $h_{\Gamma^{\rm tr}}(v_i):=n_i$. The proposition now follows from Corollary~\ref{cor:lengthsfromweights}. The uniqueness of $\Gamma^{\rm tr}$ is obvious.
\end{proof}
\begin{remark}
To any edge $e\in E(\Gamma^{\rm tr})$ with $N_e\ne 0$ corresponds a cone $\sigma_e\in \Sigma^2_{\Gamma^{\rm tr}}$, to any finite vertex $v\in V^f(\Gamma^{\rm tr})$ corresponds a ray $\rho=\RR_+(h_\Gamma(v),1)\in \Sigma^1_{\Gamma^{\rm tr}}$, and to any infinite vertex $v\in V^\infty(\Gamma^{\rm tr})$ with $h_\Gamma(v)\ne 0$ corresponds a ray $\rho=\RR_+(h_\Gamma(v),0)\in \Sigma^1_{\Gamma^{\rm tr}}$. Each ray/cone corresponds to at least one vertex/edge, however several vertices/edges may correspond to the same ray/cone.

Note that there is a natural embedding $\Sigma^1_{\Gamma, \eta}\hookrightarrow\Sigma^1_\Gamma$, and we will often identify $\Sigma^1_{\Gamma, \eta}$ with its image in $\Sigma^1_\Gamma$.
\end{remark}
\begin{notation}\label{not:VrohEsigmanrho}
Let $\rho\in \Sigma^1_{\Gamma^{\rm tr}}$ be a ray, and $\sigma\in \Sigma^2_{\Gamma^{\rm tr}}$ be a two-dimensional cone. We denote by $V_\rho(\Gamma^{\rm tr})$ the set of vertices $v\in V(\Gamma^{\rm tr})$ that correspond to the ray $\rho$ and by $E_\sigma(\Gamma^{\rm tr})$ the set of edges corresponding to the cone $\sigma$. For $\rho\in \Sigma^1_{\Gamma^{\rm tr}, \eta}$ we denote by $n_\rho$ the primitive integral vector in the ray $\rho\subseteq N_\RR$, and for $\rho\in \Sigma^1_{\Gamma^{\rm tr}}\setminus\Sigma^1_{\Gamma^{\rm tr}, \eta}$, we denote by $n_\rho$ the unique rational vector such that $(n_\rho, 1)\in \rho$.
\end{notation}
\begin{definition}\label{def:lrholsigma} Let $\sigma\in \Sigma^2_{\Gamma^{\rm tr}}$ and $\rho\in \Sigma^1_{\Gamma^{\rm tr}, \eta}\subset \Sigma^1_{\Gamma^{\rm tr}}$ be cones. We define the multiplicities
$l(\sigma):=\lcm\{l(e)\,|\,e\in E_\sigma(\Gamma^{\rm tr})\}$ and
$l(\rho):=\lcm\{l(v)\,|\,v\in V_\rho(\Gamma^{\rm tr})\}$, where $l(e)$ and $l(v)$ are the multiplicities defined in Definition~\ref{def:proppartropcur}.
\end{definition}

\subsection{Toric varieties assigned to $N_\QQ$-parameterized $\QQ$-tropical curves.}

Pick a natural number $a\in \NN$. Let us equip $\RR$ with the lattice $a\ZZ$, and  $N_\RR\oplus\RR$ with the lattice $N\oplus a\ZZ$, and let $X(\Gamma, a)=X(\Gamma^{\rm tr}, a)$ be the toric variety over $\ZZ$ associated to the fan $\Sigma_\Gamma=\Sigma_{\Gamma^{\rm tr}}$.
The projection $N_\RR\oplus\RR\to \RR$ induces a natural map from the fan $\Sigma_\Gamma$ to the fan $\{0, \RR_+\}$. Hence the variety $X(\Gamma, a)$ admits a natural projection to $\AAA^1_a$, where the subscript $a$ indicates that the integral structure on $\RR$ is given by $a\ZZ$.
Let now $X(\Gamma, \eta)$ be the toric variety associated to the fan $\Sigma_{\Gamma, \eta}$. The following proposition is an easy exercise in toric geometry:
\begin{proposition}\label{prop:PropTorMod}
$ $

\begin{enumerate}
    \item The morphism $X(\Gamma, a)\to \AAA^1_a$ is flat.
    \item Let $\KK(\AAA^1_a)$ be the field of rational functions on $\AAA^1_a$. Then the general fiber of $X(\Gamma, a)\to \AAA^1_a$ is canonically isomorphic to $X(\Gamma, \eta)\times_{\Spec\ZZ}\Spec \KK(\AAA^1_a)$.
    \item The fiber over $0$ is reduced if and only if $an_\rho\in N$ for any ray $\rho\in\Sigma_\Gamma^1\setminus \Sigma_{\Gamma,\eta}^1$. Furthermore, if it is reduced then it is a union of irreducible components parameterized by the rays $\rho\in\Sigma_\Gamma^1\setminus \Sigma_{\Gamma,\eta}^1$, the component corresponding to $\rho$ is the closure $\overline{O}_\rho$ of the orbit $O_\rho$, and
        $$\overline{O}_{\rho_1}\cap \overline{O}_{\rho_2}=\left\{
                                    \begin{array}{ll}
                                      \overline{O}_\sigma, & \hbox{if $\sigma=\rho_1+\rho_2\in \Sigma_\Gamma$;} \\
                                      \emptyset, & \hbox{otherwise.}
                                    \end{array}
                                  \right.
        $$
    \item If $a$ is divisible by $a'$ then $X(\Gamma, a)$ is isomorphic to the normalization of the base change $X(\Gamma, a')\times_{\AAA^1_{a'}}\AAA^1_a$.
\end{enumerate}
\end{proposition}
\begin{remark}\label{rem:RedComp} Note first, that by the construction there is a distinguished rational function $t$ on $X(\Gamma, a)$ lifting the coordinate on $\AAA^1_a$. Recall that $\overline{O}_\sigma$ is isomorphic to the toric variety $X_{Star(\sigma)}$, and in our case (if the fiber over $0$ is reduced) they are {\it canonically isomorphic} thanks to the existence of $t$. If $\sigma\in \Sigma_\Gamma^2$ then $\overline{O}_\sigma=O_\sigma$, and if $\sigma=\rho\in\Sigma_\Gamma^1\setminus \Sigma_{\Gamma,\eta}^1$ then $Star(\sigma)$ is the fan in $N_\RR$ consisting of the zero cone and the following collection of rays: for each $\rho'\in\Sigma_\Gamma^1$ such that $\rho+\rho'\in\Sigma_\Gamma^2$ the cone $\RR_+n_{\rho'}$ if $\rho'\in\Sigma_{\Gamma,\eta}^1$ and the cone $\RR_+(n_{\rho'}-n_\rho)$ otherwise.
\end{remark}
\begin{notation}\label{not:tordegfortrcurve} Let $\Gamma$ be an $N_\QQ$-parameterized $\QQ$-tropical curve. Let $\LL/\FF$ be a finite field extension, and $t_\LL\in R_\LL$ be a uniformizer. Then there exists a unique morphism $\Spec R_\LL\to \AAA^1_{e_\LL}$ for which the pullback of the coordinate on $\AAA^1_{e_\LL}$ is the uniformizer $t_\LL$. We denote
$X_{R_\LL}(\Gamma):=X(\Gamma, e_\LL)\times_{\AAA^1_{e_\LL}}\Spec R_\LL$, $X_{\kk}(\Gamma):=X_{R_\LL}(\Gamma)\times_{\Spec R_\LL}\Spec\kk$, and $X_{\LL}(\Gamma):=X_{R_\LL}(\Gamma)\times_{\Spec R_\LL}\Spec\LL\cong X(\Gamma,\eta)\times_{\Spec\ZZ}\Spec\LL$.
\end{notation}

\subsection{Tropical degenerations and $\Gamma$-reductions.}
Let $C$ be a smooth complete curve over $\overline{\FF}$, and $f \colon  C\dashrightarrow T_{N,\overline{\FF}}$ be a rational map. For any point $p\in C$, let $n_p$ be the order of vanishing of $f^*(x^m)$ at $p$. Note that $n_p=0$ if and only if $f$ is defined at $p$, hence $n_p=0$ for all but finitely many points $p$.
\begin{Claim}\label{cl:extendsToD} Consider the toric variety $X\to \Spec \overline{\FF}$ associated to the following fan in $N_\RR$: the zero cone, and the collection of rays $\sigma_p:=\RR_+n_p$. Then the map $f$ extends to a morphism $f\colon C\to X$.
\end{Claim}
\begin{proof}
Since $X=\cup_p\Spec\overline{\FF}[\check{\sigma}_p\cap M]$, it is sufficient to show that for any $p$, the following holds: the functions $f^*(x^m)$ are regular at $p$ for all $m\in \check{\sigma}_p\cap M$. Note that $m\in \check{\sigma}_p\cap M$ if and only if $(n_p,m)\ge 0$. Hence $f^*(x^m)$ is regular at $p$ by the definition of $n_p$.
\end{proof}
\begin{remark} Let $D\subset C$ be the indeterminacy locus of $f$. Then, in terms of Notation~\ref{not:tordegfortrcurve}, $X\cong X_{\LL}(\Gamma^{\rm st}_{C,D,f})\times_{\Spec \LL}\Spec \overline{\FF}$.
\end{remark}

Assume that $(C,D)$ is stable. The goal of this section is to construct a distinguished integral model of $(C,D,f,X)$.
Set $\Gamma:=\Gamma^{\rm st}_{C,D,f}$, and let $\Gamma^{\rm tr}$ be the $N_\QQ$-parame\-terized $\QQ$-tropical curve associated to $\Gamma$ in Proposition~\ref{prop:Gammatr}. Then, by Proposition~\ref{prop:modelforgraph}, there exists a unique (up-to an isomorphism and a field extension) integral model $(C^{\rm tr}_{R_\LL}, D_{R_\LL})$ of $(C, D)$, whose associated $N_\QQ$-parameterized $\QQ$-tropical curve is $\Gamma^{\rm tr}$.
\begin{definition}
We say that $\LL$ is {\it sufficiently ramified} for $(C, D, f)$ if the stable model of $(C, D)$ is defined over $R_\LL$, and condition \eqref{eq:resonlen} is satisfied for any bounded edge $e\in E^b(\Gamma^{\rm tr})$, i.e., $e_\LL|e|\in\NN$.
\end{definition}
\begin{remark} Note that the model $C^{\rm tr}$ is defined over any sufficiently ramified field extension $\LL$ by Proposition~\ref{prop:modelforgraph}.
\end{remark}
\begin{proposition}
If $\LL$ is a sufficiently ramified extension for $(C, D, f)$ then the morphism $f_\LL\colon  C_\LL\to X_\LL$ extends to a stable map $f^{\rm tr}_{R_\LL}\colon C^{\rm tr}_{R_\LL}\to X_{R_\LL}(\Gamma^{\rm tr})$.
\end{proposition}
\begin{proof} First, note that the map $f_\LL$ extends to the generic points of the components of the reduction. Indeed, let $v\in V^f(\Gamma^{\rm tr})$ be a finite vertex, $C_v$ be the corresponding component of the reduction, and $\eta$ be its generic point. Then it is sufficient to check that $t_\LL^{e_\LL k}f_\LL^*(x^m)$ is regular at $\eta$ for all $(m,k)\in M\oplus\frac{1}{e_\LL}\ZZ$ satisfying $((m,k),(h_{\Gamma^{\rm tr}}(v),1))\ge 0$. But such $t_\LL^{e_\LL k}f_\LL^*(x^m)$ is regular at $\eta$ by the definition of $h_{\Gamma^{\rm tr}}(v)$. Moreover, the image of $\eta$ belongs to the open affine subset defined by the ray $\rho=\RR_+(h_{\Gamma^{\rm tr}}(v),1)$.

Second, if $p\in C_v$ is a {\em non-special point}, i.e., $p$ is different from $p_e$ for all edges $e$ with non-trivial slopes, then $f_\LL$ extends to $p$, since $C_{R_\LL}^{\rm tr}$ is normal and $f_\LL$ is defined in a punctured neighborhood of $p$.

Finally, it remains to prove that $f_\LL$ extends to the special points of the reduction. Let $p_e$ be the special point of the reduction corresponding to a bounded edge $e\in E_{vv'}$. Then the cone $\sigma_e$ is spanned by the rays $\rho,\rho'$ where $(e_\LL h_{\Gamma^{\rm tr}}(v),e_\LL)\in\rho$ and $(e_\LL h_{\Gamma^{\rm tr}}(v'),e_\LL)\in\rho'$ are primitive integral vectors (recall that the integral structure on the second factor of $N_\RR\oplus\RR$ is given by the lattice $e_\LL\ZZ$). In order to prove that $f_\LL$ extends to $p_e$, it is sufficient to show that $t_\LL^{e_\LL k}f_\LL^*(x^m)$ is regular at $p_e$ for all $(m,k)$ satisfying two inequalities: $((m,k),(h_{\Gamma^{\rm tr}}(v),1))\ge 0$ and $((m,k),(h_{\Gamma^{\rm tr}}(v'),1))\ge 0$. Note that such functions are regular in a punctured neighborhood of $p_e$, hence regular in codimension one. Since $C^{\rm tr}_{R_\LL}$ is normal at $p_e$ it follows that these functions are regular at $p_e$ as well. The case of special points corresponding to the unbounded edges is similar, and we leave it to the reader.
\end{proof}
\newpage
\begin{definition}\label{def:tropdegtropred}
$ $

\begin{enumerate}
  \item If $\Gamma=\Gamma^{\rm st}_{C,D,f}$ then $\Gamma^{\rm tr}$ is called the {\it $N_\QQ$-parameterized $\QQ$-tropi\-cal curve associated to the quadruple} $(C,D,f,X)$ and is denoted by $\Gamma^{\rm tr}_{C,D,f}$.
  \item The quadruple $(C^{\rm tr}_{R_\LL}, D_{R_\LL}, f^{\rm tr}_{R_\LL}, X_{R_\LL}(\Gamma^{\rm tr}))$ is called the {\it tropical degeneration} of $(C, D, f, X)$ over a sufficiently ramified extension $\LL$.
  \item The reduction $(C^{\rm tr}_\kk, D_\kk, f^{\rm tr}_\kk, X_\kk(\Gamma^{\rm tr}))$ of a tropical degeneration is called the {\it tropical reduction of $(C, D, f, X)$}.
\end{enumerate}
\end{definition}
\begin{remark}
$ $

\begin{enumerate}
\item The tropical degeneration depends on the choice of the uniformizer $t_\LL\in R_\LL$, while $C^{\rm tr}_\kk, D_\kk, X_\kk(\Gamma^{\rm tr})$, and the tropical curve associated to $(C, D, f, X)$ are independent of $t_\LL$ and of $\LL$. Note, however, that $f^{\rm tr}_\kk$ depends on $t_\LL$. In fact, for different choices of the uniformizer, $f^{\rm tr}_\kk$ differ by the action of a compatible family of $\chi_v\in T_{N,\kk}$, depending only on the residue class of the ratio of the uniformizers. Since each component of  $X_\kk(\Gamma^{\rm tr})$ is a toric variety (see Remark~\ref{rem:RedComp}), and in particular contains a distinguished point in the big orbit, we see that the tropical reduction depends on the uniformizer.
\item By Proposition \ref{prop:PropTorMod} and Remark \ref{rem:RedComp}, $X_\kk(\Gamma^{\rm tr})$ is a union of irreducible components $\overline{O}_\rho=X_{Star(\rho)}$ parameterized by the rays $\rho\in \Sigma^1_\Gamma\setminus\Sigma^1_{\Gamma, \eta}=\Sigma^1_{\Gamma^{\rm tr}}\setminus\Sigma^1_{\Gamma^{\rm tr}, \eta}$.
\item We will use the shorter notation $X_{R_\LL}^{\rm tr}:=X_{R_\LL}(\Gamma^{\rm tr})$ and $X_{\kk}^{\rm tr}:=X_{\kk}(\Gamma^{\rm tr})$, when no confusion is possible.
\end{enumerate}
\end{remark}

\begin{notation}
For $v\in V_\rho(\Gamma^{\rm tr})$, denote by $X_v$ the open subvariety of the component $\overline{O}_\rho\subseteq X_\kk^{\rm tr}$ defined by the subfan of $Star(\rho)$ consisting of the zero cone and of all rays of the form $\RR_+(h_{\Gamma^{\rm tr}}(v')-h_{\Gamma^{\rm tr}}(v))$ and $\RR_+h_{\Gamma^{\rm tr}}(v'')$, where $v'\in V^f(\Gamma^{\rm tr})$, $v''\in V^\infty(\Gamma^{\rm tr})$ for which $E_{vv'}(\Gamma^{\rm tr})\neq\emptyset\neq E_{vv''}(\Gamma^{\rm tr})$.
\end{notation}
\begin{remark}\label{rem:explicitformulaformap} It is easy to see that $f^{\rm tr}_{R_\LL}(C_v)\subset X_v$. Note that if $C_v\simeq \PP^1$ then one can describe the restriction of $f^{\rm tr}_{R_\LL}$ to $C_v$ explicitly. Indeed, let $y$ be a coordinate on $C_v$, $y_1,\dotsc,y_k\in C_v$ be the points of intersection of $C_v$ with other components of the reduction of $C$, $e_1,\dotsc,e_k\in E^b(\Gamma^{\rm tr})$ be the corresponding bounded edges, $y_{k+1},\dotsc,y_s$ be the specializations of the points of $D$ on the component $C_v$, and $v_1,\dotsc,v_s\in V(\Gamma^{\rm tr})$ be the corresponding vertices, i.e.  $y_i\in C_v\cap C_{v_i}$ for $1\le i\le k$ and $y_i$ is the specialization of $q_{v_i}$ for $i>k$. Then, since the pullback of $x^m$ to $C_v$ is invertible away from $y_1,\dotsc, y_s$, and has a zero of order $|e_i|^{-1}(h_{\Gamma^{\rm tr}}(v_i)-h_{\Gamma^{\rm tr}}(v), m)$ at $y_i$ for $i\le k$ and of order $(h_{\Gamma^{\rm tr}}(v_i), m)$ at $y_i$ for $i>k$, the restriction $f^{\rm tr}_{R_\LL}|_{C_v}$ is given by
\begin{equation}\label{eq:morphismRESTRICTEDtoCOMPONENTS}
\begin{array}{l}
  t_\LL^{-e_\LL(h_{\Gamma^{\rm tr}}(v),m)}\left(f^{\rm tr}_{R_\LL}|_{C_v}\right)^*(x^m)= \\
  \chi_v(m)\prod_{i=1}^k(y-y_i)^{|e_i|^{-1}\left(h_{\Gamma^{\rm tr}}(v_i)-h_{\Gamma^{\rm tr}}(v), m\right)}\prod_{i=k+1}^s(y-y_i)^{\left(h_{\Gamma^{\rm tr}}(v_i), m\right)}
\end{array}
\end{equation}
for some multiplicative character $\chi_v\colon M\to \kk^*$, which depends on the choice of the coordinate $y$.
\end{remark}
\begin{definition} Let $\Gamma$ be an $N_\QQ$-parameterized $\QQ$-tropical curve, $(C_\kk, D_\kk)$ be a semi-stable curve with marked points, and $f_\kk\colon C_\kk\to X_\kk(\Gamma^{\rm tr})$ be a morphism. Denote $X_\kk(\Gamma^{\rm tr})$ by $X_\kk^{\rm tr}$. Then the quadruple $(C_\kk, D_\kk, f_\kk, X_\kk^{\rm tr})$ is called a {\it $\Gamma$-reduction} if the following conditions are satisfied:
\begin{enumerate}
\item The set of irreducible components of $C_\kk$ is $\{C_v\}_{v\in V^f(\Gamma^{\rm tr})}$,
\item $D_\kk=\{p_e\}_{e\in E^\infty(\Gamma^{\rm tr})}$,
\item For any $e\in E^\infty(\Gamma^{\rm tr})$, $p_e\in C_v$ if and only if there exists $v'\in V^\infty(\Gamma^{\rm tr})$ such that $e\in E_{vv'}(\Gamma^{\rm tr})$,
\item $C_v\cap C_{v'}=\{p_e\}_{e\in E_{vv'}(\Gamma^{\rm tr})}$ for all $v,v'\in V^f(\Gamma^{\rm tr})$,
\item $f_\kk(C_v)\subset X_v$ for all $v\in V_\rho(\Gamma^{\rm tr})$, and
\item $(f_\kk|C_v)^*\partial X_v=\sum_{v'\in V(\Gamma^{\rm tr}),e\in E_{vv'}(\Gamma^{\rm tr})}l(e)p_e$ for all $v\in V^f(\Gamma^{\rm tr})$.
\end{enumerate}
If, in addition, all components of $C_\kk$ are rational then we say that $(C_\kk, D_\kk, f_\kk, X_\kk^{\rm tr})$ is a {\em Mumford $\Gamma$-reduction}.
\end{definition}
The last condition of the definition implies that if the reduction is Mumford, and $y$ is a coordinate on a component $C_v\subseteq C_\kk$, then the map $f_\kk|_{C_v}$ is given by formula \eqref{eq:morphismRESTRICTEDtoCOMPONENTS}.
\begin{Claim} The tropical reduction of $(C, D, f)$ is a $\Gamma^{\rm st}_{C, D, f}$-reduction. Furthermore, it is Mumford if and only if the curve $C$ is Mumford.
\end{Claim}
\begin{proof} Obvious.
\end{proof}
Next, we shall analyze the set of isomorphism classes of Mumford $\Gamma$-reductions for a given $N_\QQ$-paramete\-rized $\QQ$-tropical curve $\Gamma$.
\begin{proposition}\label{prop:ParamOfTrLimits} Let $\Gamma$ be an $N_\QQ$-parameterized $\QQ$-tropical curve satisfying \eqref{stcond} for which $c(\Gamma)=0$. Assume that $\EE^2_{\kk^*}(\Gamma)=1$. Then the set of isomorphism classes of Mumford $\Gamma$-reductions has a natural structure of a $\EE^1_{\kk^*}(\Gamma^{\rm st})$-torsor over the product $\prod_{v\in V^f(\Gamma^{\rm st})}\CM_{0,val(v)}$, where $\CM_{0,n}$ denotes the moduli space of smooth genus zero curves with $n$ marked points over the field $\kk$.
\end{proposition}
\begin{proof}
First, observe that $\EE^1_{\kk^*}(\Gamma^{\rm tr})$ acts naturally on the set of Mumford $\Gamma$-reductions. Indeed, let $(C_\kk, D_\kk, f_\kk, X_\kk^{\rm tr})$ be a Mumford $\Gamma$-reduction. Fix a coordinate on each component of $C_\kk$. Then the restriction $f_\kk|_{C_v}$ is given by a character $\chi_v$ (cf. \eqref{eq:morphismRESTRICTEDtoCOMPONENTS}), and the collection of characters $\chi=(\chi_v)$ must satisfy the following compatibility conditions at any $p_e\in C_v\cap C_{v'}$: $\chi_v\chi_{v'}^{-1}$ restricted to $N_e^\perp$ is a given character, which depends on the choice of the coordinates on $C_v$ and $C_{v'}$. Thus, for any element $\chi^0=(\chi_v^0)\in \EE^1_{\kk^*}(\Gamma^{\rm tr})$ the collection $\chi\chi^0=(\chi_v\chi_v^0)$ defines another morphism $\chi^0(f_\kk)\colon C_\kk\to X_\kk^{\rm tr}$, hence another Mumford $\Gamma$-reduction. Plainly, the action we have constructed is independent of the choice of the coordinates we made, and it induces an action on the isomorphism classes of Mumford $\Gamma$-reductions. Note that the latter action is transitive on the fibers of the natural projection to the product of the coarse moduli spaces $\prod_{v\in V^f(\Gamma)}\CM_{0,val(v)}$, and
$\bigoplus_{v\in V_2(\Gamma)}(N_{e_v})_{\kk^*}=\ker\big(\EE^1_{\kk^*}(\Gamma^{\rm tr})\to \EE^1_{\kk^*}(\Gamma^{\rm st})\big)$ is the kernel of this action. Hence, the induced action of $\EE^1_{\kk^*}(\Gamma^{\rm st})$ on the set of isomorphism classes of Mumford $\Gamma$-reductions is free, and the set of isomorphism classes of Mumford $\Gamma$-reductions is a torsor over its image under the natural projection to $\prod_{v\in V^f(\Gamma)}\CM_{0,val(v)}\cong\prod_{v\in V^f(\Gamma^{\rm st})}\CM_{0,val(v)}$. Finally, observe that if $\EE^2_{\kk^*}(\Gamma)=1$ then $\EE^2_{\kk^*}(\Gamma^{\rm tr})=1$ by Proposition~\ref{prop:E12ForSubdivisions}, and the projection to $\prod_{v\in V^f(\Gamma^{\rm st})}\CM_{0,val(v)}$ is surjective.
\end{proof}

\subsubsection{Tropical degenerations of toric constraints.}\label{subsec:tropdegofcons}

Let $L_i\subset N$, $1\le i\le k$, be sublattices of arbitrary  coranks greater than or equal to two, such that $N/L_i$ is torsion free for any $i$. Let $T_{L_i,\FF}\subset T_{N,\FF}$ be the corresponding subtori, $O=\{O_i\}_{i=1}^k$ be a set of $T_{L_i,\FF}$-orbits in $T_{N,\FF}$, and $A=A_O$ be the corresponding affine constraint. Consider an $N_\QQ$-parameterized $\QQ$-tropical curve $\Gamma$, and assume that $A$ is a simple constraint for $\Gamma$. Then $h_\Gamma(v_i)\in A_i$ for all $1\le i\le k$, where $v_i$ is the unique finite vertex connected to the $i$-th infinite vertex, which we denote by $v'_i$.

For any $1\le i\le k$, pick a point $r_i\in O_i$ such that the corresponding point in $A$ is precisely $h_\Gamma(v_i)$, i.e., the following equality holds: $\val(x^m(r_i))=(h_\Gamma(v_i),m)$ for all $m\in M$. Let $T_{L_i,\LL}\to T_{N,\LL}\subset X_\LL:=X_\LL(\Gamma)$ be the $T_{L_i,\LL}$-equivariant morphism sending $1\in T_{L_i,\LL}$ to $r_i$. Set $Y_\LL:=\coprod_{i=1}^k T_{L_i,\LL}$. Consider the map $Y_\LL\to T_{N,\LL}\subset X_\LL$, and let us construct a natural integral model $Y_{R_\LL}\hookrightarrow X_{R_\LL}(\Gamma)$ of $Y_\LL\hookrightarrow X_\LL$.
Let $\rho_i\in \Sigma^1_\Gamma$ be the ray corresponding to $v_i$, and $X_{\rho_i}\subset X_{R_\LL}(\Gamma)$ be the open subvariety defined by $\rho_i$. Note that $\val\left(t_\LL^{-(e_\LL h_\Gamma(v_i),m)}x^m(r_i)\right)=0$ for all $m\in M$ by the choice of $r_i$. Thus, the morphism $T_{L_i,\LL}\to T_{N,\LL}$ is given by $x^m\mapsto \chi_i(m)t_\LL^{(e_\LL h_\Gamma(v_i),m)}x^m$ for some character $\chi_i\colon M\to \LL^*$ satisfying $\val\circ\chi_i=0$. Hence the pullbacks of the regular functions on $X_{\rho_i}$ belong to $\CO(T_{L_i,R_\LL})$, and the map $T_{L_i,\LL}\to T_{N,\LL}$ extends to a morphism
$T_{L_i,R_\LL}\to X_{\rho_i}\subset X_{R_\LL}(\Gamma)$. Set $Y_{R_\LL}:=\coprod_{i=1}^k T_{L_i,R_\LL}$. Thus, by the construction, $Y_{R_\LL}$ is an integral model of $Y_\LL$ and the morphism  $Y_{R_\LL}\hookrightarrow X_{R_\LL}(\Gamma)$ extends $Y_\LL\hookrightarrow X_\LL$; we denote it by $g_{R_\LL}$.
\begin{definition} Morphism $g_{R_\LL}\colon Y_{R_\LL}\hookrightarrow X_{R_\LL}(\Gamma)$ is called the {\it tropical degeneration of toric constraint $O$ associated to $\Gamma$}, and the reduction $g_\kk\colon Y_\kk\hookrightarrow X_\kk(\Gamma)$ is called the {\it reduction of the toric constraint $O$ associated to $\Gamma$}.
\end{definition}
The reduction of the toric constraint can be written explicitly as $\coprod_{i=1}^k T_{L_i,\kk}\to T_{N,\kk}$, where each map  $T_{L_i,\kk}\to X_{\rho_i}\subset X_\kk$ is given by $x^m\mapsto \overline{\chi_i}(m)x^m$, and $\overline{\chi_i}$ is the composition of $\chi_i$ followed by the residue map $R_\LL\to \kk$ (recall that $\val\circ\chi_i=0$, hence $\chi_i(M)\subset R_\LL$).

\begin{definition}
Let $\Gamma$ be an $N_\QQ$-parameterized $\QQ$-tropical curve, $O$ be a toric constraint, and $(C_\kk, D_\kk, f_\kk, X_\kk^{\rm tr})$ be a $\Gamma$-reduction. $(C_\kk, D_\kk, f_\kk, X_\kk^{\rm tr})$ is called {\it $O$-constrained} if and only if $f_\kk(p_i)\in g_\kk(Y_\kk)$ for the first $k$ marked points $p_1,...,p_k\in D_\kk$.
\end{definition}
\begin{remark}
If $O$ is a toric constraint, and $(C,D,f)$ satisfies $O$, then the corresponding $\Gamma$-reduction is $O$-constrained.
\end{remark}

\begin{proposition}\label{prop:ParamOfTrLimitsConstr} In the above notation, if $c(\Gamma)=0$ and $\EE^2_{\kk^*}(\Gamma, A)=1$ then the set of isomorphism classes of $O$-constrained Mumford $\Gamma$-reductions has a natural structure of a $\EE^1_{\kk^*}(\Gamma^{\rm st}, A)$-torsor over $\prod_{v\in V^f(\Gamma^{\rm st})}\CM_{0,val(v)}$.
\end{proposition}
\begin{proof} The proof is identical to the proof of Proposition~\ref{prop:ParamOfTrLimits}.
\end{proof}

We conclude this section by explaining the motivation for the introduction of the stacky tropical degenerations and reductions. In the tropical approach to enumerative problems, one counts algebraic curves satisfying certain constraints, e.g., toric constraints, in terms of their tropical reductions; or, combinatorially, in terms of the corresponding $N_\QQ$-parameterized $\QQ$-tropical curves. To do so, one must be able to reconstruct uniquely the algebraic curve from its tropical reduction, or, equivalently, to reconstruct the integral model from the tropical reduction. In other words, one must solve the following deformation-theoretic problem:
Given a diagram of solid arrows, extend it to a commutative square:
\begin{equation}\label{diag:intmodeltotalcentral}
\xymatrix@!0{
& & D^k_{R_\LL} \ar@{-->}[rrr]\ar@{-->}'[d][dd]
& & & Y_{R_\LL} \ar[dd]
\\
D^k_\kk \ar@{^{(}-->}[urr]\ar[rrr]\ar[dd]
& & & Y_\kk \ar@{^{(}->}[urr]\ar[dd]
\\
& & C_{R_\LL} \ar@{-->}'[rr][rrr]
& & & X_{R_\LL}^{\rm tr}
\\
C_\kk \ar[rrr]\ar@{^{(}-->}[urr]
& & & X_\kk^{\rm tr} \ar@{^{(}->}[urr]
}
\end{equation}

One can compute the tangent and the obstruction spaces to  this deformation problem. In many cases, the deformation space has the expected dimension, e.g., the dimension is zero if one imposes the ``correct" number of constraints. However, usually, it is singular and obstructed! One of the reasons for this, is the non-trivial torsion in the normal sheaf of the map $f_\kk\colon C_\kk\to X_\kk^{\rm tr}$. In the next section we will equip tropical degenerations with a natural stacky structure, which will make the deformation space smooth and unobstructed in many cases. I shall mention that the idea, that by introducing an appropriate stacky structure, one can make the deformation space smooth and unobstructed I learned from Dan Abramovich.

\section{Toric stacks and stacky tropical degenerations.}\label{sec:toricstacksStackLimits}
In 2005, Borisov, Chen, and Smith introduced toric stacks \cite{BCS05}. Their construction gives rise to two kinds of stacks: smooth Deligne-Mumford stacks, if the fan is simplicial, and Artin stacks with infinite stabilizers otherwise. Below we introduce singular Deligne-Mumford stacks, or, more generally, tame Artin stacks with finite stabilizers. Our construction generalizes Borisov-Chen-Smith's construction, and produces the kind of toric stacks we need for the correspondence theorems.

\subsection{Toric stacks.}
\begin{definition}
Let $\Sigma$ be a fan in $N_\RR$. {\it Toric stacky data} is a collection $\Sigma'$ of sublattices $N'_\sigma\subseteq N_\sigma=N\cap \Span(\sigma)$ of maximal possible rank for all $\sigma\in\Sigma$, satisfying the following compatibility condition $N'_\sigma\cap \Span(\sigma\cap\tau)=N'_\tau\cap \Span(\sigma\cap\tau)$ for all $\sigma,\tau\in\Sigma$.
\end{definition}
For given toric stacky data $\Sigma'$, let us construct a tame Artin stack $\CX_{\Sigma'}$:
Let $\sigma\in\Sigma$ be a cone, and let $N'_\sigma\subset N_\sigma$ be the corresponding sublattice. Choose a sublattice $N'\subset N$ of full rank, such that $N'_\sigma=N'\cap \Span(\sigma)$. Then the sequences $0\to N'\to N\to N/N'\to 0$ and $0\to M\to M'\to M'/M\to 0$ are exact, where $M'=\Hom_\ZZ(N',\ZZ)$ is the dual lattice, and $N/N'$ and $M'/M$ are torsion groups. Thus,
$$1\to G_{N',N}\to T_{N'}\to T_N\to 1$$
is an exact sequence of algebraic groups, where $G_{N',N}\to \Spec\ZZ$ is a finite group-scheme.

Consider the affine toric varieties $X_\sigma=\Spec \ZZ[\check{\sigma}\cap M]$ and $X'_\sigma=\Spec \ZZ[\check{\sigma}\cap M']$. Then the natural map $X'_\sigma\to X_\sigma$ is invariant under the action of $G_{N',N}$. It is well known that $X_\sigma$ is the geometric quotient of $X'_\sigma$ by the action of $G_{N',N}$. Thus, $X_\sigma$ is the coarse moduli space of the quotient stack $[X'_\sigma/G_{N',N}]$.

Let $N''\subset N$ be another sublattice of full rank for which $N''\cap \Span(\sigma)=N'_\sigma$. If $N''\subseteq N'$ then there is a natural map
$X''_\sigma\to X'_\sigma$, and $X'_\sigma$ is the geometric quotient of $X''_\sigma$ by the action of the group $G_{N'',N'}$. Since $N'_\sigma=N'\cap \Span(\sigma)=N''\cap \Span(\sigma)$, it follows that $G_{N'',N'}$ acts freely on $X''_\sigma$. Hence
$X'_\sigma=[X''_\sigma/G_{N'',N'}]$, and the natural map $[X''_\sigma/G_{N'',N}]\to [X'_\sigma/G_{N',N}]$ is an isomorphism, which, by \cite[Lemma 4.2.3]{AV02}, has no non-trivial 2-automorphisms. We constructed a compatible system of isomorphisms. Note that the system of sublattices $N'\subset N$ of full rank for which $N'_\sigma=N'\cap \Span(\sigma)$, is partially ordered by embeddings, and any two elements are dominated by a third one. Thus, we can define the stack $\CX_\sigma$ to be $[X'_\sigma/G_{N',N}]$; and it is well defined up-to unique isomorphism.

Let $\tau$ be a face of $\sigma\in \Sigma$, and let $N'\subseteq N$ be a sublattice of full rank such that $N'_\sigma=N'\cap\Span(\sigma)$, and hence $N'\cap\Span(\tau)=N'_\tau$. Then there is a natural isomorphism $\alpha_{\sigma\tau}\colon \CX_\sigma\times_{X_\sigma}X_\tau \to \CX_\tau$, which, again, has no non-trivial 2-automorphisms. If $\sigma,\varrho\in\Sigma$ then we define $\alpha_{\sigma\varrho}\colon  \CX_\sigma\times_{X_\sigma}X_\tau\to \CX_\varrho\times_{X_\varrho}X_\tau$, where $\tau=\sigma\cap\varrho$, to be the composition $\alpha_{\varrho\tau}^{-1}\circ\alpha_{\sigma\tau}$.
Plainly, the collection $\{\alpha_{\sigma\varrho}\}$  satisfies the cocycle condition, and the 2-cocycle condition is empty by the construction. Thus, we can glue the stacks $\CX_\sigma$ together, and we obtain the desired stack $\CX_{\Sigma'}$.

Note that $\CX_{\Sigma'}$ is a normal separated tame Artin stack with coarse moduli space $X_\Sigma$. It is clear from the construction that torus $T_N$ acts on $\CX_{\Sigma'}$, and there is a one-to-one order reversing correspondence between the orbits of $T_N$ and the cones $\sigma\in\Sigma$.
\begin{Claim}\label{cl:Gsigma}
Orbit $\Theta_\sigma$ is isomorphic to $O_\sigma\times \CB G_\sigma$, where $G_\sigma=\Ker (T_{N'_\sigma}\to T_{N_\sigma})$.
\end{Claim}
\begin{proof}
First, note that there exists a sublattice $N'\subseteq N$ such that $N'_\sigma=N'\cap\Span(\sigma)$ and $N/N'=N_\sigma/N'_\sigma$. To construct such a lattice we use the fact that $N_\sigma\subseteq N$ splits, hence there exists $N^1_\sigma\subseteq N$ such that $N=N_\sigma\oplus N^1_\sigma$. Thus $N'=N'_\sigma+N^1_\sigma\subseteq N$ is the desired sublattice.
Since $\sigma^\perp\cap M'=\sigma^\perp\cap M$ and $G_{N',N}=G_\sigma$ acts trivially on $\Spec\kk[\sigma^\perp\cap M']$ it follows that
$\Theta_\sigma=\left[\Spec\ZZ[\sigma^\perp\cap M]/G_\sigma\right]\simeq O_\sigma\times\CB G_\sigma\, .$
\end{proof}
\begin{corollary}\label{cor:DMcriterion} $\CX_{\Sigma'}(R_\LL):=\CX_{\Sigma'}\times_{\Spec\ZZ}\Spec R_\LL$ is a Deligne-Mumford stack if and only if $char(\kk)$ does not divide $|N_\sigma/N'_\sigma|$ for all $\sigma\in\Sigma$.
\end{corollary}
\begin{definition}
Let $\CX$ be a normal separated Deligne-Mumford stack, and $\partial \CX\subset \CX$ be a divisor. A rational differential form $\omega$ on $\CX$ is called a {\it log-differential form} if it is regular on the complement of $\partial \CX$, and has at worst simple poles along $\partial \CX$, i.e., if $\kappa\colon U\to \CX$ is an \'etale covering and $D=\kappa^{-1}(\partial \CX)$ then $\kappa^*\omega$ has at worst simple pole along $D$. Log-differential forms form a sheaf on $\CX$ denoted by $\Omega_\CX\big(\log (\partial\CX)\big)$.
If $\CX=\CX_{\Sigma'}(R_\LL)$ then $\Omega_\CX\big(\log (\partial\CX)\big)$ denotes the sheaf of log-differential forms on $\CX$ with respect to $\partial\CX=\cup_{\rho\in\Sigma^1}\Theta_\rho\times_{\Spec\ZZ}\Spec R_\LL$.
\end{definition}

\begin{Claim}\label{cl:logtoric}
Let $\Sigma'$ be toric stacky data and $\CX_{\Sigma'}(R_\LL)$ be the corresponding toric stack. Assume that $\CX_{\Sigma'}(R_\LL)$ is a Deligne-Mumford stack. Then $\Omega_{\CX_{\Sigma'}(R_\LL)}\big(\log (\partial\CX_{\Sigma'}(R_\LL))\big)$ is canonically isomorphic to $M\otimes_\ZZ\CO_{\CX_{\Sigma'}(R_\LL)}$.
\end{Claim}
\begin{proof}
We claim that the map $\iota\colon m\otimes f\mapsto \frac{dx^m}{x^m}f$ is an isomorphism. It is sufficient to check this locally. Let $\sigma\in\Sigma$ be a cone, and let $N'\subseteq N$ be a sublattice of maximal rank such that $N'_\sigma=N'\cap \Span(\sigma)$ and $|N/N'|$ is not divisible by $char(\kk)$. Then the natural map $\kappa\colon X'_\sigma(R_\LL)=\Spec R_\LL[\check{\sigma}\cap M']\to \CX_\sigma(R_\LL)\subset\CX_{\Sigma'}(R_\LL)$ is an \'etale covering. Note that the natural embedding $M\subseteq M'$ induces an isomorphism $M\otimes_\ZZ R_\LL\to M'\otimes_\ZZ R_\LL$. Thus, $$\kappa^*\left(M\otimes_\ZZ\CO_{\CX_{\Sigma'}(R_\LL)}\right)=M\otimes_\ZZ\CO_{X'_\sigma(R_\LL)}\cong M'\otimes_\ZZ\CO_{X'_\sigma(R_\LL)},\:\:\: \text{and}$$ $$\kappa^*\Omega_{\CX_{\Sigma'}(R_\LL)}\big(\log (\partial\CX_{\Sigma'}(R_\LL))\big)=\Omega_{X'_\sigma(R_\LL)}\big(\log (\partial X'_\sigma(R_\LL))\big)\cong M'\otimes_\ZZ\CO_{X'_\sigma(R_\LL)}.$$ Hence $\kappa^*(\iota)$ is an isomorphism.
\end{proof}

\subsection{Stacky tropical degenerations and reductions.} Let $(C, D, f, X)$ be as in Definition~\ref{def:tropdegtropred}, $\LL$ be a sufficiently ramified extension, and $(f^{\rm tr}_{R_\LL}, C^{\rm tr}_{R_\LL}, D_{R_\LL}, X^{\rm tr}_{R_\LL})$ be the corresponding tropical degeneration. The goal of this subsection is to introduce natural stacky structures on $C^{\rm tr}_{R_\LL}$ and $X^{\rm tr}_{R_\LL}$.

Let us first, construct the stack $\CX^{\rm tr}_{R_\LL}$ with coarse moduli space $X^{\rm tr}_{R_\LL}$. Recall that for $\Gamma=\Gamma^{\rm tr}_{C,D,f}$, we constructed a fan $\Sigma_{\Gamma^{\rm tr}}$, and defined $X^{\rm tr}_{R_\LL}=X(\Gamma^{\rm tr}, e_\LL)\times_{\AAA_{e_\LL}^1}\Spec R_\LL$, where $X(\Gamma^{\rm tr}, e_\LL)$ is the toric variety assigned to the fan $\Sigma_{\Gamma^{\rm tr}}$ in $(N\oplus e_\LL\ZZ)_\RR$. Thus, to introduce the stacky structure on $X^{\rm tr}_{R_\LL}$, it is sufficient to specify stacky data on $\Sigma_{\Gamma^{\rm tr}}$, i.e., a compatible collection of sublattices $N'_\sigma\subseteq (N\oplus e_\LL\ZZ)_\sigma$ for $\sigma\in\Sigma_{\Gamma^{\rm tr}}$.

Let $\rho\in \Sigma^1_{\Gamma^{\rm tr}}$ be a ray. If $\rho\notin \Sigma^1_{\Gamma^{\rm tr},\eta}$ then set $N'_\rho:=(N\oplus e_\LL\ZZ)_\rho$, otherwise set $N'_\rho:=\ZZ\cdot(l(\rho)n_\rho, 0)$ (cf. Definition~\ref{def:lrholsigma} and Notation~\ref{not:VrohEsigmanrho}).

Let now $\sigma\in \Sigma^2_{\Gamma^{\rm tr}}$ be a two-dimensional cone, and $\rho_1, \rho_2$ be the facets of $\sigma$.
If one of them belongs to $\Sigma^1_{\Gamma^{\rm tr},\eta}$ then we set $N'_\sigma:=N'_{\rho_1}+N'_{\rho_2}$.  Otherwise, $\sigma$ is generated by vectors $(n_{\rho_i},1)\in\rho_i$, $i=1,2$. Let $n$ be the primitive integral vector in the direction of $n_{\rho_2}-n_{\rho_1}$. We define $N'_\sigma\subseteq (N\oplus e_\LL\ZZ)_\sigma$ to be the sublattice generated by $(e_\LL n_{\rho_1}, e_\LL)$ and $(l(\sigma)n, 0)$ (cf. Definition~\ref{def:lrholsigma}). Recall that the integral length of $e_\LL n_{\rho_2}-e_\LL n_{\rho_1}$ is divisible by $l(e)$ for all $e\in E_\sigma(\Gamma^{\rm tr})$ (cf. Remark~\ref{rem:divisib}). Thus, it is divisible by $l(\sigma)$. Hence $N'_\sigma\cap\rho_i=N'_{\rho_i}$.
We constructed stacky data $\Sigma'_{\Gamma^{\rm tr}}$, hence a toric stack $\CX_{\Sigma'_{\Gamma^{\rm tr}}}$, and we define $\CX_{R_\LL}(\Gamma^{\rm tr}):=\CX_{\Sigma'_{\Gamma^{\rm tr}}}\times_{\AAA_{e_\LL}^1}\Spec R_\LL$. As before, we shall use shorter notation $\CX_{R_\LL}^{\rm tr}=\CX_{R_\LL}(\Gamma^{\rm tr})$ if no confusion is possible.

\begin{remark}\label{rem:DMcriterion} The stacky structure on $\CX^{\rm tr}_{R_\LL}$ is concentrated over the intersections of the irreducible components of the reduction $X^{\rm tr}_{\kk}$, and along the closures of the boundary divisors of  the generic fiber $X^{\rm tr}_\LL$. It follows from Corollary~\ref{cor:DMcriterion} and Corollary~\ref{cor:l(e)fordifferentmodels} that $\CX^{\rm tr}_{R_\LL}$ is Deligne-Mumford if and only if \begin{equation}\label{cond:DM}
char(\kk)\nmid \prod_{e\in E(\Gamma^{\rm st}), N_e\ne 0}l(e).
\end{equation}
\end{remark}
\begin{convention} From now on we assume that $\CX^{\rm tr}_{R_\LL}$ is Deligne-Mumford.
\end{convention}
\begin{remark} Note that if we repeat the construction above, but for $\rho\in \Sigma^1_{\Gamma^{\rm tr},\eta}$ chose the sublattice generated by $(n_\rho,0)$ rather that $(l(\rho)n_\rho,0)$, then we will obtain a stack $\CX^{\rm tr'}_{R_\LL}$ with coarse moduli space $X^{\rm tr}_{R_\LL}$, and a natural map $\CX^{\rm tr}_{R_\LL}\to\CX^{\rm tr'}_{R_\LL}$ compatible with the projections to $X^{\rm tr}_{R_\LL}$. Furthermore, the stacky structure on $\CX^{\rm tr'}_{R_\LL}$ is concentrated on the intersections of the irreducible components of the reduction of $X^{\rm tr}_{R_\LL}\to \Spec R_\LL$ only, and $\CX^{\rm tr}_{R_\LL}$ is obtained from $\CX^{\rm tr'}_{R_\LL}$ by extracting the roots of order $l(\rho)$ along $\overline{O}_\rho$ for all $\rho\in \Sigma^1_{\Gamma^{\rm tr},\eta}$, i.e., $\CX^{\rm tr}_{R_\LL}=\big(\CX^{\rm tr'}_{R_\LL}\big)_{\DD,\vec{r}}$, where $\DD=\cup \overline{O}_\rho$ and $\vec{r}=(l(\rho))$ (see \cite{Cad07a} for the definition and the basic properties of the root stacks).
\end{remark}

Next, we define the stack $\CC^{\rm tr}_{R_\LL}$ with coarse moduli space  $C^{\rm tr}_{R_\LL}$, and a morphism $\varphi_{R_\LL}^{\rm tr}\colon\CC^{\rm tr}_{R_\LL}\to \CX^{\rm tr}_{R_\LL}$ lifting the morphism $f^{\rm tr}_{R_\LL}\colon C^{\rm tr}_{R_\LL}\to X^{\rm tr}_{R_\LL}$. We do it in two steps: First, we consider the twisted stable map $\psi_{R_\LL}^{\rm tr}\colon\CC^{\rm tr'}_{R_\LL}\to \CX^{\rm tr'}_{R_\LL}$ extending the stable map
$f_\LL\colon C_\LL\to X_\LL\subset \CX^{\rm tr'}_{R_\LL}$ (see \cite{AV02} for the definition and the properties of the twisted stable maps). Note that the coarse moduli space of $\CC^{\rm tr'}_{R_\LL}$ is $C^{\rm tr}_{R_\LL}$, and the stacky structure on $\CC^{\rm tr'}_{R_\LL}$ is concentrated at the nodes of the reduction of $C^{\rm tr}_{R_\LL}$. Second, observe that $\big(\psi_{R_\LL}^{\rm tr}\big)^*\overline{O}_\rho=\sum_{v\in V_\rho(\Gamma^{\rm tr})} l(v)q_v$ for any $\rho\in\Sigma^1_{\Gamma^{\rm tr},\eta}$. Thus, by \cite[Theorem~3.3.6]{Cad07a}, there exists a stack $\CC^{\rm tr}_{R_\LL}$, and a unique representable morphism $\varphi_{R_\LL}^{\rm tr}\colon\CC^{\rm tr}_{R_\LL}\to \CX^{\rm tr}_{R_\LL}$ lifting $\psi^{\rm tr}_{R_\LL}$, such that the coarse moduli space of $\CC^{\rm tr}_{R_\LL}$ is $C^{\rm tr}_{R_\LL}$. More explicitely, $\CC^{\rm tr}_{R_\LL}$ is the root stack $\big(\CC^{\rm tr'}_{R_\LL}\big)_{\DD_1,\vec{r}_1}$ for the divisor $\DD_1=\sum_{\rho\in\Sigma^1_{\Gamma^{\rm tr},\eta}, v\in V_\rho(\Gamma^{\rm tr})} q_v$ and the vector of multiplicities $\vec{r}_1=\big(l(\rho)/l(v)\big)_{\rho\in\Sigma^1_{\Gamma^{\rm tr},\eta}, v\in V_\rho(\Gamma^{\rm tr})}$. In fact, this is the {\it minimal} stacky structure on $C^{\rm tr}_{R_\LL}$ such that the map $f^{\rm tr}_{R_\LL}$ lifts to a map $\varphi^{\rm tr}_{R_\LL}\colon \CC^{\rm tr}_{R_\LL}\to \CX^{\rm tr}_{R_\LL}$. Moreover, we can describe it explicitly at each node and each marked point.
\begin{notation}\label{not:GeGv}
Let $e\in E_\sigma(\Gamma^{\rm tr})$ be an edge, and $v\in V_\rho(\Gamma^{\rm tr})$ be a vertex. We denote $G_e:=\Spec \ZZ[l(e)\ZZ/l(\sigma)\ZZ]$ and $G_v:=\Spec \ZZ[l(v)\ZZ/l(\rho)\ZZ]$.
\end{notation}
Let $\sigma\in \Sigma^2_{\Gamma^{\rm tr}}$ be a cone, $e\in E_\sigma(\Gamma^{\rm tr})$ be an edge, and $p_e$ be the corresponding node. Then, {\'e}tale locally at $p_e$, the scheme $C^{\rm tr}_{R_\LL}$ is given by $xy=t_\LL^{e_\LL|e|}=t_\LL^{r_e+1}$.
Recall that $l(e)(r_e+1)$ is the integral length of $e_\LL(h_{\Gamma^{\rm tr}}(v)-h_{\Gamma^{\rm tr}}(v'))$. Hence, Zariski locally at $f^{\rm tr}_{R_\LL}(p_e)$, $X^{\rm tr}_{R_\LL}$ is given by $XY=t_\LL^{l(e)(r_e+1)}$. Furthermore, locally $\CX^{\rm tr}_{R_\LL}=[X'_{R_\LL}/G_\sigma]$, where $X'_{R_\LL}$ is given by $X'Y'= t_\LL^{l(e)(r_e+1)/l(\sigma)}$, and $G_\sigma=\Spec \ZZ[\ZZ/l(\sigma)\ZZ]$ acts by $\xi\colon (X', Y', t_\LL)\mapsto(\xi X', \xi^{-1}Y', t_\LL)$. Consider the affine curve $C'$ given by the equation $x'y'=t_\LL^{l(e)(r_e+1)/l(\sigma)}$. Then $G_e\unlhd G_\sigma$ acts diagonally on $G_\sigma\times C'$, where the action on $C'$ is given by $\xi\colon(x', y', t_\LL)\mapsto(\xi x', \xi^{-1}y', t_\LL)$, and $\CC^{\rm tr}_{R_\LL}\times_{\CX_{R_\LL}}X'_{R_\LL}\simeq (G_\sigma\times C')/G_e$ \'etale locally. Finally, $\CC^{\rm tr}_{R_\LL}=G_\sigma\backslash(G_\sigma\times C')/G_e=[C'/G_e]$. In particular, if $l(\sigma)=l(e)$ then the stacky structure at $p_e$ is trivial. Similarly, one describes the stacky structure at the marked points. Indeed, if $\rho\in \Sigma^1_{\Gamma^{\rm tr},\eta}$ and $v\in V_\rho(\Gamma^{\rm tr})$ then, \'etale locally at $q_v$, the curve $C^{\rm tr}_{R_\LL}$ is isomorphic to $\AAA^1_{R_\LL}$. Consider the map $C'=\AAA^1_{R_\LL}\to \AAA^1_{R_\LL}\subset C^{\rm tr}_{R_\LL}$ given by $x\mapsto x^{l(\rho)/l(v)}$. Then $\CC^{\rm tr}_{R_\LL}$ is locally isomorphic to $[C'/G_v]$, where $G_v$ acts naturally on $C'$.
\begin{definition}
$ $

\begin{enumerate}
  \item The quadruple $(\CC^{\rm tr}_{R_\LL}, D_{R_\LL}, \varphi^{\rm tr}_{R_\LL}, \CX^{\rm tr}_{R_\LL})$ is called the {\it stacky tropical degeneration} of $(C, D, f, X)$ over a sufficiently ramified extension $\LL$.
  \item The reduction $(\CC^{\rm tr}_\kk, D_\kk, \varphi^{\rm tr}_\kk, \CX^{\rm tr}_\kk)$ of $(\CC^{\rm tr}_{R_\LL}, D_{R_\LL}, \varphi^{\rm tr}_{R_\LL}, \CX^{\rm tr}_{R_\LL})$ is called the {\it stacky tropical reduction} of $(C, D, f, X)$.
\end{enumerate}
\end{definition}

Recall that in Subsection~\ref{subsec:tropdegofcons} we constructed tropical degenerations of toric constraints $Y_{R_\LL}\to X_{R_\LL}$. Note that by the construction of $\CX_{R_\LL}$ and of $Y_{R_\LL}\to X_{R_\LL}$, the latter morphism lifts to the map $Y_{R_\LL}\to \CX_{R_\LL}$, which has no non-trivial 2-automorphisms by \cite[Lemma 4.2.3]{AV02}.
\begin{definition} Let $\Gamma$ be an $N_\QQ$-parameterized $\QQ$-tropical curve, $\varphi_\kk\colon\CC_\kk\to \CX_\kk^{\rm tr}$ be a representable morphism of Deligne-Mumford stacks, $f_\kk\colon C_\kk\to X_\kk^{\rm tr}$ be the corresponding morphism between their coarse moduli spaces, and $D_\kk\subset C_\kk$ be a divisor. The quadruple $(\CC_\kk, D_\kk, \varphi_\kk, \CX_\kk^{\rm tr})$ is called a {\it stacky $\Gamma$-reduction} if the following holds:
\begin{enumerate}
\item $\CX_\kk^{\rm tr}$ is the reduction of the stack $\CX^{\rm tr}_{R_\LL}$,
\item $(C_\kk,D_\kk,f_\kk,X_\kk^{\rm tr})$ is a $\Gamma$-reduction, and
\item $\varphi_\kk(\CC_v)$ is transversal to $\partial \CX_v$ for all $v\in V^f(\Gamma^{\rm tr})$.
\end{enumerate}
If, in addition, all components of $C_\kk$ are rational then we say that $(\CC_\kk, D_\kk, \varphi_\kk, \CX_\kk^{\rm tr})$ is a {\em Mumford stacky $\Gamma$-reduction}.
\end{definition}
\begin{Claim} The quadruple $(\CC^{\rm tr}_\kk, D_\kk, \varphi^{\rm tr}_\kk, \CX^{\rm tr}_\kk)$ is a stacky $\Gamma^{\rm st}_{C, D, f}$-reduction. Furthermore, it is Mumford if and only if the curve $C$ is Mumford.
\end{Claim}
\begin{proof} Obvious.
\end{proof}
\begin{proposition}\label{prop:ParamOfStakyTrLimits} Let $\Gamma$ be an $N_\QQ$-parameterized $\QQ$-tropical curve satisfying \eqref{stcond} and having $c(\Gamma)=0$. Let $(C_\kk,D_\kk,f_\kk,X_\kk^{\rm tr})$ be a $\Gamma$-reduction. Assume that $\EE^2_{\kk^*}(\Gamma)=1$. Then the number of isomorphism classes of stacky $\Gamma$-reductions $(\CC_\kk, D_\kk, \varphi_\kk, \CX_\kk^{\rm tr})$ with coarse moduli isomorphic to $(C_\kk,D_\kk,f_\kk,X_\kk^{\rm tr})$ is equal to
$\prod_{e\in E^b(\Gamma^{\rm st})} l(e)$.
\end{proposition}
\begin{proof}
The group of automorphisms $Aut$ of the $\Gamma$-reduction $(C_\kk,D_\kk,f_\kk,X_\kk^{\rm tr})$ is isomorphic to $\prod_{v\in V_2^f(\Gamma^{\rm tr})}\mu_{l(e_v)}$, where $e_v$ denotes any of the two edges containing $v$. Given a stacky $\Gamma$-reduction $(\CC_\kk, D_\kk, \varphi_\kk, \CX_\kk^{\rm tr})$ with coarse moduli $(C_\kk,D_\kk,f_\kk,X_\kk^{\rm tr})$, one obtains a family of maps $\CC_v\to \CX_\rho$ lifting the maps $C_v\to X_\rho$ for any ray $\rho$ and any vertex $v\in V_\rho$. Furthermore,  this family satisfies compatibility conditions on the intersections $\CC_v\cap\CC_{v'}$. Vice versa, a family of maps $\CC_v\to \CX_\rho$ lifting the maps $C_v\to X_\rho$ with identifications on the intersections defines a stacky $\Gamma$-reduction $(\CC_\kk, D_\kk, \varphi_\kk, \CX_\kk^{\rm tr})$ with coarse moduli $(C_\kk,D_\kk,f_\kk,X_\kk^{\rm tr})$. For any $\rho$, and any $v\in V_\rho$, the map $C_v\to X_\rho$ can be lifted to $\CC_v\to \CX_\rho$, and the lifting is unique since $\CC_v$ and $\CX_\rho$ are generically schemes, and $\CX_\rho$ is separated. Thus, the set of isomorphism classes of stacky $\Gamma$-reductions $(\CC_\kk, D_\kk, \varphi_\kk, \CX_\kk^{\rm tr})$ with coarse moduli isomorphic to $(C_\kk,D_\kk,f_\kk,X_\kk^{\rm tr})$ is in one-to-one correspondence with the product of the automorphism groups of the maps $\CC_e=\CC_v\cap\CC_{v'}\to \CX_\rho\cap\CX_{\rho'}=\CX_{\sigma_e}$ modulo the action of the group $Aut$. Finally, observe that $Aut(\CC_e\to \CX_{\sigma_e})=G_{\sigma_e}/G_e\cong \mu_{l(e)}$ and the group $Aut\cong\prod_{v\in V_2^f(\Gamma^{\rm tr})}\mu_{l(e_v)}$ acts on $\prod_{e\in E^b(\Gamma^{\rm tr})} Aut(\CC_e\to \CX_{\sigma_e})\cong\prod_{e\in E^b(\Gamma^{\rm tr})} \mu_{l(e)}$ via the diagonal embedding, i.e., if we fix an orientation on the bounded edges of $\Gamma^{\rm tr}$ then $\xi_v\colon \zeta_e\to\xi_v^{\epsilon(e,v)}\zeta_e$, where $\epsilon(e,v)=-1$ if $v$ is the initial point of $v$, $\epsilon(e,v)=1$ if $v$ is the target of $v$, and $\epsilon(e,v)=0$ otherwise. Hence, the number of isomorphism classes of stacky $\Gamma$-reductions $(\CC_\kk, D_\kk, \varphi_\kk, \CX_\kk^{\rm tr})$ with coarse moduli isomorphic to $(C_\kk,D_\kk,f_\kk,X_\kk^{\rm tr})$ is $$\left(\prod_{e\in E^b(\Gamma^{\rm tr})}l(e)\right)\div\left(\prod_{v\in V_2^f(\Gamma^{\rm tr})}l(e_v)\right)=\prod_{e\in E^b(\Gamma^{\rm st})}l(e)$$
\end{proof}

\begin{proposition}\label{prop:ParamOfConstrStakyTrLimits}
Let $\Gamma$ be an $N_\QQ$-parameterized $\QQ$-tropical curve satisfying \eqref{stcond} for which $c(\Gamma)=0$. Let $O$ be a toric constraint, and $A$ be the corresponding affine constraint. Assume that $\Gamma$ satisfies $A$, and that $A$ is a simple constraint for $\Gamma$. Let $(C_\kk,D_\kk,f_\kk,X_\kk^{\rm tr})$ be an $O$-constrained  $\Gamma$-reduction. If $\EE^2_{\kk^*}(\Gamma, A)=1$ then the number of isomorphism classes of $O$-constrained stacky $\Gamma$-reductions $(\CC_\kk, D_\kk, \varphi_\kk, \CX_\kk^{\rm tr})$ with the coarse moduli isomorphic to $(C_\kk,D_\kk,f_\kk,X_\kk^{\rm tr})$ is equal to
$\prod_{e\in E^b(\Gamma^{\rm st})} l(e)$.
\end{proposition}
\begin{proof} Identical to the proof of the previous proposition.
\end{proof}
\begin{remark}\label{rem:actiononstackylimits}
Fix an orientation of the bounded edges of $\Gamma^{\rm st}$. It induces an orientation on the bounded edges of $\Gamma$. Then, under the assumptions of the propositions, the set of isomorphism classes of (resp. $O$-constrained) Mumford stacky  $\Gamma$-reductions has a natural structure of a $\CE_{\kk^*}^1(\Gamma^{\rm st})$-torsor (resp. $\CE_{\kk^*}^1(\Gamma^{\rm st}, A)$-torsor) over $\prod_{v\in V^f(\Gamma^{\rm st})}\CM_{0,val(v)}$. The action of $\CE_{\kk^*}^1(\Gamma^{\rm st}, A)$ on the set of $O$-constrained stacky tropical reductions is defined similarly to the action of $\EE_{\kk^*}^1(\Gamma^{\rm st}, A)$ on the set of $O$-constrained tropical reductions (cf. Propositions \ref{prop:ParamOfTrLimits} and \ref{prop:ParamOfTrLimitsConstr}). Indeed, given an $O$-constrained stacky tropical reduction with the corresponding $O$-constrained tropical reduction $(C_\kk, D_\kk, f_\kk, X_\kk^{\rm tr})$, pick a coordinate on each component of $C_\kk$. Then the restriction of $f_\kk|_{C_v}$ is given by a character $\chi_v$ (cf. \eqref{eq:morphismRESTRICTEDtoCOMPONENTS}), and the set of characters must satisfy the following compatibility conditions at any $p_e\in C_v\cap C_{v'}$, $e\in E_{vv'}$: $\chi_v\chi_v^0|_{N_e^\perp}=\chi_{v'}\chi_{v'}^0|_{N_e^\perp}$ where $\chi_v^0$ and $\chi_{v'}^0$ are two fixed characters depending {\it only} on the choice of the coordinates on $C_v$ and $C_{v'}$. Assume that $v$ is the initial point of $e$. Then $\varphi_\kk$ determines, and is determined by the choice of $\chi_e\in (N_e)_{\kk^*}$ satisfying $\chi_e^{l(e)}=\frac{\chi_v\chi_v^0}{\chi_{v'}\chi_{v'}^0}$.
Now we can describe the action explicitly: For $$\xi=[(\xi_v),(\xi_e)]\in \CE_{\kk^*}^1(\Gamma^{\rm st}, A)\subseteq \left(\bigoplus_{v\in V^f(\Gamma)}N_{\kk^*}\right)\oplus\left(\bigoplus_{e\in E^b(\Gamma)}(N_e)_{\kk^*}\right)$$
we define $\xi(\CC_\kk, D_\kk, \varphi_\kk, \CX_\kk^{\rm tr})$ to be the stacky reduction defined by the collection $\left((\xi_v\chi_v)_{v\in V^f}, (\xi_e\chi_e)_{e\in E^b}\right)$. Plainly, this defines an action on the set of $O$-constrained stacky tropical reductions, and the action is independent of the choices we made. Furthermore, it is compatible with the action of $\EE_{\kk^*}^1(\Gamma^{\rm st}, A)$ on the set of $O$-constrained tropical reductions.
\end{remark}

\subsubsection{Tropical degenerations of elliptic constraint.}\label{subsec:tropdegofelcons}

Assume that $C$ and the corresponding tropical curve have genus one. We have seen in Subsection~\ref{subsec:ellconstr} that in this case $\val(j(C))=-j(\Gamma)<0$. Let $\LL$ be a sufficiently ramified extension, and fix a uniformizer $t_\LL\in R_\LL$. Then the residue class of $t_\LL^{-e_\LL\val(j(C))}j(C)$ is a well defined element of $\kk^*$, which we denote by $j_\kk(C)$. The goal of this subsection is to describe the set of isomorphism classes of stacky tropical reductions of triples $(C, D, f)$ satisfying an affine constraint and having given $\val(j(C))$ and $j_\kk(C)$. We start with necessary preparations.

Let $(C'_{R_\LL},D_{R_\LL})$ be a regular semi-stable model of $(C,D)$. Set $\Gamma':=\Gamma_{C'_{R_\LL},D_{R_\LL}}$. Let $v_1, v_2,\dotsc, v_k$ be the vertices of the cycle of minimal length generating the first homology of $\Gamma'$, and set $v_{k+1}:=v_1$. Let $e_i\in E_{v_i,v_{i+1}}$, $i=1,\dotsc, k$, be the edges of the cycle, and set $e_0:=e_k$. Then the reduction of $C'_{R_\LL}$ contains the cycle of projective lines $\cup_{i=1}^k C_{v_i}$. For any $1\le i\le k$, pick a coordinate $y_i$ on $C_{v_i}$, such that $y_i$ vanishes at $p_{e_i}$, and has a pole at $p_{e_{i-1}}$.

Consider the infinitesimal deformation of the reduction $C'_\kk=C'_{R_\LL}\times_{\Spec R_\LL}\Spec\kk$ to $\Spec R_\LL/(t_\LL^2)$ defined by $C'_{R_\LL/(t_\LL^2)}:=C'_{R_\LL}\times_{\Spec R_\LL} \Spec R_\LL/(t_\LL^2)$. Then there exists an exact sequence $0\to \CO_{C'_\kk}\to \CO_{C'_{R_\LL/(t_\LL^2)}}\to \CO_{C'_\kk}\to 0$, since $C'_{R_\LL/(t_\LL^2)}$ is flat over $\Spec R_\LL/(t_\LL^2)$. For any $1\le i\le k$, let $z_i\in \CO_{C'_\kk,p_{e_i}}$ and $w_{i+1}\in\CO_{C'_\kk,p_{e_i}}$ be the liftings of $y_i$ and $\frac{1}{y_{i+1}}$ respectively, such that $z_i$ vanishes on $C_{v_{i+1}}$ and $w_{i+1}$ vanishes on $C_{v_i}$. Pick arbitrary liftings of $z_i$ and $w_{i+1}$ to $\CO_{C'_{R_\LL/(t_\LL^2)},p_{e_i}}$, and denote them also by $z_i$ and $w_{i+1}$. Then, locally at $p_{e_i}$, $C'_{R_\LL/(t_\LL^2)}$ is given by $z_iw_{i+1}=t_\LL f_i$ for some $f_i\in \CO_{C'_\kk, p_{e_i}}$.

Four remarks are in place here: first, $f_i(p_{e_i})$ are independent of the choice of the liftings we made; second, $\prod_{i=1}^k f_i(p_{e_i})$ is independent of the choice of $y_i$, since if $\{y'_i\}$ is another set of coordinates with the same properties then $y'_i=\lambda_iy_i$ for some non-zero constants $\lambda_i\in\kk$, $f'_i=f_i\frac{\lambda_i}{\lambda_{i+1}}$, and $\prod_{i=1}^k f_i(p_{e_i})=\prod_{i=1}^k f'_i(p_{e_i})$; third, $\prod_{i=1}^k f_i(p_{e_i})\ne 0$ since $C'_{R_\LL}$ is regular; finally, if $C'_{R_\LL}$ was not regular, but would have a singularity of type $A_{r_i}$ at $p_{e_i}$, and would be given locally by $z_iw_{i+1}\equiv t_\LL^{r_i+1}f_i\mod t_\LL^{r_i+2}$, then $\prod_{i=1}^k f_i(p_{e_i})$ would give rise to the same value. To see this, one must check that the value of the product does not change when blowing up $p_e$; we leave the details to the reader.

\begin{lemma}\label{lem:leadtermofj}
$j_\kk(C)=\frac{1}{\prod_{i=1}^k f_i(p_{e_i})}$.
\end{lemma}
\begin{proof}
Set $t:=t_\LL$. Assume for simplicity that $char(\kk)\ne 2$. Then, without loss of generality, we may assume that $C$ is given by $y^2=X(X-1)(X-\lambda)$, where $\lambda\in \LL$ is such that $\val(\lambda)>0$ (cf. Subsection~\ref{subsec:ellconstr}). Set $x:=X-\frac{\lambda}{2}$. Then the equation can be rewritten as $x^2+y^2=\frac{\lambda^2}{4}+(x+\frac{\lambda}{2})(x^2-\frac{\lambda^2}{4})$. We will assume that $e_\LL\val(\lambda)>1$ since the case $e_\LL\val(\lambda)=1$ is easier and can be done along the same lines.

The reduction of the integral model of $C$ defined by the same equation has one nodal component, and the singularity of the total space is of type $A_{2e_\LL\val(\lambda)-1}$. To resolve the singularity of the total space, we will proceed with a sequence of $e_\LL\val(\lambda)$ blow ups with centers given by the ideals $\left(x,y,t\right), \left(\frac{x}{t}, \frac{y}{t}, t\right), \dotsc, \left(\frac{x}{t^{e_\LL\val(\lambda)-1}}, \frac{y}{t^{e_\LL\val(\lambda)-1}}, t\right)$. Each blow up increases the number of the components of the reduction by two, but the last blow up, which adds only one component. Thus, $k=2e_\LL\val(\lambda)$.

Let $v_1,\dotsc, v_{k+1}; e_0,\dotsc, e_k$ be as above. Without loss of generality, we may assume that $v_1$ corresponds to the strict transform of the component in the original reduction. Then $v_{1+i}$ and $v_{2e_\LL\val(\lambda)+1-i}$ correspond to the strict transforms of the components of the exceptional divisor of the $i$-th blow up if $0<i<e_\LL\val(\lambda)$, and $v_{1+e_\LL\val(\lambda)}$ corresponds to the exceptional divisor of the last blow up. The function $\frac{x}{y}$ has values $\pm\sqrt{-1}$ at the nodes $p_{e_1}$ and $p_{e_{k}}$; and without loss of generality, we may assume that the value at $p_{e_1}$ is $-\sqrt{-1}$.

Set $y_1:=\frac{x+\sqrt{-1}y}{x-\sqrt{-1}y}$, $y_{1+i}:=\frac{y}{t^i}$, and $y_{2e_\LL\val(\lambda)+1-i}:=\frac{t^i}{y}$ for $0<i<e_\LL\val(\lambda)$, and $y_{1+e_\LL\val(\lambda)}:=\frac{\alpha}{2}\frac{t^n}{(x+\sqrt{-1}y)}$. One can check by a straightforward calculation that first, for any $i$, $y_i$ is a coordinate on $C_{v_i}$ satisfying the properties required above; and second, if $\alpha$ denotes the residue class of $\lambda t^{-e_\LL\val(\lambda)}$ in the field $\kk$, then $f_1(p_{e_1})=\frac{-\sqrt{-1}}{2^2}$, $f_{2e_\LL\val(\lambda)}(p_{e_{2e_\LL\val(\lambda)}})=\frac{\sqrt{-1}}{2^2}$, $f_{e_\LL\val(\lambda)}(p_{e_{e_\LL\val(\lambda)}})=\frac{\alpha\sqrt{-1}}{2^2}$, $f_{e_\LL\val(\lambda)+1}(p_{e_{e_\LL\val(\lambda)+1}})=\frac{-\alpha\sqrt{-1}}{2^2}$, and $f_i(p_{e_i})=1$ for $i\ne 1, 2e_\LL\val(\lambda), e_\LL\val(\lambda), e_\LL\val(\lambda)+1$. Thus, $$j_\kk(C)=\frac{2^8}{\alpha^2}=\frac{1}{\prod_{i=1}^{2e_\LL\val(\lambda)}f_i(p_{e_i})},$$ since $j(C)=2^8\frac{(\lambda^2-\lambda+1)^3}{\lambda^2(\lambda-1)^2}$.
\end{proof}
Let now $(C_{R_\LL}^{\rm tr}, D_{R_\LL}, f_{R_\LL}^{\rm tr}, X_{R_\LL}^{\rm tr})$ and $(\CC_{R_\LL}^{\rm tr}, D_{R_\LL}, \varphi_{R_\LL}^{\rm tr}, \CX_{R_\LL}^{\rm tr})$ be the tropical and the stacky tropical degenerations of $(C,D,f,X)$. Fix a coordinate on each component of the reduction of $C_{R_\LL}$. Recall that the restriction of $f_\kk$ to $C_v$ is given by \eqref{eq:morphismRESTRICTEDtoCOMPONENTS}.
Pick $v,v'\in V^f$ and $e\in E_{vv'}$ with $N_e\ne 0$. Then $C_{R_\LL}$ has singularity of type $A_{e_\LL|e|-1}$ at $p_e$, and by \eqref{eq:morphismRESTRICTEDtoCOMPONENTS}, locally at $p_e$, the following equality holds on $C_{R_\LL}\times_{\Spec R_\LL}\Spec (R_\LL/t_\LL^{e_\LL|e|+1})$: $$t_\LL^{e_\LL(h_\Gamma(v),m)}\chi_v(m)\prod_{i=1}^{k_v}(y_v-y_i)^{|e_i|^{-1}(h_\Gamma(v_i)-h_\Gamma(v), m)}\prod_{i=k_v+1}^{s_v}(y_v-y_i)^{(h_\Gamma(v_i), m)}=$$ $$t_\LL^{e_\LL(h_\Gamma(v'),m)}\chi_{v'}(m)\prod_{i=1}^{k_{v'}}(y_{v'}-y'_i)^{|e'_i|^{-1}(h_\Gamma(v'_i)-h_\Gamma(v'), m)}\prod_{i=k_{v'}+1}^{s_{v'}}(y_{v'}-y'_i)^{(h_\Gamma(v'_i), m)}$$
Without loss of generality we may assume that $v_1=v'$, $v'_1=v$, $y_v$ vanishes at $p_e$, $y_{v'}$ has a pole at $p_e$, and, in a neighborhood of $p_e$, $C_{R_\LL}\times_{\Spec R_\LL}\Spec (R_\LL/t_\LL^{e_\LL|e|+1})$ is given by $\frac{y_v}{y_{v'}}\equiv a_et_\LL^{e_\LL|e|}\bmod t_\LL^{e_\LL|e|+1}$. Then, it follows from the equation above that
$a_e^{l(e)}=\frac{\chi_v\chi_v^0}{\chi_{v'}\chi_{v'}^0}\in (N_e)_{\kk^*}=\kk^*$ where
$$\chi_v^0(m)=\prod_{i=2}^{k_v}(-y_i)^{|e_i|^{-1}(h_\Gamma(v_i)-h_\Gamma(v), m)}\prod_{i=k_v+1}^{s_v}(-y_i)^{(h_\Gamma(v_i), m)}$$ and
$\chi_{v'}^0(m)=1$. Recall that the stacky tropical model determines an element $\chi_e$ satisfying $\chi_e^{l(e)}=\frac{\chi_v\chi_v^0}{\chi_{v'}\chi_{v'}^0}$ (cf. Remark~\ref{rem:actiononstackylimits}). By the construction, this element is nothing but $a_e$ above! Hence, $j_\kk(C)$ is completely determined by the stacky tropical reduction, and can be computed as $\prod_{e\in E^b}\chi_e^{-1}$ for an appropriate choice of coordinates on the components of the reduction. This leads to the following definition:
\begin{definition}
Let $\Gamma=\Gamma^{\rm tr}$ be an $N_\QQ$-parameterized $\QQ$-tropical curve of genus one, $v_1, v_2,\dotsc, v_k, v_{k+1}=v_1$ be the vertices of the cycle of minimal length generating the first homology of $\Gamma$, and $e_i\in E_{v_i,v_{i+1}}$, $i=1,\dotsc, k$, be the edges of the cycle. Assume that $N_{e_i}\ne 0$ for all $i$, and set $e_0:=e_k$. Let  $(\CC_\kk, D_\kk, \varphi_\kk, \CX_\kk^{\rm tr})$ be a Mumford stacky tropical $\Gamma$-reduction. For any $1\le i\le k$, pick a coordinate $y_i$ on $\CC_{v_i}$, such that $y_i$ vanishes at $p_{e_i}$ and has a pole at $p_{e_{i-1}}$. Then  $(\CC_\kk, D_\kk, \varphi_\kk, \CX_\kk^{\rm tr})$ defines, and is completely determined by the following data: $\left((\chi_v)_{v\in V^f}, (\chi_e)_{e\in E^b}\right)$ (cf. Remark~\ref{rem:actiononstackylimits}). We define $$j_\kk(\CC_\kk, D_\kk, \varphi_\kk, \CX_\kk^{\rm tr}):=\prod_{i=1}^k\chi_{e_i}.$$
\end{definition}
Note that $j_\kk(\CC_\kk, D_\kk, \varphi_\kk, \CX_\kk^{\rm tr})$ is well defined since if we choose different coordinates $y'_i=\lambda_iy_i$ then
$\prod_{i=1}^k\chi'_{e_i}=\prod_{i=1}^k(\frac{\lambda_i}{\lambda{i+1}}\chi_{e_i})=\prod_{i=1}^k\chi_{e_i}$.

\begin{proposition}\label{prop:leadingcoefofj}
Let $\Gamma$ be an $N_\QQ$-parameterized $\QQ$-tropical curve of genus one with $c(\Gamma)=0$ for which \eqref{stcond} is satisfied. Let $O$ be a toric constraint, and $A$ be the corresponding affine constraint. Assume that $\Gamma$ satisfies $A$, and  $A$ is a simple constraint for $\Gamma$.
If $\CE^2_{\kk^*}(\Gamma, A, j)=1$ then the set of isomorphism classes of $O$-constrained Mumford stacky  $\Gamma$-reductions with fixed $j_\kk(\CC_\kk, D_\kk, \varphi_\kk, \CX_\kk^{\rm tr})$ has a natural structure of a $\CE_{\kk^*}^1(\Gamma^{\rm st}, A, j)$-torsor over $\prod_{v\in V^f(\Gamma^{\rm st})}\CM_{0,val(v)}$.
\end{proposition}
\begin{proof}
If $\CE^2_{\kk^*}(\Gamma, A, j)=1$ then $\CE^2_{\kk^*}(\Gamma, A)=1$ and $\EE^2_{\kk^*}(\Gamma, A)=1$. Thus, the set of isomorphism classes of $O$-constrained Mumford stacky  $\Gamma$-reductions has a natural structure of a $\CE_{\kk^*}^1(\Gamma^{\rm st}, A)$-torsor over $\prod_{v\in V^f(\Gamma^{\rm st})}\CM_{0,val(v)}$. Furthermore, by the construction of the action (cf. Remark~\ref{rem:actiononstackylimits}) $j_\kk(\xi(\CC_\kk, D_\kk, \varphi_\kk, \CX_\kk^{\rm tr}))=j_\kk(\CC_\kk, D_\kk, \varphi_\kk, \CX_\kk^{\rm tr})$ if and only if $\xi\in \CE_{\kk^*}^1(\Gamma^{\rm st}, A, j)\subset \CE_{\kk^*}^1(\Gamma^{\rm st}, A)$. This implies the proposition.
\end{proof}

\section{The deformation theory.}\label{sec:deftheory}
\begin{convention} In this section all sheaves are considered as elements of derived categories of sheaves, and all functors are derived functors. In particular, we use short notation such as $f^*$, $f_*$, and $\RHom$ instead of $Lf^*$, $Rf_*$, and $\mathcal R\RHom$ respectively. All stacks in this section are Deligne-Mumford stacks, i.e. we assume that \eqref{cond:DM} holds.
\end{convention}
The reference for this section is the book of Illusie \cite{I71-72}, and we shall use Illusie's notation in this section. In particular, we use notation $L_{X/Y}$ for the cotangent complex associated to a morphism $X\to Y$.
The deformation problem we are going to deal with is the following:
\begin{equation}\label{defcube}
\xymatrix@!0{
& & D^k_{R_\LL} \ar@{-->}[rrrrr]\ar@{-->}'[d][ddd]
& & & & & Y_{R_\LL} \ar[ddd]
\\
D^k_\kk \ar@{^{(}-->}[urr]\ar[rrrrr]_{j_\kk}\ar[ddd]^{i_\kk}
& & & & & Y_\kk \ar@{^{(}->}[urr]\ar[ddd]_{g_\kk}
\\
\\
& & \CC_{R_\LL} \ar@{-->}'[rr][rrrrr]
& & & & & \CX_{R_\LL}
\\
\CC_\kk \ar[rrrrr]^{\varphi_\kk}\ar@{^{(}-->}[urr]
& & & & & \CX_\kk^{\rm tr} \ar@{^{(}->}[urr]
}
\end{equation}
In other words, we are given a stacky $\Gamma$-reduction $(\CC_\kk,D_\kk,\varphi_\kk,\CX_\kk^{\rm tr})$ satisfying a constraint $Y_\kk$, and we want to complete the corresponding diagram of solid arrows to a commutative diagram of dotted arrows; which we shall do order by order.

Recall that by \cite[p. 138]{I71-72}, for a pair of morphisms
$$\xymatrix{
X\ar[r]^f & Y\ar[r]^g & Z
}$$
there exists a distinguished triangle of cotangent complexes
\begin{equation}\label{eq:fundtriang}
\xymatrix@!0{
& & L_{X/Y} \ar@{>->}[lld]  & & \\
f^*L_{Y/Z}\ar[rrrr]& & & & L_{X/Z}\ar[ull]
}
\end{equation}
\begin{notation} For a scheme (stack) $Z$ over $\Spec\kk$, $L_Z$ denotes the cotangent complex $L_{Z/\Spec \kk}$.
\end{notation}
\begin{Claim}\label{cl:CotComAtSmPt} Let $Z$ be a stack over $\kk$, and $p\in Z$ be a smooth schematic $\kk$-point. Then $L_{p/Z}=T^*_pZ[1]$.
\end{Claim}
\begin{proof} The statement is local, thus we may assume that $Z$ is a smooth scheme.
Consider triangle \eqref{eq:fundtriang} for $p\to Z\to p$. Since $L_{p/p}=0$ and $L_{Z/p}=\Omega_Z$ is a vector bundle, we have
$L_{p/Z}=\left(L_{Z/p}\otimes_{\CO_Z}\CO_p\right)[1]=T^*_pZ[1]$.
\end{proof}

Consider distinguished triangles \eqref{eq:fundtriang} associated to the triples $\CC_\kk\to \CX_\kk^{\rm tr}\to p$, $Y_\kk\to \CX_\kk^{\rm tr}\to p$,
$D^k_\kk\to\CC_\kk\to\CX_\kk^{\rm tr}$, and $D^k_\kk\to Y_\kk\to\CX_\kk^{\rm tr}$
$$\xymatrix@!0{
& & L_{\CC_\kk/\CX_\kk^{\rm tr}} \ar@{>->}[lld]  & & \\
f_\kk^*L_{\CX_\kk^{\rm tr}}\ar[rrrr]& & & & L_{\CC_\kk}\ar[ull]
}\:\:\:\:
\xymatrix@!0{
& & L_{Y_\kk/\CX_\kk^{\rm tr}} \ar@{>->}[lld]  & & \\
g_\kk^*L_{\CX_\kk^{\rm tr}}\ar[rrrr]& & & & L_{Y_\kk}\ar[ull]
}$$
$$\xymatrix@!0{
& & & L_{D_\kk^k/\CC_\kk}\ar@{>->}[llld] & & & \\
i_\kk^*L_{\CC_\kk/\CX_\kk^{\rm tr}}\ar[rrrrrr]& & & & & & L_{D_\kk^k/\CX_\kk^{\rm tr}}\ar[ulll]
}$$
$$\xymatrix@!0{
& & & L_{D_\kk^k/Y_\kk}\ar@{>->}[llld]  & & &\\
j_\kk^*L_{Y_\kk/\CX_\kk^{\rm tr}}\ar[rrrrrr]& & & & & & L_{D_\kk^k/\CX_\kk^{\rm tr}}\ar[ulll]
}$$
Since any $q\in D^k_\kk$ is a smooth schematic point in $\CC_\kk, \CX_\kk^{\rm tr}$, and $Y_\kk$, the second pair of distinguished triangles can be rewritten as follows due to Claim~\ref{cl:CotComAtSmPt}:
$$
\xymatrix@!0{
& & & T_{D^k_\kk}^*\CC_\kk[1] \ar@{>->}[llld] & & & \\
i_\kk^*L_{\CC_\kk/\CX_\kk^{\rm tr}}\ar[rrrrrr]& & & & & & T_{D^k_\kk}^*\CX_\kk^{\rm tr}[1]\ar[ulll]
}
$$
\begin{equation}\label{eq:SecPairOfTriang}
\xymatrix@!0{
& & & T_{D^k_\kk}^*Y_\kk[1] \ar@{>->}[llld] & & & \\
j_\kk^*L_{Y_\kk/\CX_\kk^{\rm tr}}\ar[rrrrrr]& & & & & & T_{D^k_\kk}^*\CX_\kk^{\rm tr}[1]\ar[ulll]
}
\end{equation}
where $T_{D^k_\kk}^*Z=\bigoplus_{q\in D^k_\kk}T^*_qZ$ for $Z=Y_\kk,\CX_\kk^{\rm tr},\CC_\kk$.

Let us now return to deformation problem \eqref{defcube}.
By \cite[Th\'eor\`eme 2.1.7]{I71-72}, the deformation problem defined by the top square of \eqref{defcube} is unobstructed, and the set of small extensions is a torsor under the natural action of the group $$\EExt^1\bigl(L_{D^k_\kk/Y_\kk}, \CO_{D^k_\kk}\bigr)=T_{D^k_\kk}Y_\kk:=\bigoplus_{q\in D^k_\kk} T_qY_\kk\, .$$
By the same theorem, the obstructions to the deformation problem defined by the bottom square of \eqref{defcube} belong to $\EExt^2(L_{\CC_\kk/\CX_\kk^{\rm tr}}, \CO_{\CC_\kk})$, while the set of small extensions is either empty or forms a torsor under the action of the group
$\EExt^1(L_{\CC_\kk/\CX_\kk^{\rm tr}}, \CO_{\CC_\kk})$.
Consider the following commutative diagram of exact sequences assigned to distinguished triangles \eqref{eq:SecPairOfTriang}:
$$\xymatrix{
& 0\ar[d]& \\
& \EExt^0\bigl(i_\kk^*L_{\CC_\kk/\CX_\kk^{\rm tr}}, \CO_{D^k_\kk}\bigr)\ar[d]& \\
& T_{D^k_\kk}\CC_\kk\ar[d]^{di_\kk} & \\
T_{D^k_\kk}Y_\kk\ar@{^{(}->}[r]^{dj_\kk}\ar[dr]^h & T_{D^k_\kk}\CX_\kk^{\rm tr}\ar[r]\ar[d]& \EExt^1\bigl(j_\kk^*L_{Y_\kk/\CX_\kk^{\rm tr}}, \CO_{D^k_\kk}\bigr)\\
\EExt^1(L_{\CC_\kk/\CX_\kk^{\rm tr}}, \CO_{\CC_\kk})\ar[r]^{i_\kk^*}& \EExt^1\bigl(i_\kk^*L_{\CC_\kk/\CX_\kk^{\rm tr}}, \CO_{D^k_\kk}\bigr)\ar[d] &\\
& 0&\\
}$$
and assume that we are given a pair of small extensions of the top and the bottom squares of \eqref{defcube} defined by $(\xi, \zeta)\in T_{D^k_\kk}Y_\kk\oplus \EExt^1(L_{\CC_\kk/\CX_\kk^{\rm tr}}, \CO_{\CC_\kk})$. Then, by \cite[Proposition 2.2.4]{I71-72}, one can extend it to a small extension for deformation problem \eqref{defcube} if and only if $h(\xi)=i_\kk^*(\zeta)$, and the set of small extensions for given $(\xi, \zeta)$ is a torsor under the action of the group
$\ker(di_\kk)=\EExt^0\bigl(i_\kk^*L_{\CC_\kk/\CX_\kk^{\rm tr}}, \CO_{D^k_\kk}\bigr)$.
Note that if $dj_\kk\bigl(T_{D^k_\kk}Y_\kk\bigr)\cap di_\kk\bigl(T_{D^k_\kk}\CC_\kk\bigr)=0$ then $h$ is an embedding, and if, in addition, $\EExt^0\bigl(i_\kk^*L_{\CC_\kk/\CX_\kk^{\rm tr}}, \CO_{D^k_\kk}\bigr)=0$ then the set of small extensions for deformation problem \eqref{defcube} is either empty or forms a torsor under the action of the kernel of the map
\begin{equation}\label{eq:alpha}
\alpha\colon \EExt^1(L_{\CC_\kk/\CX_\kk^{\rm tr}}, \CO_{\CC_\kk})\to \dfrac{\EExt^1\bigl(i_\kk^*L_{\CC_\kk/\CX_\kk^{\rm tr}}, \CO_{D^k_\kk}\bigr)}{h\bigl(T_{D^k_\kk}Y_\kk\bigr)}\, .
\end{equation}
We can summarize the discussion above in the following proposition:
\begin{proposition}\label{prop:defexun}
Assume that $dj_\kk\bigl(T_{D^k_\kk}Y_\kk\bigr)\cap di_\kk\bigl(T_{D^k_\kk}\CC_\kk\bigr)=0$, and the map $di_\kk$ is injective.
Then the space of the first order deformations in the deformation problem \eqref{defcube} is given by $Def^1\eqref{defcube}=\ker(\alpha)$,
and the obstruction space $Ob\eqref{defcube}$ fits naturally into the exact sequence
$0\to {\rm coker}(\alpha)\to Ob\eqref{defcube}\to \EExt^2(L_{\CC_\kk/\CX_\kk^{\rm tr}}, \CO_{\CC_\kk})\to 0$. In particular, if $\alpha$ is surjective and $\EExt^2(L_{\CC_\kk/\CX_\kk^{\rm tr}}, \CO_{\CC_\kk})=0$ then the deformation space $Def\eqref{defcube}$ is smooth and unobstructed. If, in addition, $\alpha$ is an isomorphism then there exists a unique solution to the deformation problem \eqref{defcube}.
\end{proposition}

\subsection{Semi-simple computations.}
$ $\newline Let $\{\CX_\rho\}_{\rho\in\Sigma^1_{\Gamma^{\rm tr}}\setminus \Sigma^1_{\Gamma^{\rm tr},\eta}}$ be the set of irreducible components of $\CX_\kk^{\rm tr}$. Then each $\CX_\rho$ is a toric stack with coarse moduli space $\overline{O}_\rho$, whose orbit decomposition is given by $\CX_\rho=\coprod_{\rho\subset\sigma\in\Sigma^2_{\Gamma^{\rm tr}}}\CX_\sigma\coprod T_{N,\kk}$. Furthermore, $\CX_\rho\cap \CX_{\rho'}=\CX_\sigma$ if $\rho+\rho'=\sigma\in\Sigma^2_{\Gamma^{\rm tr}}$ and $\CX_\rho\cap \CX_{\rho'}=\emptyset$ otherwise.
Set $\CC_\rho:=\CC_\kk\times_{\CX_\kk^{\rm tr}}\CX_\rho$ and $\CC_\sigma:=\CC_\kk\times_{\CX_\kk^{\rm tr}}\CX_\sigma$, and denote by $\iota_\rho, \iota_\sigma$ the natural embeddings of $\CC_\rho$ and $\CC_\sigma$ into $\CC_\kk$. Then $\CC_\rho\cap \CC_{\rho'}=\CC_\sigma$ if $\rho+\rho'=\sigma\in\Sigma^2_{\Gamma^{\rm tr}}$ and $\CC_\rho\cap \CC_{\rho'}=\emptyset$ otherwise. We denote $\coprod_{\rho\subset\sigma\in\Sigma^2_{\Gamma^{\rm tr}}}\CX_\sigma$ by $\partial \CX_\rho$,  $\coprod_{\rho\subset\sigma\in\Sigma^2_{\Gamma^{\rm tr}}}\CC_\sigma$ by $\partial \CC_\rho$, and $\varphi^{\rm tr}_{R_\LL}|_{\CC_\rho}$ by $\varphi_\rho$.

\begin{lemma}\label{lem:componenttriangle}
Let $\CX_\rho$, $\CX_\sigma$, $\CC_\rho$, and $\CC_\sigma$ be as above. Then
\begin{enumerate}
\item There exists a distinguished triangle
$$\xymatrix{
& \bigoplus_{\sigma=\rho+\rho'}\RHom_{\CO_{\CC_\kk}}(L_{\CC_\kk/\CX_\kk^{\rm tr}}, (\iota_\sigma)_*\CO_{\CC_\sigma})\ar@{>->}[ld] \\
\RHom_{\CO_{\CC_\kk}}(L_{\CC_\kk/\CX_\kk^{\rm tr}}, \CO_{\CC_\kk})\ar[r] & \bigoplus_\rho\RHom_{\CO_{\CC_\kk}}(L_{\CC_\kk/\CX_\kk^{\rm tr}}, (\iota_\rho)_*\CO_{\CC_\rho})\ar[u]
}
$$
where in the upper sum $\rho,\rho'\in \Sigma^1_{\Gamma^{\rm tr}}\setminus \Sigma^1_{\Gamma^{\rm tr},\eta}$ and $\sigma\in\Sigma^2_{\Gamma^{\rm tr}}$, and in the lower sum $\rho\in \Sigma^1_{\Gamma^{\rm tr}}\setminus \Sigma^1_{\Gamma^{\rm tr},\eta}$.
\item There exists a natural quasi-isomorphism $$\RHom_{\CO_{\CC_\kk}}(L_{\CC_\kk/\CX_\kk^{\rm tr}}, (\iota_\rho)_*\CO_{\CC_\rho})\to (\iota_\rho)_*\RHom_{\CO_{\CC_\rho}}(L_{\CC_\rho/\CX_\rho}, \CO_{\CC_\rho}).$$
\item There exists a natural quasi-isomorphism $$\RHom_{\CO_{\CC_\kk}}(L_{\CC_\kk/\CX_\kk^{\rm tr}}, (\iota_\sigma)_*\CO_{\CC_\sigma})\to (\iota_\sigma)_*\RHom_{\CO_{\CC_\sigma}}(L_{\CC_\sigma/\CX_\sigma}, \CO_{\CC_\sigma}).$$
\end{enumerate}
\end{lemma}
\begin{proof}
By applying the derived functor $\RHom_{\CO_{\CC_\kk}}(L_{\CC_\kk/\CX_\kk^{\rm tr}}, \cdot)$ to the distinguished triangle
$$\xymatrix{
& \bigoplus_{\sigma=\rho+\rho'}(\iota_\sigma)_*\CO_{\CC_\sigma}\ar@{>->}[ld]  &  \\
\CO_{\CC_\kk}\ar[rr] & & \bigoplus_\rho(\iota_\rho)_*\CO_{\CC_\rho}\ar[ul]
}
$$
one proves the first part of the lemma.

The proofs of the second and the third parts of the lemma are similar, thus we prove the second statement, and leave the third statement to the reader. We shall first, prove that there is a natural quasi-isomorphism $\iota_\rho^*L_{\CC_\kk/\CX_\kk^{\rm tr}}\to L_{\CC_\rho/\CX_\rho}$. Consider the commutative diagram
$$
\xymatrix{
    \CC_\rho\ar[r]^{\iota_\rho}\ar[d] & \CC_\kk\ar[d]\\
    \CX_\rho\ar[r]                 & \CX_\kk^{\rm tr}
}
$$
It induces the natural map $\iota_\rho^*L_{\CC_\kk/\CX_\kk^{\rm tr}}\to L_{\CC_\rho/\CX_\rho}$, which is quasi-isomorphism if (a) $\Tor_q^{\CO_{\CX_\kk^{\rm tr}}}(\CO_{\CX_\rho},\CO_{\CC_\kk})=0$ for $q>0$ and (b) $L_{\CC_\rho/\CC_\kk\times_{\CX_\kk^{\rm tr}}\CX_\rho}=0$, by \cite[Corollary 2.2.3]{I71-72}. Plainly, (b) is satisfied since $\CC_\rho=\CC_\kk\times_{\CX_\kk^{\rm tr}}\CX_\rho$. To prove that (a) holds, observe that the problem is \'etale local. Thus, we may assume that $\CX_\kk^{\rm tr}=\Spec \kk[x,y,\underline{z}\,]/xy$ (here $\underline{z}$ is a multivariable) by the construction of $X_{R_\LL}^{\rm tr}$ and $\CX_{R_\LL}^{\rm tr}$. Furthermore, we may assume that $\CC_\kk=\Spec \kk[x,y]/xy$, $(\varphi^{\rm tr}_\kk)^*\colon \underline{z}\mapsto 0$, and $\CC_\rho$ and $\CX_\rho$ are given by $x=0$, since $\varphi_\kk(\CC_\rho)$ is transversal to $\partial\CX_\rho$. Thus,
$$\Tor_q^{\CO_{\CX_\kk^{\rm tr}}}(\CO_{\CX_\rho},\CO_{\CC_\kk})=\Tor_q^{\kk[x,y]/xy}(\kk[y],\kk[x,y]/xy)\otimes_\kk\kk[\underline{z}],$$ since $\kk[\underline{z}]$ is flat over $\kk$. Hence (a) holds. The second part of the lemma now follows, since $\iota_\rho^*$ and $(\iota_\rho)_*$ are adjoint functors.
\end{proof}

\begin{lemma}\label{lem:logtriang} There exists a distinguished triangle
$$\xymatrix{
& L_{\CC_\rho/\CX_\rho}\ar@{>->}[ld]  &  \\
\varphi_\rho^*\Omega_{\CX_\rho}\big(\log(\partial \CX_\rho)\big)\ar[rr] & & \Omega_{\CC_\rho}\big(\log(\partial \CC_\rho)\big)\ar[ul]
}
$$
\end{lemma}
\begin{proof}
The image $\varphi_\rho(\CC_\rho)$ is transversal to $\partial \CX_\rho$; hence, there exists a well defined map of log-differential forms $\varphi_\rho^*\Omega_{\CX_\rho}\big(\log(\partial \CX_\rho)\big)\to\Omega_{\CC_\rho}\big(\log (\partial \CC_\rho)\big)$, and the natural map
$\frac{\varphi_\rho^*\Omega_{\CX_\rho}\big(\log(\partial \CX_\rho)\big)}{\varphi_\rho^*\Omega_{\CX_\rho}}\to \frac{\Omega_{\CC_\rho}\big(\log (\partial \CC_\rho)\big)}{\Omega_{\CC_\rho}}$ is an isomorphism.
Thus, the complexes $\varphi_\rho^*\Omega_{\CX_\rho}\big(\log(\partial \CX_\rho)\big)\to\Omega_{\CC_\rho}\big(\log(\partial \CC_\rho)\big)$ and $\varphi_\rho^*\Omega_{\CX_\rho}\to \Omega_{\CC_\rho}$ are quasi-isomorphic. Note that $\CC_\rho$ is a complete intersection and $\CX_\rho$ is smooth, thus $L_{\CC_\rho}=\Omega_{\CC_\rho}$ and $L_{\CX_\rho}=\Omega_{\CX_\rho}$. Hence, there exists a distinguished triangle
$$\xymatrix{
& L_{\CC_\rho/\CX_\rho}\ar@{>->}[ld]  &  \\
\varphi_\rho^*\Omega_{\CX_\rho}\ar[rr] & & \Omega_{\CC_\rho}\ar[ul]
}
$$
and the lemma follows.
\end{proof}
\begin{lemma}\label{lem:componentcomputations} Assume that the coarse moduli space $C_\rho$ of $\CC_\rho$ is a  smooth rational curve, i.e., $C_\rho=\cup_{v\in V_\rho(\Gamma^{\rm tr})}C_v$, $C_v$ are rational for all $v\in V_\rho(\Gamma^{\rm tr})$, and no two vertices $v,v'\in V_\rho(\Gamma^{\rm tr})$ are connected by an edge. Then
\begin{enumerate}
\item $\EExt^i(L_{\CC_\rho/\CX_\rho}, \CO_{\CC_\rho})=0$ for $i\ne 1$.
\item
    \begin{enumerate}
    \item $\EExt^1(L_{\CC_v/\CX_\rho}, \CO_{\CC_v})=N_\kk$ for any $v\in V_\rho(\Gamma^{\rm tr})$ of valency three.
    \item $\EExt^1(L_{\CC_v/\CX_\rho}, \CO_{\CC_v})=(N/l(v)N_v)_\kk$ for any $v\in V_\rho(\Gamma^{\rm tr})$ of valency two, where $N_v$ and $l(v)$ denote the slope and the multiplicity of an edge containing $v$. By the balancing condition, the slope and the multiplicity are independent of the choice of the edge.
    \item $\EExt^1(L_{\CC_\rho/\CX_\rho}, \CO_{\CC_\rho})=\left(\bigoplus_{val(v)=2}(N/l(v)N_v)_\kk\right)\oplus\left(\bigoplus_{val(v)=3}N_\kk\right)$ if $\Gamma^{\rm tr}$ has no vertices of valency greater than three corresponding to the ray $\rho$.
    \item There is an exact sequence $$0\to N_\kk\to \EExt^1(L_{\CC_v/\CX_\rho}, \CO_{\CC_v})\to H^1\bigl(C_v, T_{C_v}\bigl(\log(\partial C_v)\bigr)\bigr)\to 0$$ for any vertex $v\in V_\rho(\Gamma^{\rm tr})$ with $val(v)>3$.
    \end{enumerate}
\item Let $\sigma=\rho+\rho'\in\Sigma^2_{\Gamma^{\rm tr}}$ be a cone, where $\rho,\rho'\in \Sigma^1_{\Gamma^{\rm tr}}\setminus \Sigma^1_{\Gamma^{\rm tr},\eta}$. Then
$$\EExt^i(L_{\CC_\sigma/\CX_\sigma}, \CO_{\CC_\sigma})=\left\{
                                                          \begin{array}{ll}
                                                            \bigoplus_{e\in E_\sigma(\Gamma^{\rm tr})}(N/N_e)_\kk, & \hbox{if i=1;}\\
                                                            0, & \hbox{otherwise.}
                                                          \end{array}
                                                        \right.$$
\end{enumerate}
\end{lemma}
To prove the Lemma, we need a tool for computing the $Ext$-s. Note that such computations can be reduced to the computations of cohomology of sheaves, which, in turn, can be computed using the following result of Abramovich and Vistoli \cite[Lemma~2.3.4]{AV02}: {\it Let $\pi\colon\CC\to C$ be the natural map between a tame stack $\CC$ and its coarse moduli space $C$. Then the functor $\pi_*$ is an exact functor between the categories of (quasi)coherent sheaves on $\CC$ and $C$.} Hence $H^*(\CC, \CF)=H^*(C, \pi_*\CF)$ for any (quasi)coherent sheaf $\CF$ on $\CC$.
\begin{proof} Pick $\rho$ and $v\in V_\rho(\Gamma^{\rm tr})$. First, note that $\varphi_\rho^*\Omega_{\CX_\rho}\big(\log(\partial \CX_\rho)\big)=M\otimes_\ZZ\CO_{\CC_\rho}$ by Claim~\ref{cl:logtoric}. Thus, $$\EExt^s(\varphi_\rho^*\Omega_{\CX_\rho}\big(\log(\partial \CX_\rho)\big),\CO_{\CC_v})=N\otimes_\ZZ H^s(\CC_v, \CO_{\CC_v})=\begin{cases}
                                                                N_\kk & \text{if $s=0$} \\
                                                                0 & \text{otherwise}
                                                             \end{cases}
$$
since $(\pi_\rho)_*\CO_{\CC_v}=\CO_{C_v}$ and $C_v\simeq\PP^1$.
Second, note that $$(\pi_\rho)_*\Omega_{\CC_v}\big(\log(\partial \CC_v)\big)=\Omega_{C_v}\bigl(\log(\partial C_v)\bigr).$$ Thus,
$$\EExt^s\bigl(\Omega_{\CC_v}\big(\log(\partial \CC_v)\big), \CO_{\CC_v}\bigr)=H^s\bigl(C_v, T_{C_v}\bigl(\log(\partial C_v)\bigr)\bigr),$$
which is equal to zero for $s\ne 0,1$, since $C_v\simeq\PP^1$. Finally, by Lemma \ref{lem:logtriang}, we conclude that there exists an exact sequence
\begin{equation*}
\begin{split}
0\to \EExt^0\bigl(L_{\CC_v/\CX_\rho}, \CO_{\CC_v}\bigr)\to & H^0\bigl(C_v, T_{C_v}\bigl(\log(\partial C_v)\bigr)\bigr)\to N_\kk\to \\
& \EExt^1\bigl(L_{\CC_v/\CX_\rho}, \CO_{\CC_v}\bigr)\to H^1\bigl(C_v, T_{C_v}\bigl(\log(\partial C_v)\bigr)\bigr)\to 0
\end{split}
\end{equation*}
and $\EExt^s\bigl(L_{\CC_v/\CX_\rho}, \CO_{\CC_v}\bigr)=0$ for $s\ne 0,1$. To finish the proof of the first part of the lemma it remains to show that $\EExt^0\bigl(L_{\CC_v/\CX_\rho}, \CO_{\CC_v}\bigr)=0$.

If $val(v)>2$ then $\deg(T_{C_v}\bigl(\log(\partial C_v)\bigr))<0$, hence $H^0\bigl(C_v, T_{C_v}\bigl(\log(\partial C_v)\bigr)\bigr)=0$, and we are done. If $val(v)=2$ then $h^0\bigl(C_v, T_{C_v}\bigl(\log(\partial C_v)\bigr)\bigr)=1$, and the statement follows from the second part of the Lemma.

(a) Assume that $val(v)=3$. Then $\deg \Omega_{C_v}\bigl(\log(\partial C_v)\bigr)=1$, which implies
$$\EExt^0\bigl(\Omega_{\CC_v}\bigl(\log(\partial \CC_v)\bigr),\CO_{\CC_v}\bigr)\simeq H^0(\PP^1, \CO_{\PP^1}(-1))=0.$$ Hence the map $N_\kk=\EExt^0\bigl(\varphi_\rho^*\Omega_{\CX_\rho}\bigl(\log(\partial \CX_\rho)\bigr),\CO_{\CC_v}\bigr)\to \EExt^{1}\bigl(L_{\CC_v/\CX_\rho}, \CO_{\CC_v}\bigr)$ is an isomorphism.

(b) Assume now that $val(v)=2$. Then $\EExt^0\bigl(\Omega_{\CC_v}\bigl(\log(\partial \CC_v)\bigr),\CO_{\CC_v}\bigr)$ is one-dimen\-sional since $\deg \Omega_{C_v}\bigl(\log(\partial C_v)\bigr)=0$. Thus, it remains to prove that its image in $\EExt^0\bigl(\varphi_\rho^*\Omega_{\CX_\rho}\bigl(\log(\partial \CX_\rho)\bigr),\CO_{\CC_v}\bigr)=N_\kk$ coincides with $(l(v)N_v)_\kk$. Let $v'$ be a finite vertex connected to $v$, and $e\in E_{vv'}(\Gamma^{\rm tr})$ be the edge. Note that $\CC_v\setminus \partial\CC_v=C_v\setminus \partial C_v\simeq\GG_m$. Fix such an isomorphism, and let $y$ be a coordinate on $\GG_m$ vanishing at $p_e\in C_v$. Then $\varphi_\rho|_{\GG_m}:\GG_m\to T_{N,\kk}\subset \CX_\rho$ is given by $(\varphi_\rho|_{\GG_m})^*x^m=\chi_{v,y}(m)y^{|e|^{-1}(h_{\Gamma^{\rm tr}}(v')-h_{\Gamma^{\rm tr}}(v),m)}$ for some character $\chi_{v,y}\colon M\to \kk^*$  (cf. Remark~\ref{rem:explicitformulaformap}). Thus, the map $$M\otimes_\ZZ\CO_{\CC_\rho}=f_\rho^*\Omega_{\CX_\rho}\bigl(\log(\partial \CX_\rho)\bigr)\to \Omega_{\CC_v}\bigl(\log(\partial \CC_v)\bigr)=\CO_{\CC_v}$$ is given by $m\mapsto |e|^{-1}(h_{\Gamma^{\rm tr}}(v')-h_{\Gamma^{\rm tr}}(v),m)$. By Corollary~\ref{cor:DMcriterion}, the integral length $l(e)$ of $|e|^{-1}(h_{\Gamma^{\rm tr}}(v')-h_{\Gamma^{\rm tr}}(v))$ is not divisible by $char(\kk)$, since $\CX_\rho$ is Deligne-Mumford. Thus, the map
$$\kk=\EExt^0\bigl(\Omega_{\CC_v}\bigl(\log(\partial \CC_v)\bigr),\CO_{\CC_v}\bigr)\to \EExt^0\bigl(\varphi_\rho^*\Omega_{\CX_\rho}\bigl(\log(\partial \CX_\rho)\bigr),\CO_{\CC_v}\bigr)=N_\kk$$
is given by $1\mapsto |e|^{-1}(h_{\Gamma^{\rm tr}}(v')-h_{\Gamma^{\rm tr}}(v))$ and its image coincides with the subspace $(l(v)N_v)_\kk\subset N_\kk$.

(c) Follows immediately from (a) and (b); and (d) follows from the vanishing $h^0\bigl(C_v, T_{C_v}\bigl(\log(\partial C_v)\bigr)\bigr)=0$.

For the third part of the lemma, note that in terms of Notation~\ref{not:GeGv} and Claim~\ref{cl:Gsigma}, $\CC_\sigma=\coprod_{e\in E_\sigma(\Gamma^{\rm tr})}\CB G_e(\kk)$ and $\CX_\sigma=T_{N/N_e}\times \CB G_\sigma(\kk)$, where $G_\bullet(\kk):=G_\bullet\times\Spec\kk$. Furthermore, $\CB G_e(\kk)\times_{\CX_\sigma}T_{N/N_e,\kk}=G_\sigma(\kk)/G_e(\kk),$ and its  image in $T_{N/N_e,\kk}$ is a point $u_e\in T_{N/N_e,\kk}$. Thus, by Claim~\ref{cl:CotComAtSmPt}, $\EExt^i(L_{\CC_\sigma/\CX_\sigma}, \CO_{\CC_\sigma})$ is given by
$$                                                      \left\{
                                                          \begin{array}{ll}
                                                            \bigoplus_{e\in E_\sigma(\Gamma^{\rm tr})}T_{u_e}T_{N/l(e)N_e,\kk}, & \hbox{if\:\:} i=1;\\
                                                            0, & \hbox{otherwise}
                                                          \end{array}
                                                        \right.\hspace{-0.1cm} =
                                                        \left\{
                                                          \begin{array}{ll}
                                                            \bigoplus_{e\in E_\sigma(\Gamma^{\rm tr})}(N/l(e)N_e)_\kk, & \hbox{if\:\:} i=1;\\
                                                            0, & \hbox{otherwise}
                                                          \end{array}
                                                        \right.$$
\end{proof}
\begin{corollary}\label{cor:computationofexts} Under the assumptions of Lemma \ref{lem:componentcomputations}, if $val(v)\le 3$ for any $v\in V_\rho(\Gamma^{\rm tr})$ then $\EExt^i(L_{\CC_\kk/\CX_\kk^{\rm tr}}, \CO_{\CC_\kk})=\CE^i_\kk(\Gamma^{\rm st})$ for $i=1,2$.
\end{corollary}
\begin{proof}
It follows from Lemmata \ref{lem:componenttriangle} and \ref{lem:componentcomputations} that there exists an exact sequence
$$\begin{array}{rlrrl}
    0\to&\hspace{-0.25cm}\EExt^1(L_{\CC_\kk/\CX_\kk^{\rm tr}}, \CO_{\CC_\kk})\to& & &\\
    \to&\hspace{-0.30cm}\left(\bigoplus_{v\in V_3^f(\Gamma^{\rm tr})}N_\kk\right)\oplus\left(\bigoplus_{v\in V_2^f(\Gamma^{\rm tr})}(N/l(v)N_v)_\kk\right)&\hspace{-0.25cm}\to \bigoplus_{e\in E^b(\Gamma^{\rm tr})}(N/l(e)N_e)_\kk&\hspace{-0.7cm}\to &\\
    & &\hspace{-0.25cm}\to \EExt^2(L_{\CC_\kk/\CX_\kk^{\rm tr}}, \CO_{\CC_\kk})&\hspace{-0.3cm}\to &\hspace{-0.25cm}0
  \end{array}$$
where the central map is given by $x_v\mapsto \sum_{e\in E^b}(\epsilon(e,v)x_v\bmod N_e)$, and, as before, $\epsilon(e,v)=-1$ if $v$ is the initial point of $e$, $\epsilon(e,v)=1$ if $v$ is the target of $e$, and $\epsilon(e,v)=0$ otherwise. Consider the natural injective map of complexes
$$
\xymatrix{
0\ar[d]\ar[r] & 0\ar[d]\\
\bigoplus_{v\in V^f(\Gamma^{\rm st})}N_\kk\ar[d]\ar[r] & \left(\bigoplus_{v\in V_3^f(\Gamma^{\rm tr})}N_\kk\right)\oplus\left(\bigoplus_{v\in V_2^f(\Gamma^{\rm tr})}(N/l(v)N_v)_\kk\right)\ar[d]\\
\bigoplus_{e\in E^b(\Gamma^{\rm st})}(N/l(e)N_e)_\kk\ar[d]\ar[r] &  \bigoplus_{e\in E^b(\Gamma^{\rm tr})}(N/l(e)N_e)_\kk\ar[d]\\
0\ar[r]& 0
}
$$
Plainly, its cokernel is quasi-isomorphic to zero. Thus, the map itself is a quasi-isomorphism. Recall, that the cohomology of the left column is $\CE^\bullet_\kk(\Gamma^{\rm st})$ (cf. Remark~\ref{rem:anotherformcomplex}). Hence, $\EExt^i(L_{\CC_\kk/\CX_\kk^{\rm tr}}, \CO_{\CC_\kk})=\CE^i_\kk(\Gamma^{\rm st})$ for $i=1,2$.
\end{proof}

\begin{corollary}\label{cor:computationofconstrainedexts} Let $O$ be a toric constraint, and $A$ be the corresponding affine constraint. Assume that $\Gamma^{\rm tr}$ satisfies $A$, and $A$ is a simple constraint for $\Gamma^{\rm tr}$. Then, under the assumptions of Corollary~\ref{cor:computationofexts},
$Def^1\eqref{defcube}=\CE^1_\kk(\Gamma^{\rm st}, A)$ and $Ob\eqref{defcube}=\CE^2_\kk(\Gamma^{\rm st}, A)$.
\end{corollary}
\begin{proof}
Pick $l\le k$. Let $v\in V^f(\Gamma^{\rm tr})$ be the unique vertex connected to $v_{q_l}$, and $\rho$ be the corresponding ray in $\Sigma^1_{\Gamma^{\rm tr}}$. Since $v$ has valency three it follows from the balancing condition that it is connected to two finite vertices, and the slopes of the bounded edges containing $v$ coincide. Let us denote them by $N_v$. Consider the map $\CC_v\to \CX_\rho$ and its restriction $\GG_m\to T_N\subset \CX_\rho$. As we have seen above, it is given by $x^m\mapsto\chi_{v,y}(m)y^{(n,m)}$, where $y$ is a coordinate on $\GG_m$ and $n\in N_v$. Furthermore, the integral length of $n$ is not divisible by the characteristic of $\kk$. Thus, in the notation of Proposition~\ref{prop:defexun}, we have $di_\kk(T_{q_l}\CC_\kk)=(N_v)_\kk$, hence $di_\kk$ is injective since the slope $N_v\ne 0$. Moreover, $di_\kk(T_{q_l}\CC_\kk)\cap dj_\kk(T_{q_l}(Y_\kk))=0$ since $A$ is a simple constraint for $\Gamma^{\rm tr}$. Then,  by Proposition~\ref{prop:defexun}, $Def^1\eqref{defcube}=\ker(\alpha)$ and the obstruction space $Ob\eqref{defcube}$ fits naturally into the exact sequence $0\to {\rm coker}(\alpha)\to Ob\eqref{defcube}\to \EExt^2(L_{\CC_\kk/\CX_\kk^{\rm tr}}, \CO_{\CC_\kk})\to 0$. Observe, that in our case the right-hand side of \eqref{eq:alpha} is just $\bigoplus_{i=1}^k (N/L_i)_\kk$, and, by Corollary~\ref{cor:computationofexts}, $\EExt^i(L_{\CC_\kk/\CX_\kk^{\rm tr}}, \CO_{\CC_\kk})=\CE^i_\kk(\Gamma^{\rm st})$. The corollary now follows from Proposition~\ref{prop:E1E2GammaAGamma}.
\end{proof}

\begin{remark}\label{rem:representabilityOfregularcurves}
If in Lemma~\ref{lem:componentcomputations}, one assumes only that $C_\rho$ is rational (but, probably, singular) for any $\rho$ then, by repeating the same computations we did in Lemma~\ref{lem:componentcomputations} and Corollaries~\ref{cor:computationofexts}-\ref{cor:computationofconstrainedexts}, one can show that there exists an exact sequence
$$
\begin{array}{lll}
  0\to \CE^1(\overline{\Gamma}^{\rm st}, A) &\hspace{-0.25cm}  \to & \hspace{-0.25cm} \EE^1(L_{\CC_\kk/\CX_\kk},\CO_{\CC_\kk})\to \\
   & \hspace{-0.25cm} \to &\hspace{-0.25cm}  \bigoplus_{e\in E^b(\Gamma), N_e=0}\kk\to\CE^2(\overline{\Gamma}^{\rm st}, A)\to \EE^2(L_{\CC_\kk/\CX_\kk},\CO_{\CC_\kk})\to 0,
\end{array}
$$
where $\overline{\Gamma}$ denotes the tropical curve obtained from $\Gamma$ by contraction of the maximal connected subgraphs of finite vertices connected by edges with trivial slopes (cf. Proposition~\ref{prop:contOfTrCurve}). In particular, if $(\overline{\Gamma}^{\rm st},A)$ is $\kk$-regular then it is representable.
\end{remark}

Assume now, that $\Gamma^{\rm tr}$ has genus one, $val(v)\le 3$ for all $v\in V(\Gamma^{\rm tr})$, $\CC_\kk$ has only rational components, and no bounded edge of $\Gamma^{\rm tr}$ has trivial slope.
Consider the stabilizations $\Gamma^{\rm st}$ and $\CC^{\rm st}_\kk$ of $\Gamma^{\rm tr}$ and $\CC_\kk$. Denote by $v_1,\dotsc, v_k, v_{k+1}=v_1$ and $e_i\in E_{v_i,v_{i+1}}$ the vertices and the edges in the cycle of minimal length generating the first homology of $\Gamma^{\rm st}$. Consider the following  deformation problem:
\begin{equation}\label{defcurve}
\xymatrix@!0{
(\CC^{\rm st}_\kk, D_\kk)\ar@{-->}[rrr]\ar[dd] & & & (\CC_{R_\LL}, D_{R_\LL})\ar@{-->}[dd]\\
& & & \\
\Spec\kk\ar[rrr] & & & \Spec{R_\LL}
}
\end{equation}

Since the components of $\CC^{\rm st}_\kk$ are smooth, and each component contains at most three special points, the deformations of $(\CC^{\rm st}_\kk, D_\kk)$ are induced from the deformations of the nodes. Thus, the solutions of \eqref{defcurve} are given by the $R_\LL$-homomorphisms $R_\LL[[x_e]]_{e\in E^b(\Gamma^{\rm st})}\to R_\LL$ mapping the coordinate functions to the maximal ideal $(t_\LL)$.

Let us blow up $\Spec\left(R_\LL[[x_e]]_{e\in E^b(\Gamma^{\rm st})}\right)$ along $x_e=t_\LL=0$ for any $e\in E^b(\Gamma^{\rm st})$, and proceed with $\sum (e_\LL|e|-1)$ more blow ups till we get the chart with the coordinates $(t_\LL^{-e_\LL|e|}x_e,t_\LL)_{e\in E^b(\Gamma^{\rm st})}$, which we denote by $\widetilde{Def}\eqref{defcurve}$. By the construction, the projection $Def\eqref{defcube}\to Def\eqref{defcurve}$ lifts to a map $\pi_1\colon Def\eqref{defcube}\to \widetilde{Def}\eqref{defcurve}$. Let $\widetilde{Def}\eqref{defcurve}\to \overline{\CM}_{1,1}\times_{\Spec\ZZ}\Spec R_\LL$ be a projection forgetting all but one marked point, and contracting the unstable components. Its image belongs to the chart $\AAA^1_{R_\LL}$, where the origin corresponds to the infinity of $\overline{\CM}_{1,1}$, in other words the coordinate on $\AAA^1_{R_\LL}$ is $\frac{1}{j}$.
Let us blow up $\AAA^1_{R_\LL}$ along $\frac{1}{j}=t_\LL=0$ and proceed with $r-1=\sum_{i=1}^ke_\LL|e|-1$  more blow ups till we get a chart with coordinates $(\frac{1}{jt_\LL^r}, t_\LL)$. Denote this chart by $\widetilde{\CM}_{1,1}$. By the construction, the map $\widetilde{Def}\eqref{defcurve}\to \overline{\CM}_{1,1}\times_{\Spec\ZZ}\Spec R_\LL$ lifts to a map $\pi_2\colon\widetilde{Def}\eqref{defcurve}\to \widetilde{\CM}_{1,1}$. Finally, consider the composition $Def\eqref{defcube}\to \widetilde{\CM}_{1,1}$. The following lemma is a straightforward computation:
\begin{lemma}\label{lem:diffjinv}
Let $L$ be an $O$-constrained Mumford stacky $\Gamma^{\rm tr}$-reduction. Then, under the assumptions of Corollary~\ref{cor:computationofconstrainedexts}, the following hold:
$T_LDef\eqref{defcube}=\CE_\kk^1(\Gamma^{\rm st}, A),$ $T_{\pi_1(L)}\widetilde{Def}\eqref{defcurve}=\bigoplus_{i=1}^k\kk\simeq\bigoplus_{i=1}^k (N_{e_i})_\kk,$ $d\pi_1\colon\CE_\kk^1(\Gamma^{\rm st}, A)\to \bigoplus_{i=1}^k (N_{e_i})_\kk$ is the natural projection, $T_{\pi_2(\pi_1(L))}\widetilde{\CM}_{1,1}=\kk,$ and if the orientation on $\Gamma^{\rm st}$ is such that the cycle $e_1,\dotsc, e_k$ is oriented then $d\pi_2\colon \bigoplus_{i=1}^k \kk\to \kk$ is given by $d\pi_2(x_i)=\sum_{i=1}^k x_i$.
\end{lemma}
\begin{corollary}\label{cor:smoothproj} Under the assumptions of Lemma~\ref{lem:diffjinv}, the map $\pi_2\circ\pi_1$ is smooth if and only if $\CE_\kk^2(\Gamma^{\rm st}, A, j)\to \CE_\kk^2(\Gamma^{\rm st}, A)$ is an isomorphism. In particular, if $(\Gamma^{\rm st}, A)$ is elliptically $\kk$-regular then $\pi_2\circ\pi_1$ is smooth.
\end{corollary}
\begin{proof} By Claim~\ref{claim:EvsEJ}, $\CE_\kk^2(\Gamma^{\rm st}, A, j)\to \CE_\kk^2(\Gamma^{\rm st}, A)$ is an isomorphism if and only if $\delta_\kk\colon \CE_\kk^1(\Gamma^{\rm st}, A)\to \kk$ is surjective, and under the identifications of Lemma~\ref{lem:diffjinv}, $\delta_\kk$ is a non-zero multiple of $d\pi_2\circ d\pi_1\colon T_LDef\eqref{defcube}\to T_{\pi_2(\pi_1(L))}\widetilde{\CM}_{1,1}$, which implies the corollary.
\end{proof}

\section{Correspondence Theorems}\label{sec:CorrThm}
\begin{definition}
A smooth complete curve with marked points $(C, D)$ is called a {\it simple Mumford curve} if it is stable, the graph $\Gamma_{C,D,f}^{\rm st}$ is trivalent, i.e., all finite vertices have valency three, and $g(\Gamma_{C,D,f}^{\rm st})=g(C)$.
\end{definition}

\begin{theorem}[Correspondence Theorem]\label{thm:corrthm}
Let $\Gamma$ be a stable $N_\QQ$-parame\-terized $\QQ$-tropical curve, $O$ be a toric constraint, and $A$ be the corresponding affine constraint. Assume that
\begin{enumerate}
\item $\Gamma$ is trivalent, i.e., all finite vertices have valency three,
\item $\Gamma$ satisfies $A$,
\item $(\Gamma, A)$ is $\kk$-regular,
\item $\codim A=\rank(\Gamma)$,
\item All bounded edges $e\in E^b(\Gamma)$ have non-trivial slopes,
\item The multiplicities of all edges are not divisible by the characteristic.
\end{enumerate}
Then, there exist precisely $|\CE^1_{\kk^*}(\Gamma, A)|=|\EE^1_{\kk^*}(\Gamma, A)|\cdot\prod_{e\in E^b(\Gamma)}l(e)$ isomorphism clas\-ses of triples $(C, D, f)$ satisfying the toric constraint $O$ such that
$(C,D)$ is a simple Mumford curve of genus $g$ and $\Gamma^{\rm st}_{C,D,f}=\Gamma$. Moreover, there is a one-to-one correspondence between such curves and the isomorphism classes of $O$-constrained stacky $\Gamma$-reductions.
\end{theorem}
\begin{remark} Under the assumptions of the theorem, we have $|\EE^1_{\kk^*}(\Gamma, A)|=|\EE^2(\Gamma, A)|$ and $|\CE^1_{\kk^*}(\Gamma, A)|=|\CE^2(\Gamma, A)|$ by Claim~\ref{cl:e1k*}.
\end{remark}
\begin{proof}
First, note that $|\EE^1_{\kk^*}(\Gamma, A)|$ is finite and $\EE^2_{\kk^*}(\Gamma, A)=0$ since $\codim A=\rank(\Gamma)$ and $(\Gamma, A)$ is $\kk^*$-regular by Claim~\ref{cl:e1k*}. Then, by Proposition~\ref{prop:ParamOfTrLimitsConstr}, the number of isomorphism classes of $O$-constrained $\Gamma$-reductions is equal to $|\EE^1_{\kk^*}(\Gamma, A)|$ since $\Gamma$ is trivalent. Hence, by Proposition~\ref{prop:ParamOfConstrStakyTrLimits}, the number of isomorphism classes of $O$-constrained stacky $\Gamma$-reductions is equal to the product $|\EE^1_{\kk^*}(\Gamma, A)|\cdot\prod_{e\in E^b(\Gamma)}l(e)$. Thus, it is sufficient to prove the ``moreover part'' of the theorem.

Pick a finite extension $\FF\subseteq\LL\subseteq\overline{\FF}$ sufficiently ramified for any triple $(C, D, f)$ satisfying $O$, and fix a uniformizer $t_\LL\in R_\LL$.
Plainly, any triple $(C, D, f)$ satisfying $O$ defines an $O$-constrained stacky $\Gamma$-reduction.  Vice versa, any $O$-constrained stacky $\Gamma$-reduction $(\CC_\kk, D_\kk, \varphi_\kk, \CX_\kk^{\rm tr})$ defines a unique isomorphism class of triples $(C, D, f)$ satisfying the toric constraint $O$: by Corollary~\ref{cor:computationofconstrainedexts}, $Def^1\eqref{defcube}=\EE^1_\kk(\Gamma^{\rm st}, A)=0$ and $Ob\eqref{defcube}=\EE^2_\kk(\Gamma^{\rm st}, A)=0$, since $\codim A=\rank(\Gamma)$ and $(\Gamma, A)$ is $\kk$-regular. Hence, there exists a unique solution to the deformation problem \eqref{defcube} for any $O$-constrained stacky $\Gamma$-reduction, and we are done.
\end{proof}

\begin{theorem}[Yet Another Correspondence Theorem]\label{thm:yacs}
Let $\Gamma$ be a stable $N_\QQ$-para\-meterized $\QQ$-tropical curve of genus one, $O$ be a toric constraint, and $A$ be the corresponding affine constraint. Assume that
\begin{enumerate}
\item $\Gamma$ is trivalent,
\item $\Gamma$ satisfies $A$,
\item $(\Gamma, A)$ is elliptically $\kk$-regular,
\item $\rank(\Gamma)=\codim A+1$,
\item All bounded edges $e\in E^b(\Gamma)$ have non-trivial slopes,
\item The multiplicities of all edges are not divisible by the characteristic.
\end{enumerate}
Then for any $J\in \overline{\FF}$ with $\val(J)=-j(\Gamma)$ there exist precisely $|\CE^1_{\kk^*}(\Gamma, A, j)|$ isomorphism classes of triples $(C, D, f)$ satisfying the toric constraint $O$ such that $(C,D)$ is a simple Mumford curve of genus one, $j(C)=J$, and $\Gamma^{\rm st}_{C,D,f}=\Gamma$. Moreover, there is a one-to-one correspondence between such curves and the isomorphism classes of stacky $O$-constrained Mumford $\Gamma$-reductions  with $j_\kk(\CC_\kk,D_\kk,\varphi_\kk,\CX_\kk^{\rm tr})=j_\kk^{-1}(C)$.
\end{theorem}
\begin{proof}
First, note that $\codim A+1=\rank(\Gamma)$ and $(\Gamma, A)$ is elliptically $\kk^*$-regular by Claim~\ref{cl:e1jk*}. Thus, $|\CE^1_{\kk^*}(\Gamma, A,j)|$ is finite and $\CE^2_{\kk^*}(\Gamma, A)=0$. Then, by Proposition~\ref{prop:leadingcoefofj}, the number of isomorphism classes of $O$-constrained stacky $\Gamma$-reductions with given $j_\kk(\CC_\kk,D_\kk,\varphi_\kk,\CX_\kk^{\rm tr})$ is equal to $|\CE^1_{\kk^*}(\Gamma, A, j)|$. Thus, it is sufficient to prove the ``moreover part'' of the theorem.

Pick a finite extension $\FF\subseteq\LL\subseteq\overline{\FF}$ sufficiently ramified for any triple $(C, D, f)$ satisfying $O$ and having $j(C)=J$. Fix a uniformizer $t_\LL\in R_\LL$. Plainly, any such triple $(C, D, f)$ defines a stacky $O$-constrained Mumford $\Gamma$-reduction $(\CC_\kk,D_\kk,\varphi_\kk,\CX_\kk^{\rm tr})$. Furthermore, $j_\kk(\CC_\kk,D_\kk,\varphi_\kk,\CX_\kk^{\rm tr})=j^{-1}_\kk(C)$ by Lemma~\ref{lem:leadtermofj}. Vice versa, any stacky $O$-cons\-trained $\Gamma$-reduction $(\CC_\kk, D_\kk, \varphi_\kk, \CX_\kk^{\rm tr})$ defines a unique isomorphism class of triples $(C, D, f)$ with $j(C)=J$ that satisfy the toric constraint $O$. Indeed, by Corollary~\ref{cor:smoothproj}, the projection $Def\eqref{defcube}\to\widetilde{\CM}_{1,1}$ is smooth, and both spaces have dimension one. Hence there exists a unique solution to the deformation problem \eqref{defcube} with given $j$-invariant for any $O$-constrained stacky $\Gamma$-reduction.
\end{proof}

\section{Appendix}
In this appendix we summarize well known facts about nodal and (semi-)stable models of algebraic curves that we use in our paper.
Let $(C,D)$ be as in the introduction. By the nodal reduction theorem, one can find a finite extension $\LL$ of $\FF$ and a nodal model $(C_{R_\LL},D_{R_\LL})$
$$\xymatrix{
  D_{R_\LL}\ar@{^{(}->}[rr]\ar[dr] & &C_{R_\LL}\ar[dl] \\
  & \Spec R_\LL &
}$$
i.e., a triple consisting of a proper curve $C_{R_\LL}\to\Spec R_\LL$, a finite ordered set $D_{R_\LL}$ of $R_\LL$-points in $C_{R_\LL}$, and an isomorphism $(C_{R_\LL}, D_{R_\LL})\times_{\Spec R_\LL}\Spec\overline{\FF}\simeq (C,D)$ such that the {\em reduction} $(C_{R_\LL}, D_{R_\LL})\times_{\Spec R_\LL}\Spec\kk$ is a reduced nodal curve with marked points, and the total space $C_{R_\LL}$ is normal. In particular, $(C,D)$ is defined over $\LL$. A model is called {\em regular} if the total space $C_{R_\LL}$ is regular.

It is well known that the singularities of a nodal model are concentrated at the nodes of the reduction. Moreover, any singular point is of type $A_r$, i.e., \'etale locally, it is given by an equation $xy=t_\LL^{r+1}$. It is also known that any nodal model is dominated by a (minimal) regular nodal model and the preimage of a singular point of type $A_r$ is a chain of $r$ lines of self-intersection $-2$.

\begin{algorithm}\label{alg:models}
Let $(C'_{R_\LL}, D_{R_\LL})\to (C_{R_\LL}, D_{R_\LL})$ be nodal models. Then the model $C_{R_\LL}$ can be obtained from $C'_{R_\LL}$ using the following three steps:
\begin{enumerate}
\item blow down the maximal forest of trees of relative unstable components, such that each tree intersects the remaining components at one point - the root of the tree, and no marked point belongs to the forest;
\item  blow down the relative unstable chains of projective lines containing a unique marked point, which belongs to the first line of the chain, and intersecting the remaining curve at a unique point, which belongs to the last line of the chain; and
\item blow down the remaining relative unstable chains of projective lines.
\end{enumerate}
\end{algorithm}
Note that if one starts with a regular nodal model $C'_{R_\LL}$, and proceeds as above, then after the second step one obtains the minimal regular nodal model dominating $C_{R_\LL}$. Note also that if $(C,D)$ is stable then by applying the algorithm to {\it all} unstable components of $C'_{R_\LL}$ one obtains the stable model $(C^{\rm st}_{R_\LL}, D_{R_\LL})$ over $\Spec R_\LL$.

\end{document}